\documentclass[leqno,11pt]{article}

\usepackage[utf8]{inputenc}
\usepackage[T1]{fontenc}
\usepackage{microtype}

\usepackage[dvipsnames]{xcolor}
\usepackage{xcolor-solarized}

\usepackage{calc}
\usepackage[a4paper]{geometry}

\ifdefined\screenLayout
  \pagecolor{solarized-base3}
  \color{solarized-base03}
  \usepackage{fancyhdr}
\fi

\usepackage{amsmath}
\usepackage{amsthm}

\usepackage{tikz-cd}
\usetikzlibrary{arrows} 
\tikzset{
  commutative diagrams/.cd, 
  arrow style=tikz, 
  diagrams={>=stealth}
}
\usetikzlibrary{matrix,decorations.pathreplacing,calc}
\usetikzlibrary{graphs,graphs.standard}
\usepackage{pgfplots}
\usepgfplotslibrary{external}

\usepackage{afterpage}
\usepackage{pdflscape}

\usepackage{colortbl}
\usepackage{adforn}
\usepackage{nameref}

\usepackage{textcomp}
\usepackage[sb]{libertine}
\usepackage[varqu,varl]{zi4}%
\usepackage[libertine,bigdelims,vvarbb]{newtxmath}
\usepackage[supscaled=1.2,raised=-.13em]{superiors}
\useosf
\usepackage[scr=boondox,cal=euler]{mathalfa}
\usepackage{marvosym}
\usepackage{slashed}
\usepackage{esint} 
\usepackage[ngerman,english]{babel}
\usepackage{imakeidx}
\makeindex[intoc]
\usepackage{csquotes}
\usepackage[
  backend=biber,
  hyperref=true,
  backref=true,
  isbn=false,
  doi=true,
  natbib=true,
  eprint=true,
  useprefix=true,
  maxcitenames=99,
  maxbibnames=99,  
  maxalphanames=99, 
  minalphanames=99,
  safeinputenc,
  style=alphabetic,
  citestyle=alphabetic,
  block=space,
  datamodel=preamble/ext-eprint,
  sorting=nyt
]{biblatex}
\usepackage[
  bookmarksnumbered = true,
  hypertexnames = false,
  colorlinks    = true,
  citecolor     = solarized-blue,
  linkcolor     = solarized-blue,
  urlcolor      = solarized-blue,
  breaklinks
]{hyperref}

\DeclareFieldFormat{url}{%
  \href{#1}{\ComputerMouse}
}
\DeclareFieldFormat{doi}{%
  \mkbibacro{DOI}\addcolon\space\href{https://doi.org/#1}{#1}
}
\makeatletter
\DeclareFieldFormat{arxiv}{%
  arXiv\addcolon\space\href{http://arxiv.org/\abx@arxivpath/#1}{#1}
}
\makeatother
\DeclareFieldFormat{mr}{%
  MR\addcolon\space\href{http://www.ams.org/mathscinet-getitem?mr=MR#1}{#1}
}
\DeclareFieldFormat{zbl}{%
  Zbl\addcolon\space\href{http://zbmath.org/?q=an:#1}{#1}
}
\renewbibmacro*{eprint}{%
  \printfield{arxiv}%
  \newunit\newblock
  \printfield{mr}%
  \newunit\newblock
  \printfield{zbl}%
  \newunit\newblock
  \iffieldundef{eprinttype}
  {\printfield{eprint}}
  {\printfield[eprint:\strfield{eprinttype}]{eprint}}
}

\AtEveryBibitem{%
  \clearlist{address}%
}
\DeclareFieldFormat[article,inproceedings,inbook,incollection,thesis]{title}{\textit{#1}}
\renewbibmacro{in:}{}
\usepackage[inline,shortlabels]{enumitem}

\usepackage{subcaption}

\usepackage[yyyymmdd]{datetime}

\usepackage{etoolbox}
\ifundef{\abstract}{}{\patchcmd{\abstract}%
    {\quotation}{\quotation\noindent\ignorespaces}{}{}}

\usepackage[super]{nth}

\usepackage{thmtools}
\usepackage[framemethod=TikZ]{mdframed}

\numberwithin{equation}{section}

\renewcommand{\qedsymbol}{$\blacksquare$}

\newcommand{\CorollaryQED}{\qedsymbol}
\newcommand{\ConjectureQED}{$\square$}
\newcommand{\SituationQED}{$\times$}
\newcommand{\DefinitionQED}{$\bullet$}
\newcommand{\NotationQED}{$\circ$}
\newcommand{\ExampleQED}{$\spadesuit$}
\newcommand{\RemarkQED}{$\clubsuit$}
\newcommand{\ExerciseQED}{?!}

\ifdefined\screenLayout
  \declaretheoremstyle[
  bodyfont=\itshape,
  mdframed={    
    backgroundcolor=solarized-base3!90!solarized-blue,
    linewidth=0,
    innerleftmargin=.5em,
    innerrightmargin=.5em,
    innertopmargin=.5em,
    innerbottommargin=.5em,
    leftmargin=-.5em,
    rightmargin=-.5em,
  }
]{theorem}

\declaretheoremstyle[
mdframed={
  backgroundcolor=solarized-base3!90!solarized-green,
  linewidth=0,
  innerleftmargin=.5em,
  innerrightmargin=.5em,
  innertopmargin=.5em,
  innerbottommargin=.5em,
  leftmargin=-.5em,
  rightmargin=-.5em,
  }
]{definition}

\declaretheoremstyle[
  mdframed={
    backgroundcolor=solarized-base3!90!solarized-yellow,
    linewidth=0,
    innerleftmargin=.5em,
    innerrightmargin=.5em,
    innertopmargin=.5em,
    innerbottommargin=.5em,
    leftmargin=-.5em,
    rightmargin=-.5em,
  }
]{example}

\declaretheoremstyle[
  mdframed={
    backgroundcolor=solarized-base3!90!solarized-orange,
    linewidth=0,
    innerleftmargin=.5em,
    innerrightmargin=.5em,
    innertopmargin=.5em,
    innerbottommargin=.5em,
    leftmargin=-.5em,
    rightmargin=-.5em,
  }
]{remark}
\else
  \declaretheoremstyle[
  bodyfont=\itshape
  ]{theorem}
  \declaretheoremstyle[]{definition}
  \declaretheoremstyle[]{example}
  \declaretheoremstyle[]{remark}
\fi

\declaretheorem[numberlike=equation,style=theorem]{theorem}
\declaretheorem[numbered=no,name=Theorem,style=theorem]{theorem*}
\declaretheorem[numberlike=equation,name=Lemma,style=theorem]{lemma}
\declaretheorem[numberlike=equation,name=Proposition,style=theorem]{prop}
\declaretheorem[numberlike=equation,name=Corollary,qed=\CorollaryQED,style=theorem]{cor}

\declaretheorem[numberlike=equation,name=Hypothesis]{hypothesis}

\declaretheorem[numberlike=equation,name=Definition,style=definition,qed=\DefinitionQED]{definition}
\declaretheorem[numbered=no,name=Definition,style=definition,qed=\DefinitionQED]{definition*}

\declaretheorem[numberlike=equation,style=definition,qed=\ExampleQED]{example}

\declaretheorem[numberlike=equation,style=remark,qed=\RemarkQED]{remark}
\declaretheorem[numbered=no,style=remark,name=Remark,qed=\RemarkQED]{remark*}

\def\makeautorefname#1#2{\AtBeginDocument{\expandafter\def\csname#1autorefname\endcsname{#2}}}
\makeautorefname{table}{Table}        
\makeautorefname{chapter}{Chapter}
\makeautorefname{section}{Section}
\makeautorefname{subsection}{Section}
\makeautorefname{subsubsection}{Section}
\makeautorefname{footnote}{Footnote}
\AtBeginDocument{\def\itemautorefname~#1\null{(#1)\null}}
\AtBeginDocument{\def\equationautorefname~#1\null{(#1)\null}}

\newtheorem{step}{Step}

\numberwithin{substep}{step}
\makeautorefname{step}{Step}
\makeautorefname{substep}{Step}

\makeautorefname{case}{Case}
\makeautorefname{substep}{Step}

\setlist[description]{leftmargin=!,labelindent=1em}
\setlist[enumerate]{label={\rm (\arabic*)},ref=\arabic*}
\setlist[enumerate,2]{label={\rm (\alph*)},ref=\theenumi.\alph*}
\setlist[enumerate,3]{label={\rm (\roman*)},ref=\theenumii.\roman*}

\let\C\undefined
\let\U\undefined

\usepackage{bm}
\usepackage{mathtools} 
\usepackage{stmaryrd} 

\DeclareFontFamily{U}{mathx}{\hyphenchar\font45}
\DeclareFontShape{U}{mathx}{m}{n}{
      <5> <6> <7> <8> <9> <10>
      <10.95> <12> <14.4> <17.28> <20.74> <24.88>
      mathx10
      }{}
\DeclareSymbolFont{mathx}{U}{mathx}{m}{n}
\DeclareFontSubstitution{U}{mathx}{m}{n}
\DeclareMathAccent{\widecheck}{0}{mathx}{"71}
\DeclareMathAccent{\wideparen}{0}{mathx}{"75}

\DeclareMathOperator{\Ad}{Ad}

\DeclareMathOperator{\Aut}{Aut}

\DeclareMathOperator{\Bl}{Bl}

\DeclareMathOperator{\End}{End}

\DeclareMathOperator{\GL}{GL}
\DeclareMathOperator{\Gr}{Gr}

\DeclareMathOperator{\HF}{\HF}

\DeclareMathOperator{\Hom}{Hom}

\DeclareMathOperator{\Map}{Map}

\DeclareMathOperator{\Res}{Res}

\DeclareMathOperator{\Sym}{Sym}

\DeclareMathOperator{\ad}{ad}

\DeclareMathOperator{\coker}{coker}

\DeclareMathOperator{\im}{im}
\DeclareMathOperator{\ind}{index}

\DeclareMathOperator{\rk}{rk}
\DeclareMathOperator{\sign}{sign}

\DeclareMathOperator{\supp}{supp}

\DeclarePairedDelimiter\paren{\lparen}{\rparen}
\DeclarePairedDelimiter\sqparen{[}{]}
\DeclarePairedDelimiter{\Abs}{\|}{\|}

\DeclarePairedDelimiter{\Inner}{\langle}{\rangle}

\DeclarePairedDelimiter{\abs}{\lvert}{\rvert}
\DeclarePairedDelimiter{\bracket}{\langle}{\rangle}

\DeclarePairedDelimiter{\set}{\lbrace}{\rbrace}
\def\({\left(}
\def\){\right)}
\def\<{\left\langle}
\def\>{\right\rangle}

\newcommand{\CP}{{\C P}}

\newcommand{\C}{{\mathbf{C}}}

\newcommand{\N}{{\mathbf{N}}}

\newcommand{\PSL}{\P\SL}

\newcommand{\R}{\mathbf{R}}

\newcommand{\SL}{\mathrm{SL}}
\newcommand{\SO}{\mathrm{SO}}
\newcommand{\SU}{\mathrm{SU}}

\newcommand{\Span}[1]{\bracket{#1}}
\newcommand{\Spin}{\mathrm{Spin}}
\newcommand{\Sp}{\mathrm{Sp}}

\newcommand{\U}{\mathrm{U}}

\newcommand{\Z}{\mathbf{Z}}

\newcommand{\co}{\mskip0.5mu\colon\thinspace}

\newcommand{\cyl}{{\rm{cyl}}}

\newcommand{\defined}[2][\key]{\def\key{#2}\textbf{#2}\index{#1}}
\newcommand{\delbar}{\bar{\del}}

\newcommand{\del}{\partial}
\newcommand{\ev}{\mathrm{ev}}

\newcommand{\hkred}{{/\!\! /\!\! /}}

\newcommand{\id}{\mathrm{id}}
\newcommand{\incl}{\hookrightarrow}

\newcommand{\into}{\hookrightarrow}
\newcommand{\iso}{\cong}

\newcommand{\loc}{\mathrm{loc}}

\newcommand{\one}{\mathbf{1}}
\newcommand{\onto}{\twoheadrightarrow}

\newcommand{\pr}{\mathrm{pr}}
\newcommand{\qandq}{\quad\text{and}\quad}

\newcommand{\qwithq}{\quad\text{with}\quad}

\newcommand{\su}{\mathfrak{su}}

\newcommand{\vol}{\mathrm{vol}}

\renewcommand{\H}{\mathbf{H}}
\renewcommand{\Im}{\operatorname{Im}}
\renewcommand{\O}{\mathrm{O}}
\renewcommand{\P}{\mathbf{P}}
\renewcommand{\Re}{\operatorname{Re}}
\renewcommand{\det}{\operatorname{det}}

\renewcommand{\emptyset}{\varnothing}
\renewcommand{\epsilon}{\varepsilon}
\renewcommand{\setminus}{{\backslash}}

\renewcommand{\leq}{\leqslant}
\renewcommand{\geq}{\geqslant}

\newcommand{\Wedge}{\Lambda}

\makeatletter
\renewcommand*\env@matrix[1][*\c@MaxMatrixCols c]{%
  \hskip -\arraycolsep
  \let\@ifnextchar\new@ifnextchar
  \array{#1}}

\renewcommand\xleftrightarrow[2][]{%
  \ext@arrow 9999{\longleftrightarrowfill@}{#1}{#2}}
\newcommand\longleftrightarrowfill@{%
  \arrowfill@\leftarrow\relbar\rightarrow}
\makeatother

\newcommand{\Crit}[1]{\mathrm{Crit}(#1)}

\newcommand{\Reg}{\mathrm{Reg}}

\newcommand{\rc}{{\rm c}}
\newcommand{\rd}{{\rm d}}

\newcommand{\rC}{{\rm C}}

\newcommand{\rH}{{\rm H}}
\newcommand{\rI}{{\rm I}}
\newcommand{\rII}{{\rm II}}

\newcommand{\bp}{{\mathbf{p}}}

\newcommand{\bD}{{\mathbf{D}}}
\newcommand{\bE}{{\mathbf{E}}}
\newcommand{\bF}{{\mathbf{F}}}

\newcommand{\bL}{{\mathbf{L}}}
\newcommand{\bM}{{\mathbf{M}}}

\newcommand{\bP}{{\mathbf{P}}}

\newcommand{\bS}{{\mathbf{S}}}

\newcommand{\bV}{{\mathbf{V}}}
\newcommand{\bW}{{\mathbf{W}}}

\newcommand{\cM}{\mathcal{M}}

\newcommand{\sA}{\mathscr{A}}
\newcommand{\sB}{\mathscr{B}}
\newcommand{\sC}{\mathscr{C}}

\newcommand{\sE}{\mathscr{E}}
\newcommand{\sF}{\mathscr{F}}
\newcommand{\sG}{\mathscr{G}}
\newcommand{\sH}{\mathscr{H}}

\newcommand{\sJ}{\mathscr{J}}

\newcommand{\sL}{\mathscr{L}}
\newcommand{\sM}{\mathscr{M}}
\newcommand{\sN}{\mathscr{N}}
\newcommand{\sO}{\mathscr{O}}
\newcommand{\sP}{\mathscr{P}}

\newcommand{\sR}{\mathscr{R}}
\newcommand{\sS}{\mathscr{S}}
\newcommand{\sT}{\mathscr{T}}
\newcommand{\sU}{\mathscr{U}}
\newcommand{\sV}{\mathscr{V}}
\newcommand{\sW}{\mathscr{W}}
\newcommand{\sX}{\mathscr{X}}

\newcommand{\sZ}{\mathscr{Z}}

\newcommand{\fa}{{\mathfrak a}}

\newcommand{\fd}{{\mathfrak d}}

\newcommand{\fg}{{\mathfrak g}}

\newcommand{\fl}{{\mathfrak l}}
\newcommand{\fm}{{\mathfrak m}}
\newcommand{\fn}{{\mathfrak n}}
\newcommand{\fo}{{\mathfrak o}}

\newcommand{\fr}{{\mathfrak r}}

\newcommand{\fu}{{\mathfrak u}}

\newcommand{\fw}{{\mathfrak w}}

\newcommand{\fL}{{\mathfrak L}}

\newcommand{\fR}{{\mathfrak R}}

\newcommand{\fV}{{\mathfrak V}}

\newcommand{\ubR}{{\underline \R}}

\author{
  Aleksander Doan
  \and
  Thomas Walpuski
}
\title{  
  Chambered invariants of real Cauchy--Riemann operators
}
\date{2025-08-19}

\usetikzlibrary{babel}
\newcommand{\Flag}{\mathrm{Fl}}

\newcommand{\Conf}{\mathrm{Conf}}
\newcommand{\Br}{\mathrm{Br}}

\newcommand{\dom}{\mathrm{dom}}
\newcommand{\SpaceOfRealCauchyRiemannOperators}{\sC\sR}
\newcommand{\ubC}{\underline\C}
\newcommand{\SpaceOfHomogeneousAlmostComplexStructures}{\sJ_h}
\newcommand{\ConfigurationSpace}{\mathrm{Conf}}
\newcommand{\SpaceOfRamifiedEuclideanLineBundles}{\mathrm{RamLinBun}}
\newcommand{\ver}{\mathrm{vert}}

\newcommand{\SpaceOfBlowUpAlmostComplexStructures}{\sJ_{\varheartsuit}(\sO_{\P V}(-2))}
\newcommand{\BlowUpInvariant}{n_{\Bl}}
\newcommand{\ADHMCurvesInvariant}{n_{1,2}}
\newcommand{\ADHMGaugeInvariant}{n_{2,1}}
\newcommand{\BlowUpModuliSpace}{\sM_{\Bl}}
\newcommand{\ADHMCurvesModuliSpace}{\sM_{1,2}}
\newcommand{\ADHMGaugeModuliSpace}{\sM_{2,1}}
\newcommand{\ADHMrkModuliSpace}{\sM_{r,k}}
\newcommand{\HOMT}{\sH}
\newcommand{\SerreOperator}{\gamma}

\newcommand{\SpaceOfWallCrossingFormulae}{\mathrm{WCF}}
\newcommand{\Val}{G}
\newcommand{\TheWall}{\sW}

\begin{document}

\maketitle

\begin{abstract}
  Motivated by counting pseudo-holomorphic curves in symplectic Calabi--Yau $3$--folds, this article studies a chamber structure in the space of real Cauchy--Riemann operators on a Riemann surface, and constructs three chambered invariants associated with such operators:
  $\BlowUpInvariant$, $\ADHMCurvesInvariant$, $\ADHMGaugeInvariant$.
  The first of them is defined by counting pseudo-holomorphic sections of bundles whose fibre is the blow-up of $\C^2/\set{\pm 1}$.  
  The other two are defined by counting solutions to the ADHM vortex equations. 
  We conjecture that $\ADHMCurvesInvariant$ and $\ADHMGaugeInvariant$ are related to the Pandharipande--Thomas and rank $2$ Donaldson--Thomas invariants in algebraic geometry.
\end{abstract}


\section{Introduction}
\label{Sec_Introduction}

\subsection{Pseudo-holomorphic curves in dimension six}

This article is motivated by counting pseudo-holomorphic curves in symplectic Calabi--Yau $3$--folds;
that is: symplectic manifolds $(X,\omega)$ with $\dim_\R X = 6$ and $\rc_1(X,\omega)=0$. 
The standard approach in symplectic geometry is Gromov--Witten theory \cites{LiTian1998,FukayaOno1999,Ruan1999,Pardon2016}.
In contrast, algebraic geometry abounds with other curve counting invariants constructed using sheaf theory, such as Donaldson--Thomas \cite{Thomas2000,Maulik2006a,Maulik2006} and Pandharipande--Thomas invariants \cite{Pandharipande2009}.
Physics also suggests the existence of the Gopakumar--Vafa BPS counts \cite{Gopakumar1998,Gopakumar1998a},
whose direct mathematical construction remains elusive despite numerous efforts \cite{Hosono2001,Katz2008,Kiem2012,Maulik2018}.

The sheaf-theoretic methods used in algebraic geometry have no obvious counterpart in symplectic topology. 
However, the recent proof of the MNOP conjecture \cite{Pardon2023} shows that the DT/PT invariants of projective Calabi--Yau $3$--folds are invariant under symplectomorphisms, suggesting the possibility of a purely symplectic construction. 
Foundational ideas in this direction were introduced in \cites{Donaldson1998,Donaldson2009} and developed in \cite{Haydys2011,Walpuski2013} and \cite{Doan2017d}.

The goal of the present article is to advance the proposal \cite{Doan2017d}.
Its basic idea is simple.
The DT/PT invariants in algebraic geometry are indexed by a homology class $A \in \rH_2(X;\Z)$ but not the genus of curves like the Gromov--Witten invariants.
Thus, a naive approach to defining them in symplectic topology is to count embedded $J$--holomorphic curves \emph{of arbitrary genus} in class $A$ for an almost complex structure $J$ tamed by $\omega$.
For a generic $J$, such a count is finite \cite{Doan2018a} and, 
if $A$ is a primitive class, independent of $J$ \cites{Zinger2011,Doan2019}.
However, if $A$ is not primitive, this count depends on $J$. 
Indeed, suppose that $A$ is divisible by $k \in \N$ in $\rH_2(X,\Z)$. 
As $J$ varies in the space $\sJ(X,\omega)$ of tamed almost complex structures, a sequence of $J$--holomorphic curves $\tilde C$ with $[\tilde C] = A$ may collapse with multiplicity $k$ to a curve $C$ with $A = k[C]$, which changes the count of curves in class $A$.

This collapse phenomenon is related to the notion of \emph{$k$--rigidity}.  
Recall that deformations of $C$ are governed by a differential operator $\fd_{C,J}^N \co \Gamma(C,NC) \to \Omega^{0,1}(C,NC)$ with $NC \to C$ denoting the normal bundle.
This is a \emph{real Cauchy--Riemann operator}: the sum of a $\C$--linear Dolbeault operator and a $\C$--anti-linear operator of order zero.
Let $\pi \co \tilde C \to C$ be a holomorphic branched cover of degree $k$. 
A collapse of the kind described above is infinitesimally modeled on a non-zero element of the kernel of $\pi^*\fd_{C,J}^N$, the pull-back operator.
If the curve $C$ is \emph{$k$--rigid} in the sense that $\pi^*\fd_{C,J}^N$ has trivial kernel for all $\pi$, then such a collapse cannot happen.

Denote by $\sW_k \subset \sJ(X,\omega)$ the \emph{wall of failure of $k$--rigidity}: the set of those $J$ for which at least one $J$--holomorphic curve fails to be $k$--rigid, and let $\sW_\infty = \bigcup_{k=1}^\infty \sW_k$. 
The summary of the above discussion is that the count of $J$--holomorphic curves in class $A$ depends only on the connected component of $J$ in  $\sJ(X,\omega)\setminus\sW_\infty$.
Wendl proved that this set has codimension one \cite{Wendl2016}, so that a path of $\sJ(X,\omega)$ will, in general, intersect $\sW_\infty$, causing a jump in the number of curves. 
In \cite{Doan2017d}, we proposed to replace the naive count by a weighted count:
\begin{equation}
	\label{Eq_SymplecticInvariant}
  N(X,\omega)
  =
  \sum_{A \in \rH_2(X,\Z)}
  \sum_{g=0}^\infty
  \sum_{[C] \in \sM_{A,g}^*(X,J)}
  \sum_{k = 1}^\infty
  n_k(\fd_{C,J}^N) \cdot
  q^{kA},
\end{equation} 
where $J$ is generic, $\sM_{A,g}^*(X,J)$ is the moduli space of simple $J$--holomorphic maps of genus $g$, $q$ is a formal variable, and $n_k$ are certain integer weights.
These weights themselves should depend on $\fd_{C,J}^N$ and jump whenever there exists a branched cover $\pi \colon \tilde C \to C$ such that $\pi^*\fd_{C,J}^N$ has non-trivial kernel. 
To obtain an invariant of $\sJ(X,\omega)$ from \autoref{Eq_SymplecticInvariant}, the wall-crossing formulae of the $n_k$ and the wall-crossing arising from multiple-cover phenomena need to cancel precisely.

Ansatz \autoref{Eq_SymplecticInvariant} is inspired by similar ideas in higher-dimensional gauge theory \cites[§6.2]{Donaldson2009}[Chapter 6]{Walpuski2013}{Haydys2014}{Haydys2017} and low-dimensional topology \cite{BodenHerald1998,BodenHeraldKirk2001,CappellLeeMiller2002,Lim2000,Lim2003,BaiZhang2020}.
More importantly, in the context of the present article,
Taubes' Gromov invariant for symplectic $4$--manifolds \cite{Taubes1996b}
is of the form \autoref{Eq_SymplecticInvariant} with a judicious choice of weights.
Recently, \citet{Bai2021} defined new invariants of symplectic Calabi--Yau $3$--folds and homology classes $A$ with $A/2$ primitive.
Their invariant is also of the form \autoref{Eq_SymplecticInvariant}, with the key difference that the sum is taken over curves of genus $g \leq h$ for some $h$, so that the resulting invariant is indexed by $A$ and $h$. From the analytic point of view, this means that it suffices to understand the part of $\sW_2$ corresponding to covers $\pi \colon \tilde C \to C$ with $g(\tilde C) \leq h$. 

\subsection{Walls in the space of real Cauchy--Riemann operators}

In contrast to the situations considered by \citeauthor{Taubes1996b}  and Bai--Swaminathan, in the context of DT/PT theory it seems necessary to understand how a generic path in $\sJ(X,\omega)$ intersects $\sW_\infty$ without any genus bound.
We conjecture that the number of such intersections is finite for a generic path.
If this is the case, the proof of independence of \eqref{Eq_SymplecticInvariant} on $J$ can be reduced to studying local wall-crossing formulae. 
Unfortunately, this conjecture does not follow from \citeauthor{Wendl2016}'s work \cite{Wendl2016}, even when combined with the techniques of geometric measure theory used to prove similar statements in \cite{Doan2018a,Doan2021}.
Proving the conjecture with current methods appears challenging, which reflects fundamental gaps in present understanding of degenerations of $J$--holomorphic curves to multiple covers without assuming a genus bound.

The first result of this article is a local version of the finiteness conjecture for $k=2$, i.e. for double covers.
Here,  local means that we consider only a neighborhood of the curve $C$, in the following sense.
Given a complex vector bundle $V \to C$, denote by $\SpaceOfRealCauchyRiemannOperators(V)$ the space of real Cauchy--Rieman operators on $V$.
The main example is $V = NC$ as before, so that $\fd_{C,J}^N \in  \SpaceOfRealCauchyRiemannOperators(NC)$.
In the local setting, the Calabi--Yau condition is replaced by 
\begin{equation}
	\label{Eq_LocalCalabiYau}
    2 \deg V + \rk_\C V \cdot \chi(C) = 0.
\end{equation}
The space $\SpaceOfRealCauchyRiemannOperators(V)$ serves as a local version of $\sJ(X,\omega)$. 
Let $\sW_2 \subset \SpaceOfRealCauchyRiemannOperators(V)$ be the analogue of the wall of failure of $2$--rigidity, that is: the set of operators $\fd$ for which there exists a double branched cover $\pi \colon \tilde C \to C$ such that $\pi^*\fd$ has non-trivial kernel. 

\begin{theorem}[see~\autoref{Thm_WallOfFailureOf2RigidityIsProper}]
  \label{Thm_WallOfFailureOf2RigidityIsProper_0}
  Let $C$ be a closed connected Riemann surface.
  Let $V$ be a complex vector bundle over $C$ satisfying \autoref{Eq_LocalCalabiYau}. 
  The  set
  $\sW_2 \subset \SpaceOfRealCauchyRiemannOperators(V)$
  is a proper wall in the sense of \autoref{Def_ProperWall}.
  In particular: a generic path in  $\SpaceOfRealCauchyRiemannOperators(V)$ intersects $\sW_2$ in finitely many points.
\end{theorem}

The key difficulty is that $\sW_2$ is a countable union of walls corresponding to covers $\pi \colon \tilde C \to C$ of different genera.
Therefore, a priori, generic paths might intersect $\sW_2$ in a countable dense set.
The proof of \autoref{Thm_WallOfFailureOf2RigidityIsProper_0} builds on ideas from \cite{Doan2018a,Doan2019,Doan2021} and combines transversality techniques and geometric measure theory \cite{Allard1972,DeLellis2017a,DeLellis2017b,DeLellis2017c}.

\subsection{Chambered invariants from blow-ups}

In what follows, $V \to C$ is a complex vector bundle satisfying $\rk_\C V = 2$ and \autoref{Eq_LocalCalabiYau}. 
The remainder of this article is concerned with the construction of three weights
\begin{equation*}
  \BlowUpInvariant, \quad
  \ADHMCurvesInvariant, \qandq  
  \ADHMGaugeInvariant
\end{equation*}
which are candidates for $n_2$ in \autoref{Eq_SymplecticInvariant} and can potentially be used to define curve counting invariants for homology classes $A$ such that $A/2$ is primitive. 
The key property we require is that they depend on a real Cauchy--Riemann operator on $V$ and are constant on the connected components of  $\SpaceOfRealCauchyRiemannOperators(V) \setminus \sW_2$; we will refer to any function with this property as a \emph{chambered invariant}.

The first of these chambered invariants counts pseudo-holomorphic sections of a non-compact fibration over $C$.
Let $\check p \co V/\set{\pm 1} \to C$ be the quotient of $V$ by fibrewise multiplication by $\pm 1$..
Define a fibration $q \co \sO_{\P V}(-2) \to C$ by blowing-up each fibre $\C^2/\set{\pm 1}$ of $\check p$.

\begin{theorem}[see~\autoref{Sec_BlowUp}]
  \label{Thm_BlowUpInvariant_0}
  There is a chambered invariant
  $\BlowUpInvariant
  \in
  \rH^0\paren{\SpaceOfRealCauchyRiemannOperators(V)\setminus\sW_2; \Z[[x]]}$
  such that for generic $\fd \in \SpaceOfRealCauchyRiemannOperators(V)\setminus \sW_2$ and generic $J \in \SpaceOfBlowUpAlmostComplexStructures$ compatible with $\fd$, 
  \begin{equation*}
    \BlowUpInvariant(\fd)
    = \sum_{d \in \N_0} \#\BlowUpModuliSpace(J,d) \cdot x^d.
  \end{equation*}
  Here
  $\SpaceOfBlowUpAlmostComplexStructures$, the space of admissible almost complex structures on $\sO_{\bP V}(-2)$, the compatibility condition means $\fd_\infty(J) = \fd$ where
$\fd_\infty \co \SpaceOfBlowUpAlmostComplexStructures \to \SpaceOfRealCauchyRiemannOperators(V)$
  is a map defined in \autoref{Sec_AlmostComplexStructuresOnBlowup}, and
  $\BlowUpModuliSpace(J,d)$ denotes the space of $J$--holomorphic sections of $q$ of degree $d$.
\end{theorem}

This is reminiscent of the construction of the Gromov--Witten invariants of fibrations in \cite[§8.6]{McDuff2012}.
However, a substantial novel complication arises because $\sO_{\P V}(-2)$ is non-compact and there are no a priori $C^0$ or energy bounds for $J$--holomorphic sections.
Indeed, sequences of $J$--holomorphic sections can escape to infinity,
introducing a non-compactness phenomenon in $ \BlowUpModuliSpace(J,d)$ different from the familiar bubbling phenomenon.

At the heart of \autoref{Thm_BlowUpInvariant_0} lies a relationship between this non-compactness phenomenon and the failure of $2$--rigidity.
An operator $\fd \in\SpaceOfRealCauchyRiemannOperators(V)$ induces an almost complex structure $J_\fd$ on $V/\{\pm 1\}$ with the property that $\fd \in \sW_2$ if and only if $V/\{\pm 1\}$ admits a $J_\fd$--holomorphis section. 
An admissible $J$ agrees with $J_\fd$ away from a compact set, and if a sequence of $J$--holomorphic sections of $\sO_{\P V}(-2)$ escapes to infinity, then after rescaling it converges to a $J_\fd$--holomorphic section of $V/\{\pm 1\}$. 
In particular: this cannot occur if $\fd \notin \sW_2$.
The key difficulty in the proof is extracting a convergent  subsequence.
Since the rescaling causes the almost complex structure to degenerate,
Gromov's compactness theorem cannot be applied.
The proof uses a combination of geometric measure theory and a twist on Radó's theorem in complex analysis \cite{Rado1924}.


\subsection{Chambered invariants from gauge theory}

\medskip

The chambered invariants $\ADHMCurvesInvariant$ and $\ADHMGaugeInvariant$ are constructed in a different way using gauge theory.
They are virtual counts of solutions of the \emph{$(r,k)$ ADHM vortex equation} for $(r,k) = (1,2)$ and $(r,k) = (2,1)$ respectively.
For every $r,k \in \N$,
the $(r,k)$ ADHM vortex equation is a system of partial differential equations for a $\U(k)$ connection and a Higgs field on a Riemann surface.
For $r=k=1$ this is the well-known vortex equation  \cite{Noguchi1987,Bradlow1990,GarciaPrada1993}.
In general, it shares many features with the vortex equation and Hitchin equation and can be thought of as a generalization of the two. 
The Higgs field takes values in a vector bundle associated with the quaternionic representation
\begin{equation*}
  \U(k) \circlearrowright S_{r,k} \coloneqq \Hom_\C\paren{\C^r,\R^4 \otimes_\C \C^k} \oplus \R^4 \otimes_\R \fu(k).
\end{equation*}
This representation appears in \citeauthor{Atiyah1978}'s construction of instantons.
For $r\geq 2$ its hyperkähler quotient  $S_{r,k} \hkred \U(k)$ is the Uhlenbeck completion of the moduli space of $\SU(r)$ ASD instantons of charge $k \in \N$ on $\R^4$ \cites{Atiyah1978}[§3.3]{Donaldson1990}.
For $r =1$ it is the $k$--fold symmetric product $\Sym^k\R^4$ \cite[Proposition 2.9]{Nakajima1999}.

The $(r,k)$ ADHM vortex equations depend on a choice of $\fd \in \SpaceOfRealCauchyRiemannOperators(V)$ as well as an auxilliary $\U(k)$ bundle determined by its degree $d \in \Z$.
For reasons that shall be explained in \autoref{Sec_ChamberedInvariantsFromQuaternionicVortexEquations} it is convenient to choose a $K_C$--valued complex symplectic form $\Omega \co \Wedge_\C^2 V \to K_C$ and restrict attention to $\SpaceOfRealCauchyRiemannOperators(V,\gamma)$, the subspace of those real Cauchy--Riemann operators which are self-adjoint with respect to the isomorphism $\gamma \co V \iso V^* \otimes_\C K_C$ induced by $\Omega$.

The moduli spaces $\ADHMrkModuliSpace(\fd,d)$ of solutions to the $(r,k)$ ADHM vortex equation are of virtual dimension zero,
but possibly plagued by reducibles and non-compactness.
Non-compactness arises from pseudo-holomorphic sections of bundles with fibres $S_{r,k} \hkred \U(k)$\cite{Haydys2011}.
If $(r,k) = (1,2)$ or $(r,k) = (2,1)$, then $S_{r,k} \hkred \U(k) = \Sym^2\R^4$ and, therefore, one might expect a relation with the failure of $2$--rigidity.
Indeed, in both cases, a variation of the compactness theorem in \cite{Walpuski2019} explained in \autoref{Sec_CompactnessQuaternionicVortexEquation} confirms this expectation, leading to the following result.

\begin{theorem}[see \autoref{Thm_ChamberedInvariantsFromADHM12}]
  \label{Thm_ADHMCurvesInvariants}
  Assume that $g(C) \geq 1$.
  There is a chambered invariant
  \begin{equation*}
    \ADHMCurvesInvariant \in \rH^0\paren{\SpaceOfRealCauchyRiemannOperators(V,\SerreOperator)\setminus\sW_2,\Z[[x,x^{-1}]]}
  \end{equation*}
  such that for every $\fd \in \SpaceOfRealCauchyRiemannOperators(V,\SerreOperator)\setminus\sW_2$  \begin{equation*}
    \ADHMCurvesInvariant(\fd) = \sum_{d\in\Z} \#\ADHMCurvesModuliSpace(\fd,d) \cdot x^d
  \end{equation*}
  with $\#\ADHMCurvesModuliSpace(\fd,d)$ denoting a virtual count of $\ADHMCurvesModuliSpace(\fd,d)$.
\end{theorem}

The count being virtual has to do with the fact $\ADHMCurvesModuliSpace(\fd,d)$ is only compact and of virtual dimension zero, but possibly obstructed and containing reducibles.
We conjecture that $\ADHMCurvesInvariant$ is the weight $n_2$ a symplectic version of Pandharipande--Thomas theory.
In fact, if $\fd$ lies in the chamber containing Dolbeault operators,
then the correspondence between $(1,2)$ ADHM vortices and local stable pairs \cites{Bradlow2003}{Consul2003}[§7]{Doan2017d}{Diaconescu2012} indicate a relationship between $\ADHMCurvesInvariant(\fd)$ and local Pandhariphande--Thomas invariants of $C$.

There is an analogous result for the $(2,1)$ ADHM vortex equation.

\begin{theorem}[see \autoref{Thm_ChamberedInvariantsFromADHM21}]
  \label{Thm_ADHMGaugeInvariants}
  Assume that $g(C) \geq 1$.
  There is a chambered invariant
  \begin{equation*}
    \ADHMGaugeInvariant \in \rH^0\paren{\SpaceOfRealCauchyRiemannOperators(V,\SerreOperator)\setminus\sW_2,\Z[[x,x^{-1}]]}
  \end{equation*}
  such that for every $\fd \in \SpaceOfRealCauchyRiemannOperators(V,\SerreOperator)\setminus\sW_2$  \begin{equation*}
    \ADHMGaugeInvariant(\fd) = \sum_{d\in\Z} \#\ADHMGaugeModuliSpace(\fd,d) \cdot x^d
  \end{equation*}
  with $\#\ADHMGaugeModuliSpace(\fd,d)$ denoting a virtual count of $\ADHMGaugeModuliSpace(\fd,d)$.
\end{theorem}

We conjecture that $\ADHMGaugeInvariant$ is related to a symplectic version of Donaldson--Thomas theory counting rank $2$ instantons on symplectic Calabi--Yau $3$--folds;
this connection to $6$--dimensional gauge theory is discussed in \cites{Donaldson2009,Walpuski2013,Haydys2017}. 
When $\fd$ is a Dolbeault operator, $\ADHMGaugeInvariant(\fd,d)$ has been computed in various examples in \cite{Doan2017}.
The precise relation of $\ADHMCurvesInvariant$ and $\ADHMGaugeInvariant$ to algebro-geometric invariants will be explored in future articles.

\medskip

In addition to potential connections to enumerative geometry,
the motivation for \autoref{Thm_ADHMCurvesInvariants} and \autoref{Thm_ADHMGaugeInvariants} comes from gauge theory and low-dimensional topology.
The $(r,k)$ ADHM vortex equations are dimensional reductions of generalized Seiberg--Witten equations on manifolds of dimension $3$ and $4$.
The study of these equations has been an active area of research starting with Taubes's work on the compactness problem for the moduli space of flat $\PSL_2(\C)$ connections \cite{Taubes2012} and subsequent work \cites{Taubes2013,Haydys2014, Taubes2016,Taubes2017,Walpuski2019}. 
The picture emerging from this work is that the count of solutions to generalized Seiberg--Witten invariant is not a topological invariant.
This is in stark contrast to the well-studied invariants defined using moduli spaces of instantons and Seiberg--Witten monopoles.
Instead, this count depends on the Riemannian metric on the manifold as well as other perturbations of the equations. 
In dimension $3$ and $4$, it is an open problem to show that in the space of such perturbations $\sP$ there is a proper wall $\sW \subset \sP$ such that the count of solutions is invariant in  each connected component of $\sP\setminus\sW$. 
\autoref{Thm_WallOfFailureOf2RigidityIsProper_0}, \autoref{Thm_ADHMCurvesInvariants}, and \autoref{Thm_ADHMGaugeInvariants} establish this conjectural picture in dimension $2$ for the Seiberg--Witten equations associated with the $(1,2)$ and $(2,1)$ ADHM representations. 

\paragraph{Acknowledgements}
The authors are indebted to Jacek Rzemieniecki for a careful proofreading of parts of this article and to Clifford Taubes for answering questions about his work on harmonic $\Z_2$ spinors. 
Aleksander Doan is supported by Trinity College, Cambridge, and was a MATH+ Visiting Scholar at Humboldt-Universität zu Berlin while writing this article.


\newcommand{\Invariant}{\rI}
\newcommand{\ChamberedInvariant}{\rC\rI}

\section{Chambered invariants and wall-crossing formulae}
\label{Sec_ChamberedInvariantsWallCrossingFormulae}

The following concept is pervasive in many areas of geometry,
especially gauge theory and symplectic topology.

\begin{definition}
  \label{Def_Invariant}
  Let $\Val$ be an abelian group.
  Let $\sP$ be a topological space.
  A \defined{$\Val$--valued invariant of $\sP$} is an element
  \begin{equation*}
    \Invariant \in \rH^0(\sP;\Val) = \Hom(\pi_0(\sP);\Val)
    \subset \Map(\sP;\Val).
    \qedhere
  \end{equation*}
\end{definition}

Of course, if $\sP$ is path-connected,
then $\rH^0(\sP;\Val) = \Val$;
therefore: a posteriori, $\Invariant \in \Val$.
Typically, a choice of $p \in \sP$ determines an elliptic partial differential equation and $\Invariant(p)$ is extracted from the moduli space of solutions.
Here is an abstract caricature.

\begin{example}
  \label{Ex_InvariantsFromProperFredholmMaps}
  Let $\sM,\sP$ be a Banach manifolds.
  Let $\pi \co \sM \to \sP$ be a \emph{proper} Fredholm map of index $d \in \N_0$ together with an orientation of its index bundle.
  Denote by $\Omega_d^\SO$ the oriented bordism group (of a point) in dimension $d$.
  There is a unique
  \begin{equation*}
    \Invariant \in \rH^0(\sP;\Omega_d^\SO)
    \qwithq
    \Invariant(p) = [\pi^{-1}(p)]
  \end{equation*}
  for every $p \in \sP \setminus \pi(\Crit{\pi})$.
  If $\sM \subset \sX$,
  then this can be promoted to $\Invariant \in \rH^0(\sP;\Omega_d^\SO(\sX))$ and combined with any $\alpha \in \Hom(\Omega_d^\SO(\sX),\Val)$ to obtain $\alpha_*\Invariant \in \rH^0(\sP;\Val)$.
\end{example}

In order to implement the above in relevant geometric situations,
issues with transversality, compactness, and symmetries need to be overcome.
Sometimes this is not quite possible,
but one can still obtain the following.

\begin{definition}
  \label{Def_ChamberedInvariant}
  Let $\Val$ be an abelian group.
  Let $\sP$ be a topological space.
  \begin{enumerate}
  \item
    A \defined{$\Val$--valued chambered invariant of $\sP$} consists of a subset,
    $\sW \subset \sP$, the \defined{wall},
    and an element
    \begin{equation*}
      \ChamberedInvariant \in \rH^0(\sP\setminus\sW;\Val).
    \end{equation*}
  \item
    The \defined{wall-crossing formula} of $\ChamberedInvariant \in
    \rH^0(\sP\setminus\sW;\Val)$
    is the equivalence class
    \begin{equation*}
      [\ChamberedInvariant]
      \in
      \SpaceOfWallCrossingFormulae(\sP,\sW;\Val)
      \coloneqq
      \coker \paren{\rH^0(\sP;G) \to \rH^0(\sP\setminus \sW;G)}.
      \qedhere
    \end{equation*}
  \end{enumerate}
\end{definition}

A chambered invariant lifts to an invariant if and only if its wall-crossing formula vanishes.
This applies, in particular, to the sum of two chambered invariants with opposite wall-crossing formulae.

The space of wall-crossing formulae $\SpaceOfWallCrossingFormulae(\sP,\sW;\Val)$ depends on the intersection behavior of $\sW$ and paths $\bp \co [0,1] \to \sP$.
Indeed,
the long exact sequence of relative cohomology induces an isomorphism
\begin{equation*}
  \Delta \co \SpaceOfWallCrossingFormulae(\sP,\sW;\Val)
  \iso
  \ker\paren{\rH^1(\sP,\sP\setminus\sW;\Val) \to \rH^1(\sP;\Val)}
\end{equation*}
such that
\begin{equation*}
  \Delta([\ChamberedInvariant])([\bp])
  = \ChamberedInvariant(\bp(1)) - \ChamberedInvariant(\bp(0)),
\end{equation*}
identifying
$\rH^1(\sP,\sP\setminus\sW;\Val) \iso \Hom(\rH_1(\sP,\sP\setminus\sW);\Val)$
via the universal coefficient theorem.

\begin{example}
  \label{Ex_WallCrossingHypersurface}
  Let $\sW \subset \sP$ is a closed subset and a codimension $1$ submanifold.
  Denote by $\fo$ the coorientation local system of $\sW$.
  \begin{enumerate}
  \item
    There is a homomorphism
    \begin{equation*}
      \sharp \co \rH^0(\sW;\Val \otimes_\Z \fo) \to \Hom(\rH_1(\sP,\sP\setminus\sW),\Val) \iso \rH^1(\sP,\sP\setminus\sW;\Val)
    \end{equation*}
    such that
    if $\bp \co [0,1] \to \sP$ is a $C^1$ path with $\bp(0),\bp(1) \notin \sW$ and transverse to $\sW$,
    then
    \begin{equation}
      \label{Eq_Sharp}
      \sharp(\alpha)[\bp]
      \coloneqq
      \sum_{t \in \bp^{-1}(\sW)} \alpha(\bp(t)).   
    \end{equation}
    \emph{This sum is finite because $\sW$ is closed.}
    Here $\alpha$ is evaluated at $\bp(t)$ using the coorientation of $\sW$ induced by $\dot\bp(t)$.
  \item
    \label{Ex_WallCrossingHypersurface_Thom}
    If $\sT$ is a closed tubular neighborhood $\sW$,
    then excision and direct inspection produce an injective homomorphism
    \begin{equation*}
      \SpaceOfWallCrossingFormulae(\sP,\sW;\Val)
      \stackrel{\Delta_{\loc}}{\into}
      \SpaceOfWallCrossingFormulae(\sT,\sW;\Val)
      \iso
      \rH^0(\sW;\Val \otimes_\Z \fo).
    \end{equation*}
    (This is an elementary instance of the Thom isomorphism.)
    These assemble into the commutative diagram
    \begin{equation*}    
      \begin{tikzcd}
        \SpaceOfWallCrossingFormulae(\sP,\sW;\Val) \ar{r}{\Delta_\loc} \ar[swap]{rd}{\Delta} &
        \rH^0(\sW;\Val \otimes_\Z \fo) \ar{d}{\sharp}  \\
        & \rH^1(\sP,\sP\setminus\sW;\Val).
      \end{tikzcd}
    \end{equation*} 
  \item
    As a consequence of the above, there is a short exact sequence
    \begin{equation*}
      \SpaceOfWallCrossingFormulae(\sP,\sW;\Val) \stackrel{\Delta_{\loc}}{\into} \rH^0(\sW;\Val \otimes_\Z \fo) \stackrel{\sharp}{\to} \rH^1(\sP;\Val).
      \qedhere
    \end{equation*}
  \end{enumerate}
\end{example}

Unfortunately, $\sW \subset \sP$ rarely is a codimension $1$ submanifold.
The following weaker condition is a reasonable substitute and does hold in the situation considered in this article.

\begin{definition}
  \label{Def_ProperWall}
  Let $\sP$ be a Banach manifold.
  A subset $\sW \subset \sP$ is a \defined{proper wall in $\sP$} if the following hold:
  \begin{enumerate}
  \item
    \label{Def_ProperWall_Closed}
    $\sW \subset \sP$ is a closed subset.
  \item
    \label{Def_ProperWall_Parametrisation}
    There are a Banach manifold $\sE$ and a Fredholm map $\pi \co \sE \to \sP$ of index at most $-1$ with
    \begin{equation*}
      \sW = \im \pi.
    \end{equation*}
    Moreover:
    \begin{enumerate}
    \item
      \label{Def_ProperWall_EssentiallyInjective}
      The map $\pi$ is \defined{essentially injective};
      that is:
      there are a Banach manifold $\sN$ and a Fredholm map $\nu \co \sN \to \sE$ of index at most $-1$ such that
      \begin{equation*}
        \pi|_{\sE\setminus \im\nu} \co \sE\setminus \im\nu \to \sP\setminus \im \paren{\pi\circ\nu}
      \end{equation*}
      is injective.
    \item
      \label{Def_ProperWall_EssentiallyProper}
      The map $\pi$ is \defined{essentially proper};
      that is: there are a Banach manifold $\sS$ and a Fredholm map $\sigma \co \sS \to \sE$ of index at most $-1$ such that
      \begin{equation*}
        \pi|_{\sE\setminus \paren{\im\nu \cup \im\sigma}} \co \sE\setminus \paren{\im\nu \cup \im\sigma} \to \sP\setminus \paren{\im \paren{\pi \circ \nu} \cup \im \paren{\pi\circ\sigma}}
      \end{equation*}
      is proper.
      \qedhere
    \end{enumerate}
  \end{enumerate}
\end{definition}

\begin{prop}
  \label{Prop_ProperWall}
  Assume the situation of \autoref{Def_ProperWall}.
  If $\bp \co [0,1] \to \sP$ is a $C^1$ path with $\bp(0), \bp(1) \notin \sW$ and transverse to
  $\pi$, $\pi \circ \nu$, and $\pi \circ \sigma$,
  then
  $\bp^*\sE$ is a finite set and the map $\bp^*\sE \to (0,1)$ is injective;
  that is: along $\bp$ (potential) wall-crossing occurs in a finite subset $\set{t_1,\ldots,t_N} \subset (0,1)$.
\end{prop}

\begin{proof}
  Since $\bp$ is transverse to $\pi$,
  $\bp^*\sE$ is a $0$--dimensional manifold.
  Since $\bp$ is transverse to $\pi \circ \nu$ and $\pi \circ \nu$,
  it $\bp^*\sE \to [0,1]$ is injective and proper.
  In particular, $\bp^*\sE$ is compact; hence: finite.
\end{proof}

\autoref{Sec_LocalWallCrossing} explains the analogue of \autoref{Ex_WallCrossingHypersurface} if $\sW$ is a proper wall.


\section{The wall of failure of $2$--rigidity}
\label{Sec_WallOfFailureOf2Rigidity}

The purpose of this section is to recall (or introduce) the notion of $2$--rigidity,
define the wall of failure of $2$--rigidity in the space of real Cauchy--Riemann operators,
and to prove that it is a proper wall in the sense of \autoref{Def_ProperWall}.

\medskip

Much of the material in this section is similar to the discussion in \cite{Wendl2016,Doan2018},
which itself is inspired by \cite{Taubes1996b,Eftekhary2016}.
The essential technical novelty is contained in establishing properties
\autoref{Def_ProperWall_Closed} and \autoref{Def_ProperWall_EssentiallyProper} in \autoref{Def_ProperWall}.
This is achieved using tools from geometric measure theory:
an idea already employed in \cite{Doan2018a,Doan2021}.

\subsection{Real Cauchy--Riemann operators}
\label{Sec_SelfAdjointRealCauchyRiemannOperators}

Let $C$ be a Riemann surface.
Let $V$ be a complex vector bundle over $C$.

\begin{definition}
  \label{Def_RealCauchyRiemanOperator}
  A \defined{real Cauchy--Riemann operator} on $V$ is an $\R$--linear differential operator
  \begin{equation*}
    \fd \co \Gamma\paren{C,V} \to \Omega^{0,1}(C,V) = \Gamma\paren{C,V \otimes_\C \overline{K}_C}
  \end{equation*}
  satisfying
  \begin{equation*}
    \fd (f\cdot s) = s \otimes_\C \delbar f + f \cdot \fd s
  \end{equation*}
  for every $f \in C_c^\infty(C,\R)$ and $s \in \Gamma(C,V)$.
  Here
  $\overline{K}_C \coloneqq \Hom_\C(\overline{TC},\C)$ and
  $\delbar f \coloneqq \tfrac12\paren{\rd f + i \cdot \rd f \circ j}$ with $j$ denoting the complex structure on $C$.
\end{definition}

\begin{remark}
  \label{Rmk_RealCauchyRiemanOperator~>DolbeaultOperator}
  A \defined{Dolbeault operator} on $V$ is a real Cauchy--Riemann operator $\delbar$ on $V$ which is also $\C$--linear.
  The Koszul--Malgrange theorem \cite{KoszulMalgrange1958} establishes a correspondence between
  holomorphic vector bundles $\sV$ over $C$
  and
  pairs $(V,\delbar)$ consisting of
  a complex vector bundle $V$ over $C$ and a Dolbeault operator $\delbar$ on $V$.
  If $\fd$ is a real Cauchy--Riemann operator on $V$,
  then
  \begin{equation*}
    \fd = \delbar + \fa
    \qwithq
    \delbar \coloneqq \frac12\paren*{\fd - i \fd i}
    \qandq
    \fa \coloneqq \frac12\paren*{\fd + i \fd i}.
  \end{equation*}
  By construction,
  $\delbar$ is Dolbeault operator,
  and $\fa \in \Gamma\paren{C,\Hom_\C(\overline{V},V \otimes_\C \overline{K}_C)}$.
\end{remark}

\begin{prop}
  \label{Prop_RealCauchyRiemanOperator_AffineSpace}
  The subspace 
  \begin{equation*}
    \SpaceOfRealCauchyRiemannOperators(V) \subset \Hom\paren{\Gamma(C,V),\Omega^{0,1}(C,V)}
  \end{equation*}
  of real Cauchy--Riemann operators on $V$ is an affine subspace modelled on
  \begin{equation*}
    \Gamma\paren{C,\HOMT(V)}
    \qwithq
    \HOMT(V) \coloneqq \Hom\paren{V,V \otimes_\C \overline{K}_C}.
    \pushQED{\qed} 
    \qedhere
  \end{equation*}
\end{prop}

\begin{remark}
  \label{Rmk_RiemannRoch}
  If $C$ is closed,
  then $\fd$ extends to a Fredholm operator $\fd \co W^{1,2}\Gamma(C,V) \to L^2\Omega^{0,1}(C,V)$ with
  \begin{equation*}
    \ind \fd = 2 \deg V + \rk_\C V \cdot \chi(C)
  \end{equation*}
  by Riemann--Roch.
  Here $W^{1,2}$ and $L^2$ are with respect to arbitrary (and immaterial) choices of
  a Riemannian metric on $C$ and
  a Hermitian inner product and a unitary covariant derivative on $V$.
\end{remark}

\medskip

Set
\begin{equation*}
  V^\dagger \coloneqq \Hom_\C(V,K_C) = V^* \otimes_\C K_C  
\end{equation*}
and $K_C \coloneqq \Hom_\C(TC,\C)$.
The isomorphism $K_C \otimes_\C \overline{K}_C \iso \Wedge^2 T^*C \otimes \C$ induces (complex) perfect pairings
\begin{equation}
  \label{Eq_PerfectPairing}
  \Inner{\cdot,\cdot} \co \paren{V^\dagger \otimes_\C \overline{K}_C} \otimes_\C V \to \Wedge^2 T^*C \otimes \C
  \qandq
  \Inner{\cdot,\cdot} \co V^\dagger \otimes_\C \paren{V \otimes_\C \overline{K}_C} \to \Wedge^2 T^*C \otimes \C
\end{equation}
with respect to which
\begin{equation*}
  \Inner{\lambda \otimes_\C \overline\zeta,v} = \Inner{\lambda,v \otimes_\C \overline\zeta}
\end{equation*}
for every $\lambda \in V^\dagger$, $v \in V$, and $\overline\zeta \in \overline{K}_C$.

\begin{definition}
  \label{Def_RealCauchyRiemannOperator_Adjoint}
  The \defined{adjoint} of a real Cauchy--Riemann operator $\fd$ on $V$ is the real Cauchy--Riemann operator $\fd^\dagger$ on $V^\dagger$ over $C$ characterised by
  \begin{equation*}
    \int_C \Re\Inner{\fd^\dagger s,t} = \int_C \Re\Inner{s,\fd t}
  \end{equation*}
  for every $s \in \Gamma_c(C,V^\dagger)$ and $t \in \Gamma_c(C,V)$.
\end{definition}

\begin{prop}
  \label{Prop_RealCauchyRiemannOperator_Adjoint_Affine}
  The map
  \begin{equation*}
    \cdot^\dagger \co \SpaceOfRealCauchyRiemannOperators(V) \to \SpaceOfRealCauchyRiemannOperators(V^\dagger)
  \end{equation*}
  is an isomorphism of affine spaces;
  in fact:
  $\paren{\fd^\dagger}^\dagger = \fd$
  with respect to the identification $(V^\dagger)^\dagger = V \otimes_\C K_C^* \otimes_\C K_C = V$.
  \qed
\end{prop}

\medskip

The bulk of this article is concerned with the case
\begin{equation*}
  \ind \fd = 2\deg V + \rk_\C V \cdot \chi(C) = 0.
\end{equation*}
This is equivalent to $\deg(V) = \deg(V^\dagger)$;
hence: $V \iso V^\dagger$.

\begin{definition}
  \label{Def_RealCauchyRiemannOperator_SelfAdjoint}
  Let $\SerreOperator \co V \iso V^\dagger$ be an isomorphism.
  A real Cauchy--Riemann operator $\fd$ on $V$ over $C$ is \defined{$\SerreOperator$--self-adjoint} if
  \begin{equation*}
    \fd^\dagger = \fd^\SerreOperator \coloneqq \paren{\SerreOperator \otimes_\C \one_{\overline{K}_C}} \circ \fd \circ \SerreOperator^{-1}.
    \qedhere
  \end{equation*}
\end{definition}

\begin{prop}
  \label{Prop_RealCauchyRiemanOperator_SelfAdjoint_AffineSpace}
  Let $\SerreOperator \co V \iso V^\dagger$ be an isomorphism.
  The subspace
  \begin{equation*}
    \SpaceOfRealCauchyRiemannOperators(V,\SerreOperator) \subset \Hom\paren{\Gamma(C,V),\Omega^{0,1}(C,V)}
  \end{equation*}
  of $\SerreOperator$--self-adjoint real Cauchy--Riemann operators on $V$ is an affine subspace modelled on
  \begin{equation*}
    \Gamma\paren{C,\HOMT(V,\SerreOperator)}
    \qwithq    
    \HOMT(V,\SerreOperator)
    \coloneqq
    \set{
      \fa \in \HOMT(V)
      :
      \fa^\dagger = \fa^\SerreOperator
    }.
  \end{equation*}
\end{prop}

\begin{proof}
  In light of \autoref{Prop_RealCauchyRiemanOperator_AffineSpace} it suffices to prove that $\SpaceOfRealCauchyRiemannOperators(V,\SerreOperator)$ is non-empty.  
  A moment's thought shows that
  if $\fd \in \SpaceOfRealCauchyRiemannOperators(V)$,
  then
  $\frac12\paren{\fd + \paren{\SerreOperator\circ \one_{\overline{K}_C}}^{-1} \circ \fd^\dagger \circ \SerreOperator}$ is $\SerreOperator$--self-adjoint.
  In particular, $\SpaceOfRealCauchyRiemannOperators(V,\SerreOperator) \neq \emptyset$.  
\end{proof}

\begin{remark}
  \label{Rmk_SelfAdjointLinearStructure}
  The isomorphism $\SerreOperator$ induces an isomorphism
  $V \otimes_\C \overline{K}_C \iso V^\dagger \otimes_\C \overline{K}_C \iso V^* \otimes \Wedge^2 T^*C$.
  This, in turn, induces an isomorphism  
  \begin{equation*}
    \HOMT(V,\SerreOperator)
    \iso
    S^2V^* \otimes \Wedge^2 T^*C.
  \end{equation*}
  Here $S^2 V^*$ denotes the second symmetric power of $V^*$.
\end{remark}

\begin{remark}
  \label{Rmk_SelfAdjointRealCauchyRiemannOperators}
  Let $X$ be a Calabi--Yau $3$--fold together with a choice of holomorphic volume form $\theta \in \Omega^{3,0}(X,\C)$.
  If $C \subset X$ is a holomorphic curve,
  then $\rho \coloneqq \Re \theta$ induces an isomorphism $\gamma \co NC \iso NC^\dagger$ and the Dolbeault operator $\delbar_{NC}$ governing the infinitesimal deformation theory of $C$ is $\gamma$--self-adjoint.
  In fact, $\delbar_{NC}$ is the (formal) Hessian at $C$ of the (real part) of the Chern--Simons functional $\Psi$ mentioned in \cite[§8]{Donaldson1998}.
  
  If $(X,J)$ is an almost complex $3$--fold with $c_1(X,J) = 0$,
  then there is a definite $3$--form $\rho \in \Omega^3(X)$ in the sense of   \cite[Definition 1]{Donaldson2018};
  this amounts to a reduction of structure group along $\SL_3(\C) \incl \GL_3(\C)$.
  If $C \subset X$ is a $J$--holomorphic curve,
  then $\rho$ induces an isomorphism $\gamma \co NC \iso NC^\dagger$ and the corresponding real Cauchy--Riemann operator $\fd_{NC}$ on $NC$ is $\gamma$--self-adjoint.
  If $\rho$ is closed, then $\fd_{NC}$ is self-adjoint and $\Re\Psi$ is defined.

  Typically, $\rho$ cannot be arranged to be to be closed.
  However, there is an $h$--principle for closed definite $3$--forms \cites[§4]{Donaldson2018}[Theorem 7.2]{Mayther2023:HPrinciples}.
  In particular, there is a closed definite $3$--form $\rho'$ isotopic to $\rho$.
  Unfortunately, the $C^0$--dense version of this $h$--principle does not hold.
  In particular,
  if $(X,\omega)$ is a symplectic Calabi--Yau $3$--fold,
  then it is not clear whether or not a closed definite $3$--form tamed by $\omega$ exists.  
  It would be interesting to better understand this issue. 
\end{remark}


\subsection{Homogeneous almost complex structures}
\label{Sec_HomogeneousAlmostComplexStructure}

This section reviews the correspondence between
real Cauchy--Riemann operators on $V$ and
homogeneous almost complex structures on the total space of $V$. 
Denote by $p \co V \to C$ the projection map.

\begin{definition}
  \label{Def_HomogeneousAlmostComplexStructure}
  An almost complex structure $J \in \Gamma(V,\End(TV))$ on the total space of $V$ is \defined{homogeneous}  
  if:
  \begin{enumerate}
  \item
    \label{Def_HomogeneousAlmostComplexStructure_Compatible}
    The linear maps
    \begin{equation*}
      p^*V \stackrel{\iota}{\incl} TV \qandq TV \stackrel{p_*}{\onto} p^*TC
    \end{equation*}
    are complex linear.
  \item
    \label{Def_HomogeneousAlmostComplexStructure_Scaling}
    For every $\lambda \in \R$
    the linear map $\lambda \co V \to V$ given by multiplication with $\lambda$ is $J$--holomorphic.
    \qedhere
  \end{enumerate} 
\end{definition}

\begin{lemma}
  \label{Lem_HomogeneousSections}
  Denote by $\tau \in \Gamma(V,p^*V)$ the tautological section
  \begin{equation*}
  	\tau(v) = v.
  \end{equation*}
  Let $W$ be a vector bundle over $C$.
  Let $s \in \Gamma(V,p^*W)$.
  The following are equivalent:
  \begin{enumerate}
  \item
    \label{Lem_HomogeneousSections_Scaling}
    The section $s$ is \defined{homogeneous};
    that is:
    for every $\lambda \in \R$
    \begin{equation*}
      \lambda^*s = \lambda s
    \end{equation*}
    with respect to the identification $\Gamma(V,\lambda^*p^*W) = \Gamma(V,p^*W)$ induced by $p\circ \lambda = p$.
  \item
    \label{Lem_HomogeneousSections_Pullback}
    There is an $\hat s \in \Gamma\paren{X,\Hom(V,W)}$ such that for every $v \in V$
    \begin{equation*}
      s(v) = \hat s(p(v))(v);
    \end{equation*}
    that is:
    \begin{equation*}
      s \in \im\sqparen{\Inner{p^*\cdot,\tau} \co \Gamma(X,\Hom(V,W)) \incl \Gamma(V,p^*W)}.
    \end{equation*}
  \end{enumerate}
\end{lemma}

\begin{proof}
  Evidently, \autoref{Lem_HomogeneousSections_Pullback} implies \autoref{Lem_HomogeneousSections_Scaling}.

  To prove that \autoref{Lem_HomogeneousSections_Scaling} implies \autoref{Lem_HomogeneousSections_Pullback} it suffices to consider $V = \ubC^r$ and $W = \ubR^s$, by locality.
  If $s \in C^\infty(C\times \C^r,\R^s)$ is homogeneous,
  then
  \begin{equation*}
    s(x,v) = \rd_{(x,0)}s(v).
  \end{equation*}
  Therefore, $s \in \im \Inner{p^*\cdot,\tau}$.  
\end{proof}

\autoref{Lem_HomogeneousSections} and the inclusion
\begin{equation*}
  p^*\paren{\overline{K}_C \otimes_\C V}
  \subset p^*\Hom(TC,V)
  \xhookrightarrow{\iota \circ \cdot \circ p_*} \End(TV)
\end{equation*}
induce an inclusion
\begin{equation*}
  \Gamma\paren{C,\Hom(V,V\otimes_\C\overline{K}_C)} \incl \Gamma(V,\End(TV)).
\end{equation*}

\begin{prop}
  \label{Prop_HomogeneousAlmostComplexStructure_AffineSpace}
  The subspace
  \begin{equation*}
    \SpaceOfHomogeneousAlmostComplexStructures(V) \subset \Gamma(V,\End(TV))
  \end{equation*}
  of homogeneous almost complex structures on the total space of $V$ is an affine subspace modelled on $\Gamma\paren{C,\Hom(V,V\otimes_\C\overline{K}_C)}$.
\end{prop}

\begin{proof}
  Evidently, $\SpaceOfHomogeneousAlmostComplexStructures(V)$ is non-empty.
  
  A moment's thought shows that
  if $J_0 \in \SpaceOfHomogeneousAlmostComplexStructures(V)$ and $\fa \in \Gamma\paren{C,\Hom(V,V\otimes_\C\overline{K}_C)}$,
  then $J_0 + \fa \in \SpaceOfHomogeneousAlmostComplexStructures(V)$.
    
  Let $J_0,J \in \SpaceOfHomogeneousAlmostComplexStructures(V)$.
  Set $\fa \coloneqq J-J_0 \in \Gamma(V,\End(TV))$.
  By \autoref{Def_HomogeneousAlmostComplexStructure}~\autoref{Def_HomogeneousAlmostComplexStructure_Compatible},
  $\fa \in \Gamma(V,p^*\Hom(TC,V))$.
  Since $J^2 = (J_0+\fa)^2$ and also using \autoref{Def_HomogeneousAlmostComplexStructure}~\autoref{Def_HomogeneousAlmostComplexStructure_Compatible},
  $\fa \in \Gamma(V,p^*\paren{V\otimes_\C\overline{K}_C})$.
  By \autoref{Def_HomogeneousAlmostComplexStructure}~\autoref{Def_HomogeneousAlmostComplexStructure_Scaling},
  $\fa$ is homogeneous.
  Therefore, by \autoref{Lem_HomogeneousSections},
  $\fa \in \Gamma\paren{C,\Hom(V,V\otimes_\C\overline{K}_C)}$.
\end{proof}

\begin{prop}
  \label{Prop_HomogeneousAlmostcomplexStructure=RealCauchyRiemanOperator}
  There is an isomorphism of affine spaces 
  \begin{equation*}
    \fd_\cdot \co \SpaceOfHomogeneousAlmostComplexStructures(V) \to \SpaceOfRealCauchyRiemannOperators(V)   
  \end{equation*}  
  such that for every $J \in \SpaceOfHomogeneousAlmostComplexStructures$ and $s \in \Gamma(C,V)$
  \begin{equation*}
    \fd_J s = \tfrac12\paren{Ts - J \circ Ts \circ j} \in \Omega^{0,1}(C,V) \subset \Omega^{0,1}(C,s^*TV);
  \end{equation*}
  in particular:
  $s \in \ker \fd_J$ if and only if $s$ is $J$--holomorphic.
\end{prop}

\begin{proof}
  Let $s \in \Gamma(C,V)$.
  Define $\tilde \fd_Js \in L^2\Gamma\paren{C,s^*TV\otimes_\C\overline{K}_C}$ by the above formula.
  
  By \autoref{Def_HomogeneousAlmostComplexStructure}~\autoref{Def_HomogeneousAlmostComplexStructure_Compatible}, $p$ is $J$--holomorphic.
  Therefore and since $p \circ s = \id_C$, 
  $\tilde \fd_J s \in \Gamma\paren{C,V\otimes_\C\overline{K}_C}$.

  To show that $\tilde\fd_J$ is a real Cauchy--Riemann operator, by locality, it suffices to consider $V = \ubC^r$. 
  In this case, $TV = p^*(TC \oplus V)$ and, by \autoref{Prop_HomogeneousAlmostComplexStructure_AffineSpace},
  \begin{equation*}
    J =
    \begin{pmatrix}
      j & 0 \\
      \fa & i
    \end{pmatrix}.
  \end{equation*}  
  A direct computation shows that $\fd_J = \delbar - \frac12\fa$.
  
  This constructs the map $J \mapsto \fd_J$.
  Evidently, it is affine and its underlying linear map is the isomorphism $-\frac12$.
\end{proof}


\subsection{Four perspectives on $2$--rigidity}
\label{Sec_ThreePerspectivesOn2Rigidity}

The purpose of this section is to introduce the concept of $2$--rigidity and present four points of view on it.
Henceforth, assume that $C$ is closed and connected (unless said otherwise).
Moreover, choose a Hermitian metric on $V$.

\begin{prop}
  \label{Prop_RealCauchyRiemanOperator_Pullback}
  Let $\fd$ be a real Cauchy--Riemann operator on $V$.
  Let $\tilde C$ be a Riemann surface.
  Let $\pi \co \tilde C \to C$ be a holomorphic map.
  There is a unique real Cauchy--Riemann operator $\pi^*\fd$ on $\pi^*V$,
  the \defined{pull-back} of $\fd$ along $\pi$,
  such that the diagram
  \begin{equation*}
    \begin{tikzcd}
      \Gamma(C,V) \ar{r}{\fd} \ar{d}{\pi^*} & \Omega^{0,1}(C,V) \ar{d}{\pi^*} \\
      \Gamma(\tilde C,\pi^*V) \ar{r}{\pi^*\fd} & \Omega^{0,1}(\tilde C,\pi^*V)
    \end{tikzcd}
  \end{equation*}
  commutes.
\end{prop}

\begin{proof}
  By locality,
  it suffices to prove this for $V = \ubC^r$.
  In this case, $\pi^*V = \ubC^r$ and
  the assertion is a consequence of the isomorphism $C^\infty(C,\C^r) \otimes_{C^\infty(C)} C^\infty(\tilde C) \iso C^\infty(\tilde C,\C^r)$.
\end{proof}

\begin{definition}
  \label{Def_2Rigid}
  Suppose that $C$ is closed.
  A real Cauchy--Riemann operator $\fd$ on $V$ is \defined{$2$--rigid} if
  for every closed, connected Riemann surface $\tilde C$ and
  every non-constant holomorphic map $\pi \co \tilde C \to C$ of degree at most two
  \begin{equation*}
    \ker \pi^*\fd = 0.
    \qedhere
  \end{equation*}
\end{definition}

\begin{remark}
  \label{Rmk_KRigidity}
  There are, of course, concepts of $k$--rigidity for $k \in \N \cup\set{\infty}$;
  cf.~\cites[Definition 1.6]{Bryan2001}[§1]{Eftekhary2016}{Wendl2016}.
  \autoref{Rmk_WhyRestrictToK=2} contains an attempt to justify the restriction to $k=2$ in the present article.
\end{remark}

\medskip

The second perspective on $2$--rigidity is a translation of the above into homogeneous almost complex structures using the correspondence reviewed in \autoref{Sec_HomogeneousAlmostComplexStructure}.

\begin{prop}
  \label{Prop_HomogeneousAlmostComplexStructure_Pullback}
  Let $J \in \SpaceOfHomogeneousAlmostComplexStructures(V)$.
  Let $\pi \co \tilde C \to C$ be holomorphic.
  The following hold:
  \begin{enumerate}
  \item
    \label{Prop_HomogeneousAlmostComplexStructure_Pullback_Defined}
    $\pi^*V \subset \tilde C \times V$ is an almost complex submanifold, and
    $\pi^*J \coloneqq \paren{\tilde j \times J}|_{\pi^*V}$ is homogeneous.
  \item
    \label{Prop_HomogeneousAlmostComplexStructure_Pullback_CauchyRiemann}
    $\fd_{\pi^*J} = \pi^*\fd_J$.
  \end{enumerate}
\end{prop}

\begin{proof}
  Since $p$ is $J$--holomorphic and $\pi$ is holomorphic,
  $\pi^*V \subset \tilde C \times V$ is an almost complex submanifold.
  Since $J$ is homogeneous, so is $\pi^*J$.
  This proves \autoref{Prop_HomogeneousAlmostComplexStructure_Pullback_Defined}.

  Let $s \in \Gamma(C,V)$.
  Inspection of the commutative diagram
  \begin{equation*}
    \begin{tikzcd}
      \pi^*V \arrow[r, "p^*\pi"] \arrow[d, "\pi^*p"] & V \arrow[d,swap,"p"] \\
      \tilde C \arrow[u, bend left, "\pi^*s"] \ar[r,"\pi"] & C \arrow[u,bend right,swap,"s"]
    \end{tikzcd}
  \end{equation*}
  reveals that
  \begin{equation*}
    T p^*\pi \circ \paren{T \pi^*s - \pi^*J \circ T \pi^*s \circ \tilde j}
    =
    \paren{T s - J \circ T s \circ j} \circ T\pi.
  \end{equation*}
  This proves \autoref{Prop_HomogeneousAlmostComplexStructure_Pullback_CauchyRiemann}.
\end{proof}

\begin{cor}
  \label{Cor_JHolomorphicMapVsKernelElementOfPullback}
  Let $J \in \SpaceOfHomogeneousAlmostComplexStructures(V)$.
  Set $\fd \coloneqq \fd_J$.
  The following hold:
  \begin{enumerate}
  \item
    If $\pi \co \tilde C \to C$ is holomorphic and $s \in \ker \pi^*\fd$,
    then $u \coloneqq (p^*\pi) \circ s \co \tilde C \to V$ is $J$--holomorphic.
  \item
    If $u \co \tilde C \to V$ is $J$--holomorphic,
    then $\pi \coloneqq p \circ u$ is holomorphic and
    $s \coloneqq (\id_{\tilde C},u) \in \ker \pi^*\fd$.
    \qedhere
  \end{enumerate}
\end{cor}

\medskip

The above directly leads to the third perspective,
since $J$--holomorphic maps induce $J$--holomorphic cycles.
This is spelled out in detail in \cite[§4,§5]{Doan2018a}.
Here is a brief summary.

\begin{definition}
  \label{Def_JHolomorphicCycle}
  Let $X$ be a smooth manifold equipped with
  an almost complex structure $J$,
  a Hermitian form $\sigma$, and
  the Riemannian metric $g \coloneqq \sigma(\cdot,J\cdot)$.
  A \defined{$J$--holomorphic cycle} is
  a closed integral $2$--current $T \in \Hom\paren{\Omega_c^2(X),\R}$ which is semi-calibrated by $\sigma$ and has finite mass: $\bM(T) < \infty$.
\end{definition}

\begin{remark}
  \label{Rmk_HermitianFromAlmostRedHerring}
  The choice of a Hermitian form $\sigma$ is somewhat of a red herring:
  if $\sigma$, $\sigma'$ are Hermitian forms (with respect to $J$),
  then $T \in \Hom\paren{\Omega_c^2(X),\R}$ is a closed integral current which is semi-calibrated by $\sigma$ if and only if the same holds for $\sigma'$.
  However, the mass $\bM(T)$ depends on $\sigma$.
\end{remark}

\begin{remark}
  \label{Rmk_CycleHomologyClass}
  A $J$--holomorphic cycle $T$ with compact support $\supp(T)$ represents a homology class $[T] \in \rH_2(X,\Z)$.
  If $X$ is the total space of $V$,
  then $\rH_2(V,\Z) \iso \rH_2(C,\Z) \iso \Z$ is generated by $[0]$, the fundamental class of the zero section.
  Therefore, $[T]$ is determined by
  \begin{equation*}
    \deg(T) \coloneqq [T]/[0] \in \Z.
    \qedhere
  \end{equation*}
\end{remark}

\begin{prop}
  \label{Prop_JHolomorphicMapsDefineJHolomorphicCycles}
  Let $X$ be a smooth manifold equipped with an almost complex structure $J$ and a Hermitian form $\sigma$.
  If $C$ is a not necessarily closed Riemann surface and $u \co C \to X$ is a proper $J$--holomorphic map,
  then the closed integral $2$--current $T_u \in \Hom\paren{\Omega_c^2(X),\R}$ defined by
  \begin{equation*}
    T_u(\alpha) \coloneqq \int_C u^*\alpha
  \end{equation*}
  is a $J$--holomorphic cycle with
  \begin{equation*}
    \bM(T_u) = T_u(\sigma) \qandq
    \supp(T_u) = \im u;
  \end{equation*}
  moreover,
  if $C$ is closed,
  then $[T_u] = u_*[C]$.
  \qed
\end{prop}

\begin{theorem}[{\citet{DeLellis2017,DeLellis2017a,DeLellis2017b,DeLellis2017c}; cf.~\cite[Lemma 5.5]{Doan2018a}}]
  \label{Lem_EveryPseudoHolomorphicCycleArisesFromAPseudoHolomorphicMap}
  Let $X$ be a smooth manifold equipped with an almost complex structure $J$ and a Hermitian form $\sigma$.
  If $T$ is a $J$--holomorphic cycle,
  then there are a not necessarily closed Riemann surface $C$ and a proper $J$--holomorphic map $u \co C \to X$ with $T = T_u$;
  moreover: if $\supp T$ is compact, then $C$ is closed.
  \qed
\end{theorem}

\medskip

Here is the fourth, more involved, perspective on $2$--rigidity.
It is based on the following concept.

\begin{definition}
  \label{Def_RamifiedEuclideanLineBundle}
  A \defined{ramified Euclidean line bundle} over $C$
  consists of:
  \begin{enumerate}
  \item
    \label{Def_RamifiedEuclideanLineBundle_BranchLocus}
    a finite subset $\Br(\fl) \subset C$, the \defined{branching locus}, and
  \item
    \label{Def_RamifiedEuclideanLineBundle_Bundle}
    a Euclidean line bundle $\fl$ over $\mathring C \coloneqq C \setminus \Br(\fl)$
  \end{enumerate}
  such that:
  \begin{enumerate}[resume]
  \item
    \label{Def_RamifiedEuclideanLineBundle_Monodromy}
    the monodromy representation $\mu \co \pi_1(\mathring{C},*) \to \O(1) = \set{\pm 1}$ restricts to a non-trivial homomorphism on
    $\ker \paren{\pi_1(\mathring{C},*) \to \pi_1(\mathring{C} \cup \set{b},*)} \iso \Z$
    for every 
    $b \in \Br(\fl)$.
    \qedhere
  \end{enumerate}
\end{definition}

\begin{remark}
  \label{Rmk_EuclideanLineBundle_Monodromy}
  By Seifert--van Kampen,
  \begin{equation*}
    \pi_1(\mathring{C},*)
    \iso
    \Span{\alpha_1,\beta_1,\cdots,\alpha_g,\beta_g,x_1,\ldots,x_k \mid [\alpha_1,\beta_1]\cdots[\alpha_g,\beta_g] = x_1\cdots x_k }
  \end{equation*}
  with $g$ denoting the genus of $C$ and $k \coloneqq \#\Br(\fl)$.
  In particular,
  \autoref{Def_RamifiedEuclideanLineBundle_Monodromy} enforces that $\#\Br(\fl)$ is even.
\end{remark}

\begin{definition}
  \label{Def_RamifiedEuclideanLineBundle_Isomorphic}
  Two ramified Euclidean line bundles $\fl_1,\fl_2$ over $C$ are \defined{isomorphic} if $\Br(\fl_1) = \Br(\fl_2)$ and $\fl_1 \iso \fl_2$ as Euclidean line bundles over $\mathring{C}$.
\end{definition}

\begin{remark}
  \label{Rmk_EuclideanLineBundle_StiefelWhitney}
  If $\fl_1,\fl_2$ are Euclidean line bundles over $\mathring{C}$,
  then they are isomorphic if and only if their monodromy representations
  $\mu(\fl_i) \in \Hom(\pi_1(\mathring{C},*),\set{\pm 1})$
  or, equivalently, their first Stiefel--Whitney classes
  \begin{equation*}
    w_1(\fl_i) \in  \rH^1(\mathring{C},\Z/2\Z)
  \end{equation*}
  agree.
\end{remark}

\begin{prop}
  \label{Prop_RealCauchyRiemannOperator_Twist}
  Let $\fd$ be a real Cauchy--Riemann operator on $V$.
  Let $\fl$ be a ramified Euclidean line bundle over $C$.
  There is a unique real Cauchy--Riemann operator
  \begin{equation*}
    \fd^\fl \co \Gamma(\mathring{C},V\otimes\fl) \to \Omega^{0,1}(\mathring{C},V\otimes\fl)
  \end{equation*}
  on $V \otimes \fl$ over $\mathring{C}$,
  the \defined{twist} of $\fd$ by $\fl$,
  satisfying
  \begin{equation*}
    \fd^\fl(s\otimes t)
    = \fd s \otimes t + (s \otimes \nabla t)^{0,1}
  \end{equation*}
  for every $s \in \Gamma\paren{\mathring{C},V}$ and $t \in \Gamma\paren{\mathring{C},\fl}$.
  Here $\nabla$ denotes the unique orthogonal covariant derivative on $\fl$.
\end{prop}

\begin{proof}
  By locality, it suffices to consider $\fl = \ubR$.
  In this case $V \otimes \fl = V$ and $\fd^\fl = \fd$.
\end{proof}

\begin{remark}
  \label{Rmk_RealCauchyRiemannOperator_Twist_BoundedLinear}
  Evidently,
  the restriction of $\fd^\fl$ to $\Gamma_c\paren{\mathring{C},V\otimes\fl}$ extends to a bounded linear operator
  \begin{equation}
    \label{Eq_RealCauchyRiemannOperator_Twist_BoundedLinear}
    \fd^\fl \co W^{1,2}\Gamma\paren{\mathring{C},V\otimes\fl} \to L^2\Omega^{0,1}\paren{\mathring{C},V\otimes\fl}.
  \end{equation}
  Here $W^{1,2}$ and $L^2$ are with respect to (immaterial) choices of
  a Riemannian metric on $C$,
  a Hermitian inner product and a unitary covariant derivative on $V$; and
  the Euclidean inner product and unique orthogonal covariant derivative on $\fl$.
  In fact, $\fd^\fl$ is Fredholm; cf.~\autoref{Sec_IndexFormulaAndResidues}.
\end{remark}

\begin{prop}
  \label{Prop_RiemannSurface_DegreeTwo_+-}
  If $\fl$ is a ramified Euclidean line bundle over $C$,
  then the following hold:
  \begin{enumerate}
  \item
    \label{Prop_RiemannSurface_DegreeTwo_+-_Principal}
    The smooth map
    \begin{equation*}
      \mathring{\pi} \co \mathring{\tilde C} \coloneqq \set{ (x,v) \in \fl : \abs{v} = 1 } \to \mathring{C}
    \end{equation*}
    is a $\set{\pm 1}$--principal covering map, and
    $\fl \iso \mathring{\tilde C} \times_{\set{\pm 1}} \R$.
  \item
    \label{Prop_RiemannSurface_DegreeTwo_+-_Extension}
    The covering map $\mathring{\pi}$ extends to a holomorphic map $\pi \co \tilde C \to C$ (unique up to unique biholomorphism).
  \end{enumerate}
  Moreover,
  every holomorphic map $\pi \co \tilde C \to C$ of degree two arises in this way.
\end{prop}

\begin{proof}
  Let $\fl$ be a ramified Euclidean line bundle over $C$.
  \autoref{Prop_RiemannSurface_DegreeTwo_+-_Principal} is obvious.
  To verify \autoref{Prop_RiemannSurface_DegreeTwo_+-_Extension} it suffices to consider the following ramified Euclidean line bundle over $D \coloneqq \set{ z \in \C : \abs{z} < 1 }$:
  \begin{equation}
    \label{Eq_MobiusBundle}
    \fm \coloneqq \set{ (z,w) \in \mathring{D} \times \C : w^2 \in \R_{\geq 0}\cdot z } \to \mathring{D} \coloneqq D \setminus \set{0}
  \end{equation}
  with the Euclidean inner product
  $\Inner{(z,w_1),(z,w_2)} \coloneqq w_1w_2 z^{-1}$.
  For $\fl = \fm$,
  \begin{equation*}
    \mathring{\tilde C} = \set{ (z,w) \in \mathring{D} \times \mathring{D} : w^2 = z } \qandq
    \mathring{\pi}(z,w) = w^2.
  \end{equation*}
  Since $\mathring{\tilde C} \iso \mathring{D}$ via $(z,w) \mapsto w$,
  the double cover $\mathring{\pi} \co \mathring{\tilde C} \iso \mathring{D} \to \mathring{D}$ extends to the holomorphic map $\pi \co D \to D, w \mapsto w^2$.

  Let $\pi \co \tilde C \to C$ be a holomorphic map of degree two.
  Denote by $\Br(\fl) \subset C$ the set of critical values of $\pi$.
  Set $\mathring{C} \coloneqq C \setminus \Br(\fl)$ and $\mathring{\tilde C} \coloneqq \pi^{-1}(\mathring{C})$.
  The restriction $\mathring{\pi} \coloneqq \pi|_{\mathring{\tilde C}} \co \mathring{\tilde C} \to \mathring{C}$ is a double covering;
  hence, a $\set{\pm 1}$--principal covering map.
  The associated Euclidean line bundle
  \begin{equation*}
    \fl \coloneqq \paren{\mathring{\tilde C} \times \R}/\set{\pm 1}
  \end{equation*}
  is a ramified Euclidean line bundle over $C$ with branching locus $\Br(\fl)$.
  Evidently, the above construction applied to $\fl$ yields $\pi \co \tilde C \to C$.
\end{proof}

\newcommand{\JetSpace}{J}
\begin{remark}
  \label{Rmk_SquareRoot}
  Assume the situation of \autoref{Prop_RiemannSurface_DegreeTwo_+-}.
  Let $b \in \Br(\fl)$.
  The Euclidean line bundle $\fl$ determines a square root of $T_b^*C$:
  \begin{enumerate}
  \item
    \label{Rmk_SquareRoot_TildeC}
    Set $\set{\tilde b} \coloneqq \pi^{-1}(b)$.
    Denote by $\sO_b$ and $\sO_{\tilde b}$ the local ring of germs of holomorphic functions at $b \in C$ and $\tilde b \in \tilde C$ respectively.
    Denote by $\fm_b \subset \sO_b$ and $\fm_{\tilde b} \subset \sO_{\tilde b}$ the maximal ideals (of germs of holomorphic functions vanishing at $b \in C$ and $\tilde b \in  \tilde C$ respectively).
    The holomorphic map $\pi$ induces an injective ring homomorphism $\pi^* \co \sO_b \into \sO_{\tilde b}$ satisfying
    $\pi^*\fm_b = \fm_{\tilde b}^2$.
    Therefore,
    \begin{equation*}
      T_b^*C
      =
      \fm_b \otimes_{\sO_b} \paren{\sO_b/\fm_b}
      =
      \frac{\fm_b}{\fm_b^2}
      \iso
      \paren[\Big]{\frac{\fm_{\tilde b}}{\fm_{\tilde b}^2}}^{\otimes 2} = \paren{T_{\tilde b}^*\tilde C}^{\otimes 2}.
    \end{equation*}
  \item
    \label{Rmk_SquareRoot_12JetSpace}
    Set
    \begin{equation*}
      (\fl\otimes\C(-b))_b \coloneqq \varinjlim_{U \ni b} L^\infty\rH^0\paren{U\setminus\set{b},\fl\otimes \C}
    \end{equation*}
    with $L^\infty\rH^0$ denoting bounded holomorphic sections.
    (In the situation discussed in the proof of \autoref{Prop_RiemannSurface_DegreeTwo_+-},
    such a section has an asymptotic expansion
    $a_0 z^{1/2} + a_1 z^{3/2} + \ldots$ at $0$.)
    $(\fl\otimes\C(-b))_b$ is an $\sO_b$--module.
    Define the \defined{$\tfrac12$--jet space of $\fl\otimes\C$ at $b$} by
    \begin{equation*}
      J_b^{1/2}(\fl\otimes\C) \coloneqq (\fl\otimes\C(-b))_b \otimes_{\sO_b} \paren{\sO_b/\fm_b}.
    \end{equation*}
    The canonical isomorphism $\fl^{\otimes 2} \iso \ubR$ induces a canonical isomorphism $(\fl\otimes\C(-b))_b^{\otimes 2} \iso \fm_b$.
    Therefore, $\paren{J_b^{1/2}(\fl\otimes\C)}^{\otimes 2} \iso T_b^*C$.
  \end{enumerate}
  Of course, these constructions are related.
  The holomorphic map $\pi$ induces an $\sO_b$--module isomorphism $\pi^* \co \paren{\fl\otimes\C}_b \iso \fm_{\tilde b}$;
  therefore:
  \begin{equation*}
    J_b^{1/2}(\fl\otimes\C) \iso T_{\tilde b}\tilde C.
    \qedhere
  \end{equation*}
\end{remark}

\begin{remark}
  \label{Rmk_12JetSpace_Notation}
  The notation $(\fl\otimes\C(-b))_b$ is motivated by \autoref{Prop_X}.
\end{remark}

\begin{prop}
  \label{Prop_RealCauchyRiemannOperator_+-}
  Assume the situation of \autoref{Prop_RiemannSurface_DegreeTwo_+-}.
  Let $\fd$ be a real Cauchy--Riemann operator on $V$.
  The following hold:
  \begin{enumerate}
  \item
    \label{Prop_RealCauchyRiemannOperator_+-_Equivariant}
    The pull-back $\pi^*\fd$ is $\set{\pm 1}$--equivariant.
    In particular,
    \begin{equation*}
      \ker \pi^*\fd = (\ker \pi^*\fd)^+ \oplus (\ker \pi^*\fd)^-
    \end{equation*}
    with $(\cdot)^+$ and $(\cdot)^-$ denoting the invariant and anti-invariant subspaces respectively.
  \item
    \label{Prop_RealCauchyRiemannOperator_+-_Isomorphism_+}
    The isomorphism $\pi^* \co \Gamma(C,V) \iso \Gamma\paren{\tilde C,\pi^*V}^+$ induces an isomorphism
    \begin{equation*}
      \pi^* \co \ker \fd \iso \paren{\ker \pi^*\fd}^+.
    \end{equation*}
  \item
    \label{Prop_RealCauchyRiemannOperator_+-_Isomorphism_-}
    The inclusion $\mathring{\pi}_! \co \Gamma\paren{\tilde C,\pi^*V}^- \incl
    W^{1,2}\Gamma\paren{\mathring{C},V\otimes\fl}$ induces an isomorphism
    \begin{equation*}
      \mathring{\pi}_! \co (\ker \pi^*\fd)^- \iso \ker \fd^\fl.
    \end{equation*}
    Here (and in the following) $\ker \fd^\fl$ denotes the kernel of the extension \autoref{Eq_RealCauchyRiemannOperator_Twist_BoundedLinear}.
  \end{enumerate}
\end{prop}

\begin{proof}
  \autoref{Prop_RealCauchyRiemannOperator_+-_Equivariant} and \autoref{Prop_RealCauchyRiemannOperator_+-_Isomorphism_+} are immediate from \autoref{Prop_RealCauchyRiemanOperator_Pullback}.

  A moment's thought shows that
  if $s \in (\ker \fd)^-$,
  then its descend $\mathring{\pi}_! s \in W^{1,2}\Gamma(\mathring{C},V\otimes \fl)$ and satisfies $\fd^\fl \mathring{\pi}_! s = 0$;
  therefore, $\mathring{\pi}_! s \in \ker \fd^\fl$.
  Hence, $\mathring{\pi}_!$ induces an inclusion $\mathring{\pi}_! \co (\ker \pi^*\fd)^- \incl \ker \fd^\fl$.

  Let $s \in \ker \fd^\fl$.
  By elliptic regularity,
  $s \in \Gamma(\mathring{C},V\otimes \fl)$;
  indeed, for $B_R(x) \subset \mathring{C}$ with $R \ll 1$
  \begin{equation*}
    \abs{s}^2(x) \lesssim R^{-2} \int_{B_R(x)} \abs{s}^2.
  \end{equation*}
  It follows from \autoref{Def_RamifiedEuclideanLineBundle}~\autoref{Def_RamifiedEuclideanLineBundle_Monodromy} that
  if $B_R(b) \subset C$ with $B_R(b) \cap \Br(\fl) = \set{b}$ and $R \ll 1$,
  then
  \begin{equation*}
    R^{-2} \int_{\del B_R(b)} \abs{s}^2 \lesssim \int_{\del B_R(b)} \abs{\nabla s}^2,
  \end{equation*}
  and, therefore,
  \begin{equation}
    \label{Eq_DerivativeControlsValueAtPuncture}
    \int_{B_R(b)} d(\cdot,b)^{-2}\abs{s}^2 \lesssim \int_{B_R(b)} \abs{\nabla s}^2.
  \end{equation}
  Consequently, $s \in L^\infty\Gamma(\mathring{C},V\otimes \fl)$.

  Evidently, $\mathring{\pi}^*s$ satisfies $(\pi^*\fd)(\mathring{\pi}^*s) = 0$ on $\mathring{\tilde C}$.
  Since $\mathring{\pi}^*s \in L^\infty\Gamma(\tilde C,\pi^*V)$, it also satisfies $(\pi^*\fd)(\mathring{\pi}^*s) = 0$ in the sense of distributions.
  Therefore, by elliptic regularity, $\mathring{\pi}^*s \in \paren{\ker \pi^*\fd}^-$.
  This finishes the proof of \autoref{Prop_RealCauchyRiemannOperator_+-_Isomorphism_-}.
\end{proof}

\begin{remark}
  \label{Rmk_W12=W12_0}
  As a consequence of \autoref{Eq_DerivativeControlsValueAtPuncture},
  $W^{1,2}\Gamma(\mathring{C},V\otimes \fl) = W_0^{1,2}\Gamma(\mathring{C},V\otimes \fl)$.
\end{remark}

To summarise the above discussion,
the following four are equivalent pieces of evidence for the failure of $2$--rigidity of a real Cauchy--Riemann operator $\fd = \fd_J$ on $V$:
\begin{enumerate}
\item
  A closed Riemann surface $\tilde C$, a holomorphic map $\pi \co \tilde C \to C$ of degree at most two, and $s \in \ker \pi^*\fd \setminus \set{0}$.
\item
  A closed Riemann surface $\tilde C$ and a $J$--holomorphic map $u \co \tilde C \to V$ which does not factor through the zero section and with $\pi \circ u$ of degree at most two.
\item
  A $J$--holomorphic cycle $T$ in $V$ whose
  support is not contained in the zero section and with $\deg(T)$ at most two.
\item
  A ramified Euclidean line bundle $\fl$ over $C$ and $s \in \ker \fd^\fl\setminus\set{0}$.
\end{enumerate}


\subsection{$\sW$ is a proper wall}
\label{Sec_TheWallIsProper}

Recall the standing assumption that $C$ is closed and connected.
Henceforth, assume that
\begin{equation*}
  2\deg V + \rk_\C V \cdot \chi(C) = 0.
\end{equation*}
The aim of the upcoming sections is to prove that \defined{wall of failure of $2$--rigidity},
defined by
\begin{equation*}
  \TheWall
  \coloneqq
  \set[\big]{
    \fd \in \SpaceOfRealCauchyRiemannOperators(V)
    :
    \fd \textnormal{ fails to be $2$--rigid}
  },
\end{equation*}
is a proper wall in the sense of \autoref{Def_ProperWall}. 

Of course, $\SpaceOfRealCauchyRiemannOperators(V)$ is not a Banach manifold but an affine space modelled on the Frechet space $\Gamma\paren{C,\HOMT(V)}$;
see \autoref{Prop_RealCauchyRiemanOperator_AffineSpace}.
To rectify this,
choose $K \gg 1$ and set
\begin{equation*}
  C^K\SpaceOfRealCauchyRiemannOperators(V)
  \coloneqq
  \SpaceOfRealCauchyRiemannOperators(V) + C^K\Gamma\paren{C,\Hom_\C(\overline{TC},\C)},
\end{equation*}
and similarly define $C^K\SpaceOfRealCauchyRiemannOperators(V,\SerreOperator)$ and $C^K\SpaceOfHomogeneousAlmostComplexStructures(V)$.
The results discussed so far
continue to hold with these spaces---%
with minimal cosmetic modifications.
To ease notation,
henceforth, the prefix $C^K$ shall be omitted;
in other words, $\SpaceOfRealCauchyRiemannOperators(V)$, $\SpaceOfRealCauchyRiemannOperators(V,\SerreOperator)$, and $\SpaceOfHomogeneousAlmostComplexStructures(V)$ are \emph{redefined} to be the aforementioned spaces.
With this technicality out of the way,
the aim of the remainder of this section is to prove the following.

\begin{theorem}
  \label{Thm_WallOfFailureOf2RigidityIsProper}
  The wall of failure of $2$--rigidity $\TheWall$ is a proper wall in $ \SpaceOfRealCauchyRiemannOperators(V)$.
\end{theorem}

The upcoming discussion also establishes the following variant.

\begin{theorem}
  \label{Thm_WallOfFailureOf2RigidityIsProper+SelfAdjoint}
  If $\SerreOperator \co V \iso V^\dagger$ is an isomorphism,
  then $\TheWall \cap \SpaceOfRealCauchyRiemannOperators(V,\SerreOperator)$ is a proper wall in
  $\SpaceOfRealCauchyRiemannOperators(V,\SerreOperator)$.
\end{theorem}

\subsubsection{$\sW$ is closed}
\label{Sec_TheWallIsClosed}

\begin{prop}
  \label{Prop_TheWallIsClosed}
  The wall of failure of $2$--rigidity $\sW \subset \SpaceOfRealCauchyRiemannOperators(V)$ is closed.
\end{prop}

The proof of \autoref{Prop_TheWallIsClosed} relies on the following.

\begin{theorem}[{\citet[\S 8]{Federer1960}; \cites[\S 4.2.17]{Federer1969}[Theorem 27.3]{Simon1983}{White1989}}]
  \label{Thm_FedererFlemming}
  Let $(X,g)$ be a Riemannian manifold.
  Let $(T_n)$ be a sequence of closed integral $k$--currents.
  If there is a $\Lambda > 0$ such that
  \begin{equation*}
    \bM(T_n) \leq \Lambda,
  \end{equation*}
  then a subsequence of 
  $(T_n)$ converges to a closed integral $k$--current $T_\infty$ in the weak--$*$--topology.
  \qed
\end{theorem}

\begin{proof}[Proof of \autoref{Prop_TheWallIsClosed}]
  Suppose that $(\fd_n) \in \sW^\N$ converges to $\fd_\infty \in \SpaceOfRealCauchyRiemannOperators(V)$.
  For $n \in \N\cup\set{\infty}$ set $J_n \coloneqq J_{\fd_n}$.
  Define $\rho^2 \in C^\infty(V)$ by $\rho^2(v) \coloneqq \abs{v}^2$ (with respect to the chosen Hermitian metric).
  For every $n \in \N$
  choose a $J_n$--holomorphic cycle $T_n$ in $V$ whose support is not contained in the zero section and with $\deg(T_n)$ at most two;
  and, moreover:
   \begin{equation*}
    \max_{\supp T_n} \rho^2 = 1.
  \end{equation*}

  Set $U \coloneqq \set{ v \in V \co \abs{v} < 2 }$.
  Choose a Kähler form $\omega_C$ on $C$.
  For $\Lambda \gg 1$,
  for every $n \in \N \cup\set{\infty}$,
  \begin{equation*}
    \omega_n \coloneqq \Lambda \cdot \omega_C  - \tfrac14 \rd (\rd \rho^2 \circ J_n) \in \Omega^2(U)
  \end{equation*}
  is symplectic and tames $J_n|_U$;
  that is: $\sigma_n \coloneqq \frac12\sqparen{\omega_n + \omega_n(J_n\cdot,J_n\cdot)}$ is a Hermitian form;
  cf.~\cite[Proof of Lemma 3.6~(5)]{Doan2018a}.
  Since $T_n$ is $J_n$--holomorphic,
  $\bM(T_n) = T_n(\sigma_n)$;
  therefore,
  \begin{equation*}
    \bM(T_n) = T_n(\sigma_n) = T_n(\omega_n) \leq 2\Lambda \Inner{[C],[\omega_C]}.
  \end{equation*}

  By \autoref{Thm_FedererFlemming},
  after passing to a subsequence,
  $(T_n)$ converges to a closed integral $2$--current $T_\infty$ in the weak--$*$--topology. 
  Since $J_n \to J_\infty$ (and $\sigma_n \to \sigma_\infty$),
  $T_\infty$ is a $J_\infty$--holomorphic cycle.
  Moreover, $T_\infty$ is of degree
  \begin{equation*}
    \deg(T_\infty) = T_\infty(\omega_C) = \lim_{n \to \infty} T_n(\omega_C) = \lim_{n \to \infty} \deg(T_n)
  \end{equation*}
  at most two.
  By the monotonicity formula (cf.~\cite[Lemma 5.6]{Doan2018a}),
  $(\supp T_n)$ converges to $\supp T_\infty$ in the Hausdorff distance;
  hence: the support of $T_\infty$ is not contained in the zero section.  
  Therefore, $\fd_\infty \in \sW$.
\end{proof}


\subsubsection{Fredholm extensions and index formulae}
\label{Sec_IndexFormulaAndResidues}

This section provides an elementary and self-contained proof of the following result and its higher regularity analogue \autoref{Prop_HigherRegularityFredholm}.
Of course,
these results could be derived from \cite[Theorem 4.1]{Wendl2016} or \cite[Proposition 2.8.6]{Doan2018}.
They are needed in the construction of the Fredholm maps $\pi$, $\nu$, and $\sigma$ in \autoref{Def_ProperWall}.

\begin{theorem}
  \label{Thm_Index}
  If
  $\fd$ is a real Cauchy--Riemann operator on $V$ and  
  $\fl$ is a ramified Euclidean line bundle over $C$,
  then
  $\fd^\fl \co W^{1,2}\Gamma\paren{\mathring{C},V\otimes\fl} \to L^2\Omega^{0,1}\paren{\mathring{C},V\otimes\fl}$ is Fredholm with
  \begin{equation*}
    \ind \fd^\fl = - \# \Br(\fl) \cdot \rk_\C V.
  \end{equation*}
\end{theorem}

\begin{remark}
  \label{Rmk_Index}
  \autoref{Thm_Index} assumes $2\deg V + \rk_\C V \cdot \chi(C) = 0$.
  The operator $\fd^\fl$ continues to be Fredholm if this assumption is dropped,
  but $\ind \fd^\fl = 2\deg V + \rk_\C V \cdot \chi(C) - \# \Br(\fl) \cdot \rk_\C V$; see \autoref{Thm_Index+}.
\end{remark}

\medskip

Throughout the remainder of the section,
assume the situation of \autoref{Thm_Index}.

\begin{prop}
  \label{Prop_LeftSemiFredholm}
  $\fd^\fl$ is left semi-Fredholm; that is:
  $\dim \ker \fd^\fl < \infty$ and $\im \fd^\fl \subset L^2\Omega^{0,1}(C,V \otimes \fl)$ is closed.
\end{prop}

\begin{proof}
  If $\nabla$ is the unitary covariant derivative on $V \otimes \fl$ from \autoref{Prop_RealCauchyRiemannOperator_Twist},
  then, by the Kähler identities,
  \begin{equation*}
    2(\nabla^{0,1})^*\nabla^{0,1}
    =
    \nabla^*\nabla - i\Lambda F_\nabla.
  \end{equation*}
  Therefore,
  \begin{equation*}
    \Abs{\nabla s}^2 = 2\Abs{\nabla^{0,1} s}_{L^2}^2 - \Inner{i\Lambda F_\nabla s,s}_{L^2}
  \end{equation*}
  for every $s \in W_0^{1,2}(\mathring{C},V\otimes \fl)$.
  Since $\fd^\fl = \nabla^{0,1} + \fa$ with $\fa \in L^\infty\Gamma(\mathring{C},\HOMT(V))$,
  \begin{equation*}
    \Abs{s}_{W^{1,2}} \lesssim \Abs{\fd^\fl s}_{L^2} + \Abs{s}_{L^2}.
  \end{equation*}
  
  By Rellich's Theorem and \autoref{Eq_DerivativeControlsValueAtPuncture},
  $W_0^{1,2}\Gamma(\mathring{C},V\otimes \fl) \incl L^2\Gamma(\mathring{C},V\otimes \fl)$ is compact.  
  Therefore, by \cite[Lemma 4.3.9]{BuhlerSalamon2018:FunctionalAnalysis},
  it follows that $\fd^\fl$ is left-semi-Fredholm.
\end{proof}

Consider the extension
\begin{equation*}
  \fd_{L^2}^\fl \co L^2\Gamma(\mathring{C},V \otimes \fl) \to W^{-1,2}\Omega^{0,1}(\mathring{C},V \otimes \fl).
\end{equation*}
Because of \autoref{Rmk_W12=W12_0},
$\fd_{L^2}^{\dagger,\fl}$ is the adjoint of $\fd^\fl$.
In particular,
\begin{equation*}
  \ker \fd_{L^2}^{\dagger,\fl} \iso \paren{\coker \fd^\fl}^*.
\end{equation*}
Therefore, it remains to prove that $\ker \fd_{L^2}^{\dagger,\fl}$ is finite-dimensional, and
to determine $\ind \fd^\fl = \dim \ker \fd^\fl - \dim\ker \fd_{L^2}^{\dagger,\fl}$ (for a judicious choice of $\fd \in \SpaceOfRealCauchyRiemannOperators(V)$).
The key to the former is to understand
\begin{equation*}
  \coker(\ker \fd^\fl \into \ker \fd_{L^2}^\fl).
\end{equation*}
To this end, it is useful to extract the leading order terms at $b \in \Br(\fl)$ of elements of the following superspace of $\ker \fd_{L^2}^\fl$:
\begin{equation*}
  \dom(\fd_{\max}^\fl)
  \coloneqq
  \set[\big]{ s \in L^2\Gamma(\mathring{C},V \otimes \fl) : \fd_{L^2}^\fl s \in L^2\Omega^{0,1}(\mathring{C},V \otimes \fl) }.
\end{equation*}
This is, of course, the domain of the maximal extension of $\fd^\fl$ considered as an unbounded operator $L^2\Gamma(\mathring{C},V \otimes \fl) \to L^2\Omega^{0,1}(\mathring{C},V \otimes \fl)$.

\newcommand{\ResidueSpace}{\fR}
\begin{definition}
  \label{Def_12ResidueSpace}
  Let $b \in \Br(\fl)$.
  Set
  \begin{equation*}
    (\fl\otimes\C)_b \coloneqq \varinjlim_{U \ni b} L^2\rH^0\paren{U\setminus\set{b},\fl\otimes \C}
  \end{equation*}
  with $L^2\rH^0$ denoting $L^2$ holomorphic sections.
  (In the situation discussed in the proof of \autoref{Prop_RiemannSurface_DegreeTwo_+-},
  such a section has an asymptotic expansion
  $a_0 z^{-1/2} + a_1 z^{1/2} + \ldots$ at $0$.)
  $(\fl\otimes\C)_b$ is an $\sO_b$--module.
  The \defined{$\tfrac12$--residue space of $\fl \otimes \C$ at $b$} is
  \begin{equation*}
    \fR_b(\fl) \coloneqq (\fl\otimes\C)_b \otimes_{\sO_b} \paren{\sO_b/\fm_b}.
    \qedhere
  \end{equation*}
\end{definition}

\begin{remark}
  \label{Rmk_12JetSpaceAnd12ResidueSpaceAreDual}
  The pairing $(\fl \otimes\C(-b))_b \otimes_{\sO_b} (\fl \otimes\C)_b \to \sO_b$ induces an isomorphism
  $\fR_b(\fl) \iso \paren{J_b^{1/2}(\fl\otimes\C)}^*$.
\end{remark}

\begin{prop}
  \label{Prop_Residue_Construction}
  Let $b \in \Br(\fl)$.      
  There is a unique linear map
  \begin{equation*}
    \Res_b \co \dom(\fd_{\max}^\fl)
    \to
    \ResidueSpace_b(\fl,V)
    \coloneqq
    \ResidueSpace_b(\fl) \otimes_\C V_b,
  \end{equation*}
  the \defined{residue} at $b$,
  such that the following holds:
  if
  $s \in \dom(\fd_{\max}^\fl)$,
  $\Res_b(s) = [\rho] \otimes_\C v$
  with $\rho \in L^2\rH^0\paren{\mathring U,\fl\otimes\C}$ and $v \in V_b$, and
   $\tau \co U \times V_b \to V|_U$ is an isometry with $\tau_b = \id_{V_b}$,
  then  
  \begin{equation*}
    s - \tau_*v \otimes_\C \rho
    \in
    W^{1,2}\Gamma\paren{\mathring{U}, V\otimes \fl}.
  \end{equation*}
\end{prop}

The proof reduces to the following model case.

\begin{prop}
  \label{Prop_2DModelOperator_LeadingTerm}
  Consider the ramified Euclidean line bundle $\fm$ over $D$ defined in \autoref{Eq_MobiusBundle}.
  Choose $z^{-1/2} \in \Gamma(\mathring{D},\C \otimes \fm)$.
  If $s \in L^2\Gamma(\mathring D,\C^r \otimes \fm)$ satisfies
  $\delbar^\fm s \in L^2\Gamma(\mathring D,\C^r \otimes \fm)$,
  then there is a unique $v \in \C^r$ such that
  \begin{equation*}
    s - v \otimes_\C z^{-1/2} \in W^{1,2}\paren{\tfrac12\mathring{D},\C^r \otimes \fm}.
  \end{equation*}
  Here $\tfrac12\mathring{D} \coloneqq \set{ z \in \C : 0 < \abs{z} < 1/2}$.
\end{prop}

The proof relies on the following elementary observation.

\begin{lemma}
  \label{Lem_1DModelOperator_LeadingTerm}
  Let $\lambda \in \Z + \frac12$.
  Let $\phi \in L^2((0,1),r \rd r)$ with $\paren{\partial_r - \lambda/r}\phi \in L^2((0,1),r \rd r)$.
  \begin{enumerate}
  \item
    If $\lambda = -1/2$,
    then there is a unique $a \in \R$ such that
    $\phi(r) = ar^{-1/2} + o(1)$ as $r \downarrow 0$.
  \item
    If $\lambda \neq -1/2$,
    then $\phi(r) = o(1)$ as $r \downarrow 0$.
  \end{enumerate}
\end{lemma}

\begin{proof}
  Set $\psi \coloneqq (\partial_r-\lambda/r)\phi$.
  There is a unique $a \in \R$ such that
  \begin{equation*}
    \tilde \phi(r)
    \coloneqq \phi(r) - ar^\lambda
    =
    \begin{cases}
      r^\lambda\int_0^r s^{-(\lambda+1)}\psi(s) \,s \rd s & \text{if}~\lambda < 0 \\
      -r^\lambda\int_r^1 s^{-(\lambda+1)}\psi(s) \,s \rd s & \text{if}~\lambda > 0.
    \end{cases}
  \end{equation*}
  
  If $\lambda < 0$,
  then, by Cauchy--Schwarz and monotone convergence,
  \begin{equation*}
    \abs{\tilde\phi(r)}^2 \leq \frac{1}{2\abs{\lambda}} \int_0^r \abs{\psi(s)}^2 \,s\rd s = o(1).
  \end{equation*}
  Moreover,
  if $\lambda \leq -1$,
  then $a = 0$ because $\phi \in L^2((0,1),r\rd r)$.
     
  If $\lambda > 0$, then,
  by Cauchy–Schwarz,
  for $r \leq \epsilon \leq 1$
  \begin{equation*}
    \abs{\tilde\phi(r)}^2
    \leq
    \frac{1}{\lambda}
    \int_0^\epsilon \abs{\psi(s)}^2 \,s\rd s
    +
    \frac{(r/\epsilon)^{2\lambda}}{\lambda} \int_\epsilon^1 \abs{\psi(s)}^2 \,s\rd s
    \eqqcolon \rI(\epsilon) + \rII(r,\epsilon).
  \end{equation*}
  By monotone convergence,
  $\lim_{\epsilon \downarrow 0} \rI(\epsilon) = 0$.
  Evidently,
  $\lim_{r \downarrow 0} \rII(r,\epsilon) = 0$.
  Therefore, $\tilde\phi(r) = o(1)$ as $r \downarrow 0$.
\end{proof}

\begin{proof}[Proof of \autoref{Prop_2DModelOperator_LeadingTerm}]
  In polar coordinates $z = re^{i\alpha}$,
  \begin{equation*}
    \delbar^\fm f = (\del_r + ir^{-1}\cdot\nabla_\alpha^\fm) f \cdot \tfrac12(\rd r - ir\cdot \rd \alpha).
  \end{equation*}  
  By Fubini's theorem,
  \begin{equation*}
    L^2\Gamma(\mathring D,\C^r \otimes \fm) = L^2((0,1),r\rd r;L^2\Gamma(S^1,\C^r \otimes \fm)).
  \end{equation*}
  By Fourier analysis,
    \begin{equation*}
    s = \sum_{\lambda \in \Z + \frac12} s_\lambda(r) e^{i\lambda\alpha}
  \end{equation*}
  with $s_\lambda, (\del_r + \lambda/r)s_\lambda \in L^2\paren{(0,1),r\rd r;\C^r}$.
  
  By \autoref{Lem_1DModelOperator_LeadingTerm},
  \begin{equation*}
    s_{-1/2} = v r^{-1/2} + o(1),
  \end{equation*}
  and $s_\lambda = o(1)$ for every $\lambda \neq -\frac12$.  
  Since $v \otimes_\C z^{-1/2} \in \ker \delbar^\fm$,
  $s$ may be replaced by $s-v\otimes_\C z^{-1/2}$.
  Therefore,
  without loss of generality $s_\lambda(0) = s_\lambda(1) = 0$ for every $\lambda \in \Z + \frac12$.

  It remains to prove that $s \in W^{1,2}\Gamma\paren{\mathring{D},\C^r \otimes \fm}$.
  By direct computation,
  \begin{align*}
    2\int_{\mathring D} \abs{\delbar^\fm s}^2 \,r \rd r \rd \alpha
    &=
      \sum_{\lambda \in \Z + \frac12} \int_0^1
      \abs{(\del_r+\tfrac{\lambda}{r}) s_\lambda}^2
      r \rd r \\
    &=
      \sum_{\lambda \in \Z + \frac12} \int_0^1
      \paren[\big]{
      \abs{\del_r s_\lambda}^2 r
      +
      \tfrac{\lambda^2}{r} \abs{s_\lambda}^2
      +
      \lambda \del_r\abs{s_\lambda}^2
      }
      \rd r \\
    &=
      \sum_{\lambda \in \Z+\frac12} \int_0^1
      \paren{\abs{\del_r s}^2 + \abs{\nabla_\alpha^\fm s}^2 } \, r\rd r \\
    &=
      \int_{\mathring D} \abs{\nabla^\fm s}^2 \,r \rd r \rd \alpha.
      \qedhere
  \end{align*}
\end{proof}

\begin{proof}[Proof of \autoref{Prop_Residue_Construction}]
  Since $\fd = \delbar + \fa$ with $\abs{\fa} \in L^\infty(C)$,
  the assertion follows from \autoref{Prop_2DModelOperator_LeadingTerm}.
\end{proof}

The maps $\Res_b$ constructed in \autoref{Prop_Residue_Construction} assemble into
\begin{equation*}
  \Res\co \mathrm{dom}(\fd_{\max}^\fl) \to \ResidueSpace(\fl,V) \coloneqq \bigoplus_{b \in B}. \ResidueSpace_b(\fl,V)
\end{equation*}
The following is immediate from the construction.

\begin{cor}
  \label{Cor_Residue_Kernel}
  The sequence
  \begin{equation*}
    W^{1,2}\Gamma\paren{\mathring{C}, V\otimes\fl}
    \incl    
    \dom(\fd_{L^2}^\fl)
    \stackrel{\Res}{\onto}
    \ResidueSpace(\fl,V)
  \end{equation*}
  is exact.
  In particular,
  \begin{equation*}
    \ker \fd^\fl = \ker \paren{\Res \co \ker \fd_{L^2}^\fl \to \ResidueSpace(\fl,V)}.
    \qedhere
  \end{equation*}
\end{cor}

\begin{proof}[Proof of \autoref{Thm_Index}: Fredholm property]
  By \autoref{Prop_LeftSemiFredholm} and the discussion following its proof,
  to prove that $\fd^\fl$ is Fredholm
  it remains to establish that $\ker \fd_{L^2}^{\dagger,\fl}$ is finite-dimensional.
  This, however, is a consequence of \autoref{Cor_Residue_Kernel} since $\ResidueSpace(\fl,V)$ and $\ker \fd^{\dagger,\fl}$ are finite-dimensional.
\end{proof}

The proof of the index formula assuming $V \iso V^\dagger$ relies on the following.

\begin{prop}
  \label{Thm_Index_SelfAdjoint}
  Let $\SerreOperator \co V \iso V^\dagger$ be an isomorphism.
  \begin{enumerate}
  \item
    \label{Thm_Index_SelfAdjoint_Symplectic}
    There is a non-degenerate alternating form $G \in \Hom(\Wedge^2\ResidueSpace(\fl,V),\R)$ on $\ResidueSpace(\fl,V)$,
    the \defined{Green's form},
    such that
    \begin{equation*}
      \int_C \Re\Inner{\paren{\SerreOperator\otimes\one}\circ \fd^\fl s,t} - \Re\Inner{s,\paren{\SerreOperator\otimes\one}\circ \fd^\fl t}
      =
      G(\Res(s),\Res(t))
    \end{equation*}
    for every $s,t \in \dom(\fd_{L^2}^\fl)$.
  \item
    \label{Thm_Index_SelfAdjoint_Lagrangian}
    The subspace $\im \paren{\Res \co \ker \fd_{L^2}^\fl \to \ResidueSpace(\fl,V)} \subset \ResidueSpace(\fl,V)$ is Lagrangian with respect to $G$.
  \end{enumerate}

\end{prop}

\begin{proof}
  The unbounded operator 
  $\paren{\SerreOperator\otimes\one}\circ \fd^\fl \co W^{1,2}\Gamma\paren{\mathring{C}, V\otimes\fl} \subset L^2\Gamma\paren{\mathring{C}, V\otimes\fl} \to L^2\Omega^{0,1}\paren{\mathring{C}, V^\dagger\otimes\fl}$
  is symmetric (with respect to the perfect pairings induced by \autoref{Eq_PerfectPairing}) and closed.
  Its maximal extension is
  $\paren{\SerreOperator\otimes\one}\circ \fd^\fl \co \dom(\fd_{L^2}^\fl) \subset L^2\Gamma\paren{\mathring{C}, V\otimes\fl} \to L^2\Omega^{0,1}\paren{\mathring{C}, V^\dagger\otimes\fl}$.

  By \cite[Exercise 2.17]{McDuff1998},
  the \defined{Green's form} $\hat G \in \Hom(\Lambda^2\dom(\fd_{L^2}^\fl),\R)$ defined by
  \begin{equation*}
    \hat G(s,t) \coloneqq \int_C \Re\Inner{\paren{\SerreOperator\otimes\one}\circ \fd^\fl s,t} - \Re\Inner{s,\paren{\SerreOperator\otimes\one}\circ \fd^\fl t}
  \end{equation*}
  descends to a symplectic form on the \defined{Gelfand--Robbin quotient}
  \begin{equation*}
    {\dom(\fd_{L^2}^\fl)} / {W^{1,2}\Gamma\paren{\mathring{C}, V\otimes\fl}}
    \stackrel{\Res}{\iso} \ResidueSpace(\fl,V).
  \end{equation*}
  This proves \autoref{Thm_Index_SelfAdjoint_Symplectic}.

  It immedately follows that $\im \paren{\Res \co \ker \fd_{L^2}^\fl \to \ResidueSpace(\fl,V)}$ is isotropic.  
  To prove that it is coisotropic,
  let $v \in \ResidueSpace(\fl,V)$ such that $G(v,w) = 0$
  for every $w \in \im \paren{\Res \co \ker \fd_{L^2}^\fl \to \ResidueSpace(\fl,V)}$.
  Construct $s_0 \in \dom(\fd_{L^2}^\fl)$ with $\Res(s_0) = v$.
  Since
  \begin{equation*}
    \hat G(s_0,t) = \int_C \Re\Inner{\paren{\SerreOperator\otimes\one}\circ \fd^\fl s_0,t} = 0
  \end{equation*}
  for every $t \in \ker \fd_{L^2}^\fl$,
  $\fd_{L^2}^\fl s_0 \in \im \fd^\fl$.
  Choose $s_1 \in W_0^{1,2}\Gamma(\mathring{C},V\otimes\fl)$ with $\fd^\fl s_1 = -\fd_{L^2}^\fl s_0$.
  Finally,
  $s \coloneqq s_0+s_1 \in \ker \fd_{L^2}^\fl$ and $\Res(s) = \Res(s_0) = v$.
  This proves \autoref{Thm_Index_SelfAdjoint_Lagrangian}.  
\end{proof}

\begin{proof}[Proof of \autoref{Thm_Index}: index formula]
  Choose $\SerreOperator \co V \iso V^\dagger$.
  By \autoref{Prop_RealCauchyRiemanOperator_AffineSpace} and \autoref{Prop_RealCauchyRiemanOperator_SelfAdjoint_AffineSpace},
  it suffices to prove the index formula for $\fd \in \SpaceOfRealCauchyRiemannOperators(V,\SerreOperator)$.
  In this case,
  \begin{align*}
    \ind \fd^\fl
    &=
      \dim\ker\fd^\fl - \dim\ker\fd_{L^2}^\fl \\
    &=
      -\dim \im \paren{\Res \co \ker \fd_{L^2}^\fl \to \ResidueSpace(\fl,V)} = -\frac12 \dim\ResidueSpace(\fl,V) = -\#\Br(\fl) \cdot \rk_\C V
  \end{align*}
  by \autoref{Thm_Index_SelfAdjoint}.
\end{proof}

For parts of the discussion in \autoref{Sec_ConstructionOfSigma}
a higher regularity version of \autoref{Thm_Index} is necessary.
It turns out that the correct generalisation of $W^{1,2}$ is not $W^{k,2}$ but instead the following variant.

\begin{definition}
  \label{Def_WK2Cyl}
  Denote by $r \co \mathring{C} \to (0,\infty)$ the distance to $\Br(\fl)$.
  Set $g_{\cyl} \coloneqq r^{-2}g$.
  Let $k \in \N$.
  For $s \in W_{\loc}^{k,2}\Gamma\paren{\mathring{C},V \otimes \fl}$ set
  \begin{equation*}
    \Abs{s}_{W_\cyl^{k,2}} \coloneqq \paren*{\sum_{\ell=0}^k \int_{\mathring C} \abs{\nabla^\ell s}_{g_\cyl}^2 \vol_{g_\cyl}}^{1/2}.
  \end{equation*}
  This defines a complete norm on
  \begin{equation*}
    W_\cyl^{k,2}\Gamma\paren{\mathring{C},V \otimes \fl}
    \coloneqq
    \set*{
      s \in W_{\loc}^{k,2}\Gamma\paren{\mathring{C},V \otimes \fl}
      :
      \Abs{s}_{W_\cyl^{k,2}} < \infty
    }.
  \end{equation*}
  Similarly, define $W_\cyl^{k,2}\Omega^{0,1}\paren{\mathring{C},V \otimes \fl}$ using the Euclidean inner product on $\overline{K}_C$ induced by $g_\cyl$.
  \qedhere
\end{definition}

\begin{remark}
  \label{Rmk_GCyl}
  If $r = e^{-t}$, then
  \begin{equation*}
    \rd r^2 + r^2 \rd \alpha^2 = r^2(\rd t^2 + \rd \alpha^2).
  \end{equation*}
  Therefore, with respect to $g_\cyl$ every puncture $b \in \Br(\fl)$ corresponds to an asymptotically cylindrical end of $\mathring{C}$
\end{remark}

The following is an immediate consequence of \autoref{Eq_DerivativeControlsValueAtPuncture} and conformal invariance.

\begin{prop}
  \label{Prop_W12=W12Cyl}
  $W^{1,2}\Gamma\paren{\mathring{C},V \otimes \fl} = W_\cyl^{1,2}\paren{\mathring{C},V \otimes \fl}$
  and
  $L^2\Omega^{0,1}\paren{\mathring{C},V \otimes \fl} = L_\cyl^2\Omega^{0,1}\paren{\mathring{C},V \otimes \fl}$.  
  \qed
\end{prop}

The bounded operator $\fd^\fl \co W^{1,2}\Gamma\paren{\mathring{C},V\otimes\fl} \to L^2\Omega^{0,1}\paren{\mathring{C},V\otimes\fl}$ (co)restricts to a bounded operator
\begin{equation*}
  \fd_{W_\cyl^{k+1,2}}^\fl
  \co
  W_\cyl^{k+1,2}\Gamma\paren{\mathring{C},V \otimes \fl}
  \to
  W_\cyl^{k,2}\paren{\mathring{C},V \otimes \fl}.
\end{equation*}

\begin{prop}
  \label{Prop_Regularity}
  For every $k \in \set{0,\ldots,K}$
  if $s \in L^2\Gamma\paren{\mathring{C}, V\otimes \fl}$,
  $\fd^\fl s \in W_\cyl^{k,2}\Gamma\paren{\mathring{C}, V\otimes \fl}$, and
  $\Res(s) = 0$,
  then
  $s \in W_\cyl^{k+1,2}\Gamma\paren{\mathring{C}, V\otimes \fl}$ and
  \begin{equation*}
    \Abs{s}_{W_\cyl^{k+1,2}}
    \lesssim_k
    \Abs{\fd^\fl s}_{W_\cyl^{k,2}}
    +
    \Abs{s}_{L^2}.
  \end{equation*}
\end{prop}

\begin{proof}
  For $k = 0$,
  this is a consequence of \autoref{Prop_Residue_Construction};
  in particular:
  \begin{equation*}
    \Abs{s}_{L_\cyl^2} \lesssim \Abs{\fd^\fl s}_{L_\cyl^2} + \Abs{s}_{L^2}.
  \end{equation*}  
  By interior elliptic regularity and estimates,
  if $s \in L_\cyl^2\Gamma\paren{\mathring{C}, V\otimes \fl}$ and
  $\fd^\fl s \in W_\cyl^{k,2}\Omega^{0,1}\paren{\mathring{C}, V\otimes \fl}$,
  then $s \in W_\cyl^{k+1,2}\Gamma\paren{\mathring{C}, V\otimes \fl}$ and  
  \begin{equation*}
    \Abs{s}_{W_\cyl^{k+1,2}}
    \lesssim_k
    \Abs{\fd^\fl s}_{W_\cyl^{k,2}}
    +
    \Abs{s}_{L_\cyl^2}.
  \end{equation*}
  This proves the assertion.
\end{proof}

\begin{theorem}
  \label{Prop_HigherRegularityFredholm}
  Let $k \in \set{0,\ldots,K}$.
  If
  $\fd$ is a real Cauchy--Riemann operator on $V$ and
  $\fl$ is a ramified Euclidean line bundle over $C$,
  then
  $\fd_{W_\cyl^{k+1,2}}^\fl$ is Fredholm with
  \begin{equation*}
    \ind \fd_{W_\cyl^{k+1,2}}^\fl = - \# \Br(\fl) \cdot \rk_\C V;
  \end{equation*}
  moreover,
  \begin{equation*}
    \ker \fd_{W_\cyl^{k+1,2}}^\fl= \ker \fd^\fl
  \end{equation*}
  and the map
  \begin{equation*}
    \coker \fd_{W_\cyl^{k+1,2}}^\fl \to \coker \fd^\fl \iso \paren{\ker \fd_{L^2}^{\dagger,\fl}}^*
  \end{equation*}
  is an isomorphism.
\end{theorem}

\begin{proof}
  By \autoref{Prop_Regularity},
  $\ker \fd_{W_\cyl^{k+1,2}}^\fl= \ker \fd^\fl$;
  moreover:
  the linear map
  \begin{equation*}
    \frac{W^{1,2}\Gamma\paren{\mathring{C},V\otimes \fl}}{W_\cyl^{k+1,2}\Gamma\paren{\mathring{C},V\otimes \fl}}
    \to
    \frac{L^2\Omega^{0,1}\paren{\mathring{C},V\otimes \fl}}{W_\cyl^{k,2}\Omega^{0,1}\paren{\mathring{C},V\otimes \fl}}    
  \end{equation*}
  induced by $\fd^\fl$ is injective.
  Therefore, by the Snake Lemma,
  the canonical map
  \begin{equation*}
    \coker \fd_{W_\cyl^{k+1}}^\fl \to \coker \fd^\fl
  \end{equation*}
  is injective.
  
  Since $W_\cyl^{k,2}\Gamma\paren{\mathring{C}, V \otimes \fl}$ is dense in $L^2\Gamma\paren{\mathring{C}, V \otimes \fl}$,
  the map
  \begin{equation*}
    W_\cyl^{k,2}\Gamma\paren{\mathring{C}, V \otimes \fl} \to \paren{\ker \fd_{L^2}^{\dagger,\fl}}^*
  \end{equation*}
  is surjective.
  Since it factors through $\coker \fd_{W_\cyl^{k+1,2}}^\fl \to \coker \fd^\fl$,
  the latter must be surjective.  
\end{proof}


\subsubsection{The space of ramified Euclidean line bundles}
\label{Sec_SpaceOfRamifiedEuclideanLineBundles}

The following preparation is needed in the construction of the Banach manifolds $\sE$, $\sN$, and $\sS$ which appear in \autoref{Def_ProperWall}.

\begin{definition}
  \label{Def_SpaceOfRamifiedEuclideanLineBundles}
  The \defined{space of ramified Euclidean line bundles} over $C$ is constructed as follows:
  \begin{enumerate}
  \item
    \label{Def_SpaceOfRamifiedEuclideanLineBundles_ConfigurationSpace}
    The \defined{configuration space} is the smooth manifold
    \begin{equation*}
      \ConfigurationSpace(C) \coloneqq \coprod_{k \in \N_0} \Conf_k(C)
      \qwithq
      \ConfigurationSpace_k(C) \coloneqq (C^k\setminus \Delta_k)/S_k
    \end{equation*}
    and $\Delta_k \coloneqq \set{ (x_1,\ldots,x_k) \in C^k : \#\set{x_1,\ldots,x_k} < k }$.
    Identify $[x_1,\ldots,x_k] \in \ConfigurationSpace(C)$ with $\set{x_1,\ldots,x_k}$, and
    regard $\ConfigurationSpace(C)$ as the space of finite subsets of $C$.
  \item
    \label{Def_SpaceOfRamifiedEuclideanLineBundles_Set}
    Denote by
    $\SpaceOfRamifiedEuclideanLineBundles(C)$
    the \emph{set} of equivalence classes of ramified Euclidean line bundles over $C$.
    Consider the map $\Br \co \SpaceOfRamifiedEuclideanLineBundles(C) \to \ConfigurationSpace(C)$.
  \item
    \label{Def_SpaceOfRamifiedEuclideanLineBundles_Charts}
    For every open subset $U \subset C$,
    denote by $\sU \subset \ConfigurationSpace(C)$ the open subset of those $B \subset U$ such that $\rH^1(C\setminus U,\Z/2\Z) \to \rH^1(C\setminus B,\Z/2\Z)$ is an isomorphism.
    The monodromy representation defines an injection $\tau_U \co \Br^{-1}(\sU) \incl \sU \times \rH^1(C\setminus U,\Z/2\Z)$;
    cf.~\autoref{Rmk_EuclideanLineBundle_Monodromy}.
  \item
    \label{Def_SpaceOfRamifiedEuclideanLineBundles_Topology}
    Equip $\SpaceOfRamifiedEuclideanLineBundles(C)$ with the coarsest \emph{topology} with respect to which the maps $\tau_U$ are continuous.
    \qedhere
  \end{enumerate}
\end{definition}

\begin{prop}
  \label{Prop_Br_CoveringMap}
  The map $\Br \co \SpaceOfRamifiedEuclideanLineBundles(C) \to \ConfigurationSpace(C)$ is a covering map with finite fibers.
  In particular, $\SpaceOfRamifiedEuclideanLineBundles(C)$ inherits the structure of a smooth manifold from $\ConfigurationSpace(C)$.
\end{prop}

\begin{proof}
  To prove that $\Br$ is a covering map it suffices to verify that the transition maps $\tau_U^{-1} \circ \tau_V \co (\sU \cap \sV) \times \rH^1(C\setminus U,\Z/2\Z) \to (\sU\cap \sV) \times \rH^1(C\setminus V,\Z/2\Z)$ are continuous.
  This is an easy exercise;
  cf.~\cite[Proof of (3.3.2) Theorem]{tomDieck2008:AlgebraicTopology}.

  Since $\Br^{-1}(B)$ is in bijection with a subset the finite set $\rH_1(C\setminus B,\Z/2\Z)$, $\Br$ has finite fibers.
\end{proof}

\begin{prop}
  \label{Prop_TwistedUniversalRamifiedEuclideanLineBundle}
  $\SpaceOfRamifiedEuclideanLineBundles(C)$ carries a \defined{twisted universal ramified Euclidean line bundle} in the following sense:
  \begin{enumerate}
  \item
    \label{Prop_TwistedUniversalRamifiedEuclideanLineBundle_Curve}
    Consider the \defined{universal punctured curve}
    \begin{equation*}
      \mathring{\sC} \coloneqq \set{ (B,x) \in \Conf(C) \times C : x \notin B }.
    \end{equation*}
    The projection map
    $P \co \mathring\sC \to \ConfigurationSpace(C)$
    is a fibre bundle.
  \item
    \label{Prop_TwistedUniversalRamifiedEuclideanLineBundle_Bundle}       
    Denote by $\set{ \sU_i : i \in I }$ the set of open subsets $\sU \subset \ConfigurationSpace(C)$ such that $P|_{\sU} \co \mathring{\sC}|_\sU \to \sU$ is trivial.
    Set $\sV_i \coloneqq \Br^{-1}(\sU)$.
    Consider $\Br^*P \co \Br^*\mathring{\sC} \to \SpaceOfRamifiedEuclideanLineBundles(C)$.

    For every $i \in I$ there are a Euclidean line bundle $\fL_i$ over $\Br^*\mathring{\sC}|_{\sV_i}$ and for every ramified Euclidean local system $\fl$ with $\Br(\fl) \in \sU_i$ an isomorphism
    \begin{equation*}
      \fL_i|_{[\fl] \times \paren{C\setminus\Br(\fl)}} \iso \fl
    \end{equation*}
    unique up to $\Aut(\fl) = \set{\pm 1}$.
  \item
    \label{Prop_TwistedUniversalRamifiedEuclideanLineBundle_Cocycle}
    There are a Čech $2$--cocycle $\lambda \in \check\rC^2(\set{\sV_i : i \in I},\underline{\set{\pm 1}})$ 
    and, for every $i,j \in I$, an isomorphism
    $\phi_j^i \co \fL_i|_{\Br^*\mathring{\sC}|_{\sV_i \cap \sV_j}} \iso \fL_j|_{\Br^*\mathring{\sC}|_{\sV_i \cap \sV_j}}$
    such that
    \begin{equation*}
      (\Br^*P)^*\lambda_{ijk} \coloneqq \phi_i^k \phi_k^j \phi_j^i
      \in C^0(\Br^*\mathring{\sC}|_{\sV_i \cap \sV_j \cap \sV_k},\set{\pm 1}).
    \end{equation*}
  \end{enumerate}  
\end{prop}

\begin{proof}
  \autoref{Prop_TwistedUniversalRamifiedEuclideanLineBundle_Curve} is an easy exercise.

  \autoref{Prop_TwistedUniversalRamifiedEuclideanLineBundle_Bundle} is trivial,
  and so is the existence of the isomorphisms in \autoref{Prop_TwistedUniversalRamifiedEuclideanLineBundle_Cocycle}.
  A priori, these define a cocycle
  $\tilde \lambda \in \check\rC^2(\set{\Br^*\mathring{\sC}|_{\sV_i} : i \in I},\underline{\set{\pm 1}})$.
  However,
  since $\Br^*P$ has connected fibres, $\tilde\lambda$ descends to $\lambda$.
\end{proof}

\begin{remark}
  \label{Rmk_TwistedTwistedUniversal}
  The use of twisted universal objects goes back (at least to) \citet{Caldararu2000:PhDThesis}.
\end{remark}

Of course, $[\lambda] \in \check\rH^2\paren{\SpaceOfRamifiedEuclideanLineBundles(C),\underline{\set{\pm 1}}}$ is the obstruction to $\set{ \fL_i : i \in I }$ gluing to a \defined{universal ramified Euclidean line bundle} $\fL$ over $\Br^*\mathring{\sC}$.
If $\fL$ \emph{did exists},
then it would be straight-forward to construct
Hilbert space bundles $\bE$, $\bF$ over  $\SpaceOfRamifiedEuclideanLineBundles(C)$ such that
\begin{equation*}  
  {\bE}_{[\fl]} \iso \bE_{\fl} \coloneqq W^{1,2}\Gamma(\mathring{C}, V\otimes_\R \fl)
  \qandq
  {\bF}_{[\fl]} \iso \bF_{\fl} \coloneqq L^2\Omega^{0,1}(\mathring{C}, V \otimes_\R \fl).
\end{equation*}
These in turn would give rise to
the relative Grassmann bundle $\Gr_d(\bE)$ of $d$--planes in $\bE_{[\fl]}$,
a Banach space bundle $\sL(\bE,\bF)$ of bounded linear maps from $\bE_{[\fl]}$ to $\bF_{[\fl]}$,
etc.
Although $\bE$, $\bF$ \emph{do not exist} in an untwisted sense,
various objects derived from $\bE$ and $\bF$ do.
Indeed,
the construction of these objects over $\sV_i$ as in \autoref{Prop_TwistedUniversalRamifiedEuclideanLineBundle} is straight-forward.
The induced $2$--cocycles measuring the obstruction to gluing are induced by $\lambda$.
Indeed,
by the nature of the objects to be constructed, these cocycles are obtained from $\lambda$ via the trivial homomorphism $(\cdot)^2 \co \set{\pm 1} \to \set{\pm 1}$.
Therefore, the obstruction vanishes.
An alternative way to understand this is that the automorphism group $\Aut(\fl) = \set{\pm 1}$ acts trivially on $\Gr_d(\bE_\fl)$, $\sL(\bE_\fl,\bF_\fl)$, etc.


\subsubsection{Failure of $2$--rigidity: construction of $\pi$}
\label{Sec_ConstructionOfPi}

\newcommand{\TR}{\mathrm{TwistRes}}
\begin{prop}
  \label{Prop_Wall_OneLocalSystem}
  Let $d \in \set{1,2}$.
  Consider the Grassmann bundle $\Gr_d(\bE)$ over $\SpaceOfRamifiedEuclideanLineBundles(C)$.
  The following hold:
  \begin{enumerate}
  \item
    \label{Prop_Wall_OneLocalSystem_Bundle}
    There are a Hilbert space bundle $\Hom(\sS,\bF) \to \SpaceOfRealCauchyRiemannOperators(V) \times \Gr_d(\bE)$ and
    for every $(\fd;[\fl,\Lambda]) \in \SpaceOfRealCauchyRiemannOperators(V) \times \Gr_d(\bE)$ canonical isomorphisms
    \begin{equation*}
      \eta_{(\fd;\fl,\Lambda)} \co \Hom(\sS,\bF)_{(\fd;[\fl,\Lambda])} \iso \Hom\paren{\Lambda,\bF_\fl}.
    \end{equation*}
  \item
    \label{Prop_Wall_OneLocalSystem_Section}
    $\Hom(\sS,\bF) \to \SpaceOfRealCauchyRiemannOperators(V) \times \Gr_d(\bE)$ has a smooth section $\TR_d$ such that
    \begin{equation*}
      \TR_d(\fd;[\fl,\Lambda])
      =
      \eta_{(\fd;\fl,\Lambda)}^{-1}\paren{\fd^\fl|_\Lambda}.
    \end{equation*}
  \item
    \label{Prop_Wall_OneLocalSystem_Property}
    $\TR_d(\fd;[\fl,\Lambda]) = 0$ if and only if $\Lambda \subset \ker \fd^\fl$.
  \item
    \label{Prop_Wall_OneLocalSystem_Transverse}
    $\TR_d$ intersects the zero section transversely;
    in particular,
    \begin{equation*}
      \sE_d \coloneqq \TR_d^{-1}(0)
    \end{equation*}
    is a Banach submanifold.
  \item
    \label{Prop_Wall_OneLocalSystem_Dimension}
    The projection map $\pi_d \co \sE_d \to \SpaceOfRealCauchyRiemannOperators(V)$ is Fredholm of index
    \begin{equation*}
      \paren{2 - d \cdot \rk_\C V} \cdot \#\Br(\fl) - d^2
    \end{equation*}
    at $(\fd;[\fl,\Lambda])$.
  \end{enumerate}
\end{prop}

The extension of \autoref{Prop_Wall_OneLocalSystem} to $d \in \N$ does hold,
but the above is sufficient for the purpose of this article.
The restriction to $d \in \set{1,2}$ allows for a drastic simplification of the proof:
it can be based on the following observation instead of the the subtle results on Petri's condition obtained in \cite[Section 5]{Wendl2016}.

\begin{prop}[{cf.~\cite[Proof of Lemma 4.4]{Eftekhary2016}}]
  \label{Prop_NonParallel_Global=>AlmostEverywhere}
  Let $\fd$ be a real Cauchy--Riemann operator on $V$ over $C$.
  If $s_1,s_2 \in \ker \fd$ are linearly-independent,
  then
  \begin{equation*}
    U \coloneqq \set{ x \in C: s_1(x),s_2(x) ~\textnormal{are linearly-independent} }
  \end{equation*}
  is open and dense.
\end{prop}

\begin{proof}
  Since $C \setminus U \coloneqq \set{x \in C : \dim \R\Span{s_1(x),s_2(x)} \leq 1 }$ is closed, $U$ is open.

  To prove that $U$ is dense,
  assume (by contradiction) that $C \setminus U$ contains a non-empty open subset $O$.
  By unique continuation, neither $s_1$ nor $s_2$ can vanish on $O$.
  Without loss of generality, $s_1(x) \neq 0$ for every $x \in O$.
  Therefore and since $O \subset C \setminus U$
  $s_2 = f \cdot s_1$.
  Since $\fd s_1 = \fd s_2 = 0$,
  $f$ is holomorphic.
  In fact, $f$ is constant because it is $\R$--valued.
  By unique continuation,
  $s_2-fs_1$ vanishes on $C$---a contradiction.  
\end{proof}

\begin{proof}[Proof of \autoref{Prop_Wall_OneLocalSystem}]
  \autoref{Prop_Wall_OneLocalSystem_Bundle} and \autoref{Prop_Wall_OneLocalSystem_Section} are evident from the discussion in \autoref{Sec_SpaceOfRamifiedEuclideanLineBundles}.  
  \autoref{Prop_Wall_OneLocalSystem_Property} holds by construction.

  \medskip

  \autoref{Prop_Wall_OneLocalSystem_Transverse} asserts that for every  $(\fd;[\fl,\Lambda]) \in \TR_d^{-1}(0)$ the composition
  \begin{equation*}
    T_\fd\SpaceOfRealCauchyRiemannOperators(V) \oplus
    T_{[\fl,\Lambda]}\Gr_d(\bE)
    \xrightarrow{T_{(\fd;[\fl,\Lambda])}\TR_d}
    T_{(\fd;[\fl,\Lambda];0)}\Hom(\sS,\bF)
    \onto
    \Hom(\Lambda,\bF_\fl)
  \end{equation*}
  is surjective.
  By \autoref{Prop_RealCauchyRiemanOperator_AffineSpace},
  $T_\fd\SpaceOfRealCauchyRiemannOperators(V) = C^K\Gamma(C,\HOMT(V))$.
  Furthermore,
  $T_{[\fl,\Lambda]}\Gr_d(\bE)$ fits into the 
  short exact sequence
  \begin{equation*}
    \Hom(\Lambda,\bE_\fl/\Lambda) \iso T_\Lambda\Gr_d(\bE)_{\fl} \into T_{[\fl,\Lambda]}\Gr_d(\bE) \onto T_{\Br(\fl)}\Conf(C).
  \end{equation*}
  A moment's thought shows that the above composition restricts to the map
  $\bL \co C^K\Gamma(C,\HOMT(V)) \oplus \Hom(\Lambda,\bE_{\fl}/\Lambda) \to \Hom(\Lambda,\bF_\fl)$ 
  defined by
  \begin{equation*}
    \bL(\fa,M)s \coloneqq (\fa \otimes \id_\fl)s + \fd^\fl(Ms).
  \end{equation*}
  Therefore, it suffices to show that $\bL$ is surjective.

  Since $\coker \fd^\fl \iso \paren{\ker \fd_{L^2}^{\dagger,\fl}}^*$,
  $\bL$ is surjective if and only if the map
  $\bM \co C^K\Gamma(C,\HOMT(V)) \to \paren{\Lambda \otimes \ker \fd_{L^2}^{\dagger,\fl}}^*$ defined by
  \begin{equation*}
    \bM(\fa)(s \otimes t)
    \coloneqq
    \Inner{t,(\fa\otimes \id_\fl) s}
  \end{equation*}
  is surjective.
  $\bM$ is surjective if and only if the \defined{Petri map}
  $\varpi \co \Lambda \otimes \ker \fd_{L^2}^{\dagger,\fl} \to L^1\Gamma(\mathring{C}, V \otimes V^\dagger)$
  defined by
  \begin{equation*}
    \varpi\paren{s \otimes t}(x)
    \coloneqq
    s(x) \otimes t(x)
  \end{equation*}
  is injective.
  Since $d \in \set{1,2}$,
  by \autoref{Prop_NonParallel_Global=>AlmostEverywhere},
  $\varpi$ is injective.
  This proves \autoref{Prop_Wall_OneLocalSystem_Transverse}.

  \medskip
  
  To prove \autoref{Prop_Wall_OneLocalSystem_Dimension},
  consider the projection $\overline{\pi}_d \co \sE_d \to \SpaceOfRealCauchyRiemannOperators(V) \times \ConfigurationSpace(C)$.
  Evidently, $\pi_d$ is Fredholm if and only if $\overline{\pi}_d$ is, and
  \begin{equation*}
    \ind T_{(\fd;[\fl,\Lambda])}\pi_d
    = \ind T_{(\fd;[\fl,\Lambda])}\overline{\pi}_d + \dim T_{\Br(\fl)}\ConfigurationSpace(C)
    = \ind T_{(\fd;[\fl,\Lambda])}\overline{\pi}_d + 2\#\Br(\fl).
  \end{equation*}
  The Snake Lemma applied to 
  \begin{equation*}
    \begin{tikzcd}
      T_{(\fd;[\fl,\Lambda])}\sE_d \ar[hook]{r} \ar{d}{T_{(\fd;[\fl,\Lambda])}\overline\pi} &
      C^K\Gamma(C,\HOMT(V)) \oplus T_{[\fl,\Lambda]}\Gr_d(\bE) \ar[two heads]{r} \ar[two heads]{d} & \Hom(\Lambda,\bF_\fl) \\
      C^K\Gamma\paren{C,\HOMT(V)} \oplus T_{\Br(\fl)}\ConfigurationSpace(C) \ar[equals]{r} &
      C^K\Gamma(C,\HOMT(V)) \oplus T_{\Br(\fl)}\ConfigurationSpace(C)
    \end{tikzcd}
  \end{equation*}
  yields an exact sequence
  \begin{equation*}
    \ker T_{(\fd;[\fl,\Lambda])}\overline{\pi}_d \into \Hom(\Lambda,\bE_\fl/\Lambda) \xrightarrow{\fd^\fl \circ \cdot} \Hom(\Lambda,\bF_\fl) \onto \coker T_{(\fd;[\fl,\Lambda])}\overline{\pi}_d.
  \end{equation*}
  By \autoref{Thm_Index},
  $\fd^\fl$ is Fredholm of index $-\#\Br(\fl) \cdot \rk_\C V$.
  Therefore,
  $\fd^\fl \circ \cdot \co \Hom(\Lambda,\bE_\fl/\Lambda) \to \Hom(\Lambda,\bF_\fl)$ is Fredholm of index
  $d\cdot \paren{-\#\Br(\fl) \cdot \rk_\C V - d}$.
  This proves \autoref{Prop_Wall_OneLocalSystem_Dimension}.
\end{proof}

Set
\begin{equation*}
  \sE \coloneqq \sE_1
  \qandq
  \pi \coloneqq \pi_1.
\end{equation*}
The wall of failure of $2$--rigidity is parameterised as
\begin{equation*}
  \sW = \im \pi.
\end{equation*}


\subsubsection{Failure of injectivity: construction of $\nu$}
\label{Sec_ConstructionOfNu}

Here is the proof that $\pi$ is essentially injective;
that is:~\autoref{Def_ProperWall}~\autoref{Def_ProperWall_EssentiallyInjective} holds.
The map $\pi \co \sE \to \SpaceOfRealCauchyRiemannOperators(V)$ fails to be injective if there are distinct $(\fd,[\fl_i,\Lambda_i]) \in \sE$ ($i=1,2$).
If $\fl_1 = \fl_2 \eqqcolon \fl$,
then $(\fd,[\fl,\Lambda_1 + \Lambda_2]) \in \sE_2$.
The following describes these failures of injectivity.

\begin{prop}
  \label{Prop_FailureOfInjectivity_RankAtLeastTwo}
  Consider the flag manifold bundle $\Flag_{1,2}(\bE)$ over $\SpaceOfRamifiedEuclideanLineBundles(C)$.
  For $d \in \set{1,2}$ consider the projection maps
  $p_d \co \Flag_{1,2}(\bE) \to \Gr_d(\bE)$.
  Set
  \begin{equation*}
    \sN_2 \coloneqq (\id \times p_2)^*\sE_2 \subset \SpaceOfRealCauchyRiemannOperators(V) \times \Flag_{1,2}(\bE).
  \end{equation*}
  The map
  $\nu_2 \co \sN_2 \to \sE_1$ induced by $p_1$ is Fredholm of index
  \begin{equation*}
    -\rk_\C V \cdot \#\Br(\fl) - 2.
  \end{equation*}
\end{prop}

\begin{proof}
  The map $\rho \co \sN_2 \to \sE_2$ induced by $p_2$ is Fredholm of index $1$.
  Therefore,
  by \autoref{Prop_Wall_OneLocalSystem},
  $\pi_2 \circ \rho = \pi_1 \circ \nu_2$ is Fredholm of index $\paren{2 - 2 \cdot \rk_\C V} \cdot \#\Br(\fl) - 3$.
  This implies the assertion.
\end{proof}

Failures of injectivity with $\fl_1 \not\iso \fl_2$ are described as follows.

\begin{prop}
  \label{Prop_Wall_TwoLocalSystems}
  Set $\SpaceOfRamifiedEuclideanLineBundles^{(2)} \coloneqq \SpaceOfRamifiedEuclideanLineBundles^2\setminus\Delta$ (with $\Delta$ denoting the diagonal).
  Consider $\pr_1^*\Gr_1(\bE) \times \pr_2^*\Gr_1(\bE)$ over $\SpaceOfRamifiedEuclideanLineBundles^{(2)}$.
  Denote by $\TR_1^{(2)}$ the smooth section of $\pr_1^*\Hom(\sS,\bF) \oplus \pr_2^*\Hom(\sS,\bF) \to \SpaceOfRealCauchyRiemannOperators(V) \times \pr_1^*\Gr_1(\bE) \times \pr_2^*\Gr_1(\bE)$ induced by $\TR_1$.
  The following hold:
  \begin{enumerate}
  \item
    \label{Prop_Wall_TwoLocalSystems_Property}
    $\TR_1^{(2)}(\fd;[\fl_1,\Lambda_1];[\fl_2,\Lambda_2]) = 0$ if and only if $\Lambda_i \subset \ker \fd^{\fl_i}$ ($i=1,2$).
  \item
    \label{Prop_Wall_TwoLocalSystems_Transverse}
    $\TR_1^{(2)}$ intersects the zero section transversely.
  \item
    \label{Prop_Wall_TwoLocalSystems_Dimension}
    The projection $\pi_1^{(2)} \co \sE_1^{(2)} \coloneqq \paren{\TR_1^{(2)}}^{-1}(0) \to \SpaceOfRealCauchyRiemannOperators(V)$ is Fredholm of index
    \begin{equation*}
      \paren{2 - \rk_\C V} \cdot \paren{\#\Br(\fl_1) + \#\Br(\fl_1)} - 2
    \end{equation*}
    at $(\fd;[\fl_1,\Lambda_1];[\fl_2,\Lambda_2])$.
  \end{enumerate}  
\end{prop}

\begin{prop}
  \label{Prop_FailureOfInjectivity_OtherLocalSystem}
  Set $\sN_1^{(2)} \coloneqq \sE_1^{(2)}$.
  The map $\nu_1^{(2)} \co \sN_1^{(2)} \to \sE_1$ is Fredholm of index
  \begin{equation*}
    \paren{2 - \rk_\C V} \cdot \#\Br(\fl) - 1
  \end{equation*}
\end{prop}

\begin{proof}
  The proof is straight-forward and similar to that of \autoref{Prop_FailureOfInjectivity_RankAtLeastTwo}.
\end{proof}

The proof of \autoref{Prop_Wall_TwoLocalSystems}
requires the following preparation.

\begin{prop}
  \label{Prop_NonIsomorphicLineBundle=>LinearlyIndependent}
  Let $\fd$ be a real Cauchy--Riemann operator on $V$.
  Let $\fl_i$ be a ramified Euclidean line bundle over $C$ and $s_i \in \ker \fd^{\fl_i} \setminus \set{0}$ ($i=1,2$).
  If $\fl_1$ and $\fl_2$ are not isomorphic,
  then there are a non-empty open subset $U \subset C \setminus \paren{\Br(\fl_1) \cup \Br(\fl_2)}$ and trivialisations $\tau_i \co \fl_i|_U \iso \ubR$ ($i=1,2$) such that 
  $\tau_1 s_1, \tau_2 s_2 \in \ker \fd|_U$ are linearly-independent.
\end{prop}

\begin{proof}
  Suppose not.
  Choose a open cover $(U_\alpha)_{\alpha \in A}$ of $C \setminus \paren{\Br(\fl_1) \cup \Br(\fl_2)}$ and
  trivialisations $\tau_i^\alpha \co \fl_i|_{U_\alpha} \iso \ubR$ ($i=1,2$, $\alpha \in A$).
  Define $\lambda_\alpha \in \R^\times$ by $\tau_2^\alpha s_2 = \lambda_\alpha \cdot \tau_1^\alpha s_1$.
  A moment's thought shows that
  \begin{equation*}
    (\tau_2^\alpha)^{-1} \circ  \lambda_\alpha \circ \tau_1^\alpha
    =
    (\tau_2^\beta)^{-1} \circ  \lambda_\beta \circ \tau_1^\beta
  \end{equation*}
  on $U_\alpha \cap U_\beta$.
  Therefore, $\fl_1 \iso \fl_2$ as line bundles over $C \setminus \paren{\Br(\fl_1) \cup \Br(\fl_2)}$.
  This implies that $\fl_1 \iso \fl_2$ as ramified Euclidean line bundles over $C$.
\end{proof}

\begin{proof}[Proof of \autoref{Prop_Wall_TwoLocalSystems}]
  \autoref{Prop_Wall_TwoLocalSystems_Property} is obvious.  

  As in the proof of \autoref{Prop_Wall_OneLocalSystem},
  \autoref{Prop_Wall_TwoLocalSystems_Transverse} reduces to proving that the \defined{Petri map} $\varpi \co \Lambda_1 \otimes \ker \fd_{L^2}^{\dagger,\fl_1} \oplus \Lambda_2 \otimes \ker \fd_{L^2}^{\dagger,\fl_2} \to L^1\Gamma(C,V\otimes V^\dagger)$
  defined by
  \begin{equation*}
    \varpi(s_1\otimes t_1,s_2\otimes t_2)(x)
    \coloneqq
    s_1(x)\otimes t_1(x) + s_2(x)\otimes t_2(x)
  \end{equation*}
  is injective.
  This is a consequence of
  \autoref{Prop_NonIsomorphicLineBundle=>LinearlyIndependent} and \autoref{Prop_NonParallel_Global=>AlmostEverywhere}.

  The proof of \autoref{Prop_Wall_TwoLocalSystems_Dimension} is similar to that of   \autoref{Prop_Wall_OneLocalSystem}~\autoref{Prop_Wall_OneLocalSystem_Dimension}
  and,
  therefore, omitted.
\end{proof}

Set $\sN \coloneqq \sN_2 \amalg\sN_1^{(2)}$ and $\nu \coloneqq \nu_2 \amalg \nu_1^{(2)}$.
The above discussion shows that $\nu \co \sN \to \sE$ is Fredholm of index at most $-1$ and $\pi|_{\sE\setminus\im\nu} \co \sE\setminus\im\nu \to \sW$ is injective.


\subsubsection{Failure of smoothness: construction of $\sigma$}
\label{Sec_ConstructionOfSigma}

The crucial step in the proof of \autoref{Thm_WallOfFailureOf2RigidityIsProper} is to establish that $\pi$ is essentially proper;
that is: \autoref{Def_ProperWall}~\autoref{Def_ProperWall_EssentiallyProper} holds.
A naive hope might be that \autoref{Thm_FedererFlemming} implies outright properness.
Indeed:
if $(\fd_n;[\fl_n,\Lambda_n = \Span{s_n}]) \in \sE^\N$ is such that $(\fd_n) \in \sW^\N$ converges to $\fd_\infty \in \sW$,
then it can be arranged that the sequence of pseudo-holomorphic cycles $(T_n)$ corresponding to $(s_n)$ converges to a $J_\infty$--holomorphic cycle $T_\infty$.
This, however, does not imply that $([\fl_n,\Lambda_n])$ converges to $[\fl_\infty,\Lambda_\infty]$%
---\emph{unless:} $T_\infty$ is smooth and
$\pi^{-1}(\fd_\infty) = \set{[\fl_\infty,\Lambda_\infty = \Span{s_\infty}]}$.
Fortunately,
it is straight-forward to describe when $T_\infty$ fails to be smooth or, equivalently, when the corresponding $J_\infty$--holomorphic map $u_\infty \co \tilde C \to V$ fails to be an injective immersion.

\begin{definition}
  \label{Def_12JetEvaluation}
  Assume the situation of \autoref{Prop_RiemannSurface_DegreeTwo_+-}.
  Let $b \in \Br(\fl)$.
  Set
  $J_b^{1/2}(V \otimes \fl) \coloneq V_b \otimes_\C J_b^{1/2}(\fl \otimes \C)$.
  Set $\set{\tilde b} \coloneqq \pi^{-1}(b)$.
  Denote by $\pi^* \co J_b^{1/2}(V \otimes \fl) \iso V_b \otimes_\C T_{\tilde b}^*\tilde C$ the isomorphism induced by $\pi$ (see \autoref{Rmk_SquareRoot}).
  The \defined{$\tfrac12$--jet evaluation map at $b$} is the linear map $j_b^{1/2} \co \ker \fd^\fl \to J_b^{1/2}(V \otimes \fl)$ defined by
  \begin{equation*}
    \pi^*j_b^{1/2}(s) \coloneqq \del_{\tilde b}u \in V_b \otimes_\C T_{\tilde b}^*\tilde C
  \end{equation*}
  with $u \co \tilde C \to V$ denoting the $J$--holomorphic map corresponding to $s$ (according to \autoref{Sec_ThreePerspectivesOn2Rigidity}).
\end{definition}

\begin{prop}
  \label{Prop_InjectiveImmersion~NowhereVanishing}
  Let $\fl$ be a ramified Euclidean line bundle.
  Let $s \in \ker \fd^\fl$.
  Let $u \co \tilde C \to V$ be the corresponding $J$--holomorphic map (according to \autoref{Sec_ThreePerspectivesOn2Rigidity}).
  The following are equivalent:
  \begin{enumerate}
  \item
    \label{Prop_InjectiveImmersion~NowhereVanishing_Embedding}
    The map $u \co \tilde C \to V$ is an injective immersion.
  \item
    \label{Prop_InjectiveImmersion~NowhereVanishing_NonVanishing}
    The section $s$ is nowhere vanishing in the following sense:
    \begin{enumerate}
    \item
      \label{Prop_InjectiveImmersion~NowhereVanishing_NonVanishing_Regular}
      For every $x \in \mathring{C}$,
      $\ev_x(s) \coloneqq s(x) \neq 0$.
    \item
      \label{Prop_InjectiveImmersion~NowhereVanishing_NonVanishing_Branch}
      For every $b \in \Br(\fl)$,
      $j_b^{1/2}(s) \neq 0$.
    \end{enumerate}
  \end{enumerate}
\end{prop}

\begin{proof}
  The map $u$ is injective if and only if \autoref{Prop_InjectiveImmersion~NowhereVanishing_NonVanishing_Regular} holds.
  $T_{\tilde x}\tilde u \neq 0$ for every $\tilde x \in \mathring{\tilde C}$.  
  If $\tilde b \in \pi^{-1}(\Br(\fl))$,
  then $T_{\tilde b} u = \del_{\tilde b}u$.
  Therefore, $u$ is an immersion if and only only if \autoref{Prop_InjectiveImmersion~NowhereVanishing_NonVanishing_Branch} holds.
\end{proof}

The failure of
\autoref{Prop_InjectiveImmersion~NowhereVanishing_NonVanishing_Regular} 
is described by the following.

\begin{prop}
  \label{Prop_Wall_OneLocalSystem_RegularPointVanishing}
  Consider the fibre product $\sE \times_{\Conf(C)} \mathring{\sC}$.
  The following hold:
  \begin{enumerate}
  \item
    \label{Prop_Wall_OneLocalSystem_RegularPointVanishing_Bundle}
    There are
    a vector bundle $\Hom(\sS,V \otimes \fL) \to \sE \times_{\Conf(C)} \mathring{\sC}$ and
    for every $(\fd;[\fl,\Lambda];x) \in \sE \times_{\Conf(C)} \mathring{\sC}$
    (canonical) isomorphisms
    \begin{equation*}
      \eta_{(\fd;\fl,\Lambda;x)} \co \Hom(\sS,V \otimes \fL)_{(\fd;[\fl,\Lambda];x)} \iso \Hom\paren{\Lambda,V_x \otimes \fl_x}.
    \end{equation*}
  \item
    \label{Prop_Wall_OneLocalSystem_RegularPointVanishing_Section}
    $\Hom(\sS,V \otimes \fL) \to \sE \times_{\Conf(C)} \mathring{\sC}$ has a $C^1$ section $\ev$ such that
    \begin{equation*}
      \paren{\ev(\fd;[\fl,\Lambda];x)}
      = \eta_{(\fd;\fl,\Lambda;x)}^{-1}\paren{\ev_x}
    \end{equation*}
    with $\ev_x \in \Hom\paren{\Lambda,V_x \otimes \fl_x}$ defined by $\ev_x(s) \coloneqq s(x)$.
  \item
    \label{Prop_Wall_OneLocalSystem_RegularPointVanishing_Property}
    $\ev(\fd;[\fl,\R\Span{s}],x) = 0$ if and only if $s(x) = 0$.
  \item
    \label{Prop_Wall_OneLocalSystem_RegularPointVanishing_Transverse}
    $\ev$ intersects the zero section transversely;
    in particular, $\sS^{\Reg} \coloneqq \ev^{-1}(0)$ is a Banach submanifold.
  \item
    \label{Prop_Wall_OneLocalSystem_RegularPointVanishing_Dimension}
    The projection map $\sigma^{\Reg} \co \sS^{\Reg} \to \sE$ is Fredholm of index
    \begin{equation*}
      2 - \rk V.
    \end{equation*}
  \end{enumerate}
\end{prop}

\begin{proof}
  \autoref{Prop_Wall_OneLocalSystem_RegularPointVanishing_Bundle} and \autoref{Prop_Wall_OneLocalSystem_RegularPointVanishing_Section} are evident from the discussion in \autoref{Sec_SpaceOfRamifiedEuclideanLineBundles}.  
  \autoref{Prop_Wall_OneLocalSystem_RegularPointVanishing_Property} holds by construction.

  \medskip 
  
  To prove \autoref{Prop_Wall_OneLocalSystem_RegularPointVanishing_Transverse} it suffices to show that for $(\fd,[\fl,\R\Span{s}],x) \in \ev^{-1}(0)$ the map
  \begin{equation*}
    \ev_x \co \ker\paren[\big]{W^{1,2}\Gamma\paren{\mathring{C},V\otimes \fl} \oplus C^K\Gamma\paren{X,\HOMT(V)} \to L^2\Omega^{0,1}\paren{\mathring{C},V\otimes \fl}} \to V_x \otimes \fl_x
  \end{equation*}
  is surjective.
  Let $k \in \set{1,\cdots,K}$.
  By \autoref{Prop_Regularity},
  \begin{multline*}
    \ker\paren[\big]{W^{1,2}\Gamma\paren{\mathring{C},V\otimes \fl} \oplus C^K\Gamma\paren{X,\HOMT(V)} \to L^2\Omega^{0,1}\paren{\mathring{C},V\otimes \fl}} \\
    =
    \ker\paren[\big]{W_\cyl^{k+1,2}\Gamma\paren{\mathring{C},V\otimes \fl} \oplus C^K\Gamma\paren{X,\HOMT(V)} \to W_\cyl^{k,2}\Omega^{0,1}\paren{\mathring{C},V\otimes \fl}}.
  \end{multline*}
  The significance of the above is that
  $W_\cyl^{k+1,2}\Gamma\paren{\mathring{C},V \otimes \fl} \incl C^{k-1}\Gamma\paren{\mathring{C},V \otimes \fl}$;
  hence, $\ev_x$ is defined on $W_\cyl^{k+1,2}\Gamma\paren{\mathring{C},V \otimes \fl}$.
  By the Snake Lemma,
  it suffices to prove that the map $\bL \co W_\cyl^{k+1,2}\Gamma\paren{\mathring{C},V\otimes \fl} \oplus C^K\Gamma(X,\HOMT(V))
  \to
  W_\cyl^{k,2}\Omega^{0,1}\paren{\mathring{C},V\otimes \fl} \oplus V_x \otimes \fl_x$
  defined by
  \begin{equation*}
    \bL(\hat s,\fa) \coloneqq \paren[\big]{\fd^\fl \hat s + (\fa\otimes\id_\fl)s, \hat s(x)}
  \end{equation*}
  is surjective.
  Choose $\fV \subset \Gamma\paren{\mathring{C},V\otimes \fl}$ such that $\ev_x \co \fV \to \paren{V \otimes \fl}_x$ is an isomorphism.
  It suffices to prove that the map
  \begin{equation*}
    \bM \co C^K\Gamma(X,\HOMT(V)) \to
    \coker\paren[\big]{
      \fd_{W_\cyl^{k+1,2}}^\fl \co \fV^\perp \to W_\cyl^{k,2}\Omega^{0,1}\paren{\mathring{C},V\otimes \fl}
    }
    \iso
    \paren[\big]{\ker \fd_{L^2}^{\dagger,\fl} + \fd^\fl(\fV)}^*
  \end{equation*}
  defined by
  \begin{equation*}
    \Inner{\bM(\fa),t} \coloneqq \Inner{t,(\fa\otimes \id_\fl) s}
  \end{equation*}
  is surjective.
  $\bM$ is surjective,
  because the Petri map
  $\varpi \co \ker \fd_{L^2}^{\dagger,\fl} + \fd^\fl(\fV) \to L^1\Gamma(\mathring{C}, V \otimes V^\dagger)$ defined by $\varpi(t) \coloneqq s \otimes t$ is injective. 
  
  \autoref{Prop_Wall_OneLocalSystem_RegularPointVanishing_Dimension} is obvious.
\end{proof}

\medskip

The failure of
\autoref{Prop_InjectiveImmersion~NowhereVanishing_NonVanishing_Branch}
is described by the following.

\begin{prop}
  \label{Prop_Wall_OneLocalSystem_BranchPointVanishing}
  Consider the \defined{universal branch points}
  \begin{equation*}
    \sB \coloneqq \set{ (B,b) \in \Conf(C) \times C : b \in B }.
  \end{equation*}
  Consider the fibre product $\sE \times_{\Conf(C)} \sB$.
  The following hold:
  \begin{enumerate}
  \item
    \label{Prop_Wall_OneLocalSystem_BranchPointVanishing_Bundle}
    There are
    a vector bundle $\Hom\paren{\sS,\JetSpace^{1/2}(V\otimes\fL)} \to \sE \times_{\Conf(C)} \sB$ and
    for every $(\fd;[\fl,\Lambda];b) \in \sE \times_{\Conf(C)} \sB$
    (canonical) isomorphisms
    \begin{equation*}
      \eta_{(\fd;\fl,\Lambda;b)} \co \Hom\paren{\sS,\JetSpace^{1/2}(V\otimes\fL)}_{(\fd;[\fl,\Lambda];b)} \iso \Hom\paren{\Lambda,\JetSpace_b^{1/2}(V\otimes\fl)}.
    \end{equation*}
  \item
    \label{Prop_Wall_OneLocalSystem_BranchPointVanishing_Section}
    $\Hom\paren{\sS,\JetSpace^{1/2}(V\otimes\fL)} \to \sE \times_{\Conf(C)} \mathring{\sB}$ has a $C^1$ section $j^{1/2}$ such that
    \begin{equation*}
      \paren{j^{1/2}(\fd;[\fl,\Lambda];b)}
      = \eta_{(\fd;\fl,\Lambda;x)}^{-1}\paren{j_b^{1/2}}.
    \end{equation*}
  \item
    \label{Prop_Wall_OneLocalSystem_BranchPointVanishing_Property}
    $j^{1/2}(\fd;[\fl,\R\Span{s}],b) = 0$ if and only if $j_b^{1/2}(s) = 0$.
  \item
    \label{Prop_Wall_OneLocalSystem_BranchPointVanishing_Transverse}
    $j^{1/2}$ intersects the zero section transversely;
    in particular, $\sS^\Br \coloneqq \paren{j^{1/2}}^{-1}(0)$ is a Banach submanifold.
  \item
    \label{Prop_Wall_OneLocalSystem_BranchPointVanishing_Dimension}
    The projection map $\sigma^{\Br} \co \sS^{\Br} \to \sE$ is Fredholm of index
    \begin{equation*}
      - \rk_\C V.
    \end{equation*}
  \end{enumerate}
\end{prop}

\begin{proof}
  The crucial point is to extend $j_b^{1/2}$ from $\ker \fd^\fl$ to $W^{1,2}\Gamma\paren{\mathring{C},V\otimes \fl}$.
  By \autoref{Rmk_SquareRoot}~\autoref{Rmk_SquareRoot_12JetSpace} and
  \autoref{Rmk_12JetSpaceAnd12ResidueSpaceAreDual},
  $\paren{\JetSpace_b^{1/2}(\fl\otimes\C)}^{\otimes 2} = T_b^*C$ and
  $\ResidueSpace_b(\fl) \iso \paren{\JetSpace_b^{1/2}(\fl\otimes\C)}^*$.
  Therefore,
  \begin{equation*}
    \JetSpace_b^{1/2}(\fl\otimes\C) \iso \ResidueSpace_b(\fl) \otimes_\C T_b^*C.
  \end{equation*}
  Let $\zeta \in C^\infty(\C,C)$ with
  $\zeta^{-1}(0) = \set{b}$,
  $T_b \zeta \neq 0$, and
  $\abs{\delbar \zeta} \lesssim \abs{\zeta}$.
  Since $W^{1,2}\Gamma\paren{\mathring{C},V\otimes\fl} \incl rL^2\paren{\mathring{C},V\otimes\fl}$,
  for every $s \in W^{1,2}\Gamma\paren{\mathring{C},V\otimes\fl}$,
  $\zeta^{-1}s \in \dom(\fd_{\max}^\fl)$;
  moreover:
  if $s \in \ker\fd^\fl$,
  then
  \begin{equation*}
    \Res_b(\zeta^{-1}s) \otimes_\C \del_b\zeta = j_b^{1/2}(s).
  \end{equation*}
  This provides the extension of $j_b^{1/2}$.

  The remainder of the proof is analogous to that of that of \autoref{Prop_Wall_OneLocalSystem_RegularPointVanishing} and, therefore, omitted.  
\end{proof}

Set $\sS \coloneqq \sS^{\Reg} \amalg \sS^{\Br}$ and $\sigma \coloneqq \sigma^{\Reg} \amalg \sigma^{\Br}$.
The above discussion shows that $\sigma \co \sS \to \sE$ is Fredholm of index at most $-1$.


\subsubsection{Proofs of \autoref{Thm_WallOfFailureOf2RigidityIsProper} and  \autoref{Thm_WallOfFailureOf2RigidityIsProper+SelfAdjoint}}

Here is the final ingredient for the proofs.

\begin{prop}
  \label{Prop_AlmostProper}
  The map
  \begin{equation*}
    \pi|_{\sE\setminus\paren{\im\nu \cup \im\sigma}} \co \sE \setminus \paren{\im \nu \cup \im \sigma} \to \SpaceOfRealCauchyRiemannOperators(V) \setminus \paren{\im \paren{\pi \circ \nu} \cup \im \paren{\pi \circ \sigma}}
  \end{equation*}
  is proper.
\end{prop}

\begin{proof}
  Let $(\fd_n;[\fl_n,\Lambda_n=\Span{s_n}]) \in \paren{\sE \setminus \paren{\im \nu \cup \im \sigma}}^\N$.
  Suppose that $(\fd_n)$ converges to $\fd_\infty \in \SpaceOfRealCauchyRiemannOperators(V) \setminus \paren{\im \paren{\pi \circ \nu} \cup \im \paren{\pi \circ \sigma}}$.
  Set $(\fd_\infty;[\fl_\infty,\Lambda_n=\Span{s_\infty}]) \coloneqq \pi^{-1}(\fd_\infty)$.
  For $n \in \N\cup\set{\infty}$
  set $J_n \coloneqq J_{\fd_n}$ and
  denote by $u_n \co \tilde C_n \to V$ the $J_n$--holomorphic map corresponding to $s_n$,
  normalised such that
   \begin{equation*}
    \max_{\tilde C_n} \rho^2 \circ u_n = 1.
  \end{equation*}

  As in the proof of \autoref{Prop_TheWallIsClosed},
  after passing to a subsequence,
  $(T_n \coloneqq T_{u_n})$ converges to a $J_\infty$--holomorphic cycle $T_\infty$ in the weak--$*$--topology.
  By injectivity, $T_\infty = T_{u_\infty}$.

  Since $u_n$ ($n \in \N \cup \set{\infty}$) is an injective immersion,
  it follows from \cite[Corollary 2.29]{Doan2021} that $(u_n)$ converges to $u_\infty$;
  therefore, $([\fl_n,\Lambda_n])$ converges to $(\fl_\infty,\Lambda_\infty)$.
  (The proof of \cite[Corollary 2.29]{Doan2021} itself relies crucially on an observation due to \citet{White2005} regarding a simple version of Allard's Regularity Theorem  \cite{Allard1972}.)
\end{proof}

\begin{proof}[Proof of \autoref{Thm_WallOfFailureOf2RigidityIsProper}]
  This is an immediate consequence of the above discussion.
\end{proof}

\begin{proof}[Proof of \autoref{Thm_WallOfFailureOf2RigidityIsProper+SelfAdjoint}]
  The proof is nearly identical to that of  \autoref{Thm_WallOfFailureOf2RigidityIsProper}.
  The constructions of $\pi$, $\nu$, and $\sigma$ have to be adapted sightly to account for the $\SerreOperator$--self-adjointness;
  cf.~\cite[§1.A]{Doan2018}.
  As a result of this index formula in \autoref{Prop_Wall_OneLocalSystem}~\autoref{Prop_Wall_OneLocalSystem_Dimension} changes to
  \begin{equation*}
    \paren{2 - d \cdot \rk_\C V} \cdot \#\Br(\fl) - \binom{d+1}{2}.
    \qedhere
  \end{equation*}
\end{proof}

\begin{remark}
  \label{Rmk_WhyRestrictToK=2}
  The astute reader may have noticed that for most of section the restriction to $k$--rigidity with $k = 2$ was unnecessary.
  Here is why this restriction was made.  
  For $k > 2$ the translation between the perspectives laid out in \autoref{Sec_ThreePerspectivesOn2Rigidity} becomes more subtle.  
  In particular, if $\fL$ is a higher rank branched local system,
  then $s \in \ker \fd^\fL$ does not directly correspond to a $J$--holomorphic map $\tilde u \co \tilde C \to V$;
  cf.~\cite[\S 1.5]{Doan2018}.
  As a consequence of this, evidence of the failure of $k$--rigidity in the form of a $J$--holomorphic map $u \co \tilde C \to V$ no longer typically is an injective immersion (except for rather small $k$ and/or under additional hypothesis).
  This breaks the above proof strategy.
  
  Nevertheless,
  the restriction to $k = 2$ may not be necessary.
  Moreover,
  the use of Allard's Regularity Theorem and the perspective of $J$--holomorphic maps $u \co \tilde C \to V$ might well be red herring.  
  Roughly speaking,
  the above argument proceeds as follows.
  Equip $\sE$ with a topology that is coarser than the one discussed in \autoref{Sec_SpaceOfRamifiedEuclideanLineBundles}.
  The coarse topology on $\sE$ allows for the branch locus of the Euclidean local system $\fl$ to jump.  
  This is an essential ingredient to prove that the coarse topology has better compactness properties.
  The existence of a non-zero $s \in \ker \fd^\fl$ tames $\fl$.
  A typical $s \in \ker \fd^\fl$ vanishes only where geometry forces it to.
  As a consequence,
  typical sequences with typical limits in the coarse topology converge in the fine topology.
  Its possible that this strategy extends to branched local systems $\fL$ of higher rank,
  provided an appropriate extension of the coarser topology to this situation is found.
\end{remark}





\section{Chambered invariants from blow-ups}
\label{Sec_BlowUp}

Let $C$ be a closed connected Riemann surface.
Let $V$ be a complex vector bundle over $C$ with
\begin{equation*}
  \rk_\C V  = 2
  \qandq
  2 \deg V + \rk_\C V \cdot \chi(C) = 0.
\end{equation*}
This section constructs a chambered invariant
\begin{equation*}
  \BlowUpInvariant \in \rH^0(\SpaceOfRealCauchyRiemannOperators(V)\setminus\sW;\Z[[x]])
\end{equation*}
by counting $J$--holomorphic sections
of $q \co \sO_{\P V}(-2) \to C$,
the blow-up of the $\C^2/\set{\pm 1}$--bundle $\check p \co V/\set{\pm 1} \to C$,
with respect to an admissible almost complex structure on $\sO_{\P V}(-2)$ which is homogeneous near infinity and, therefore, induces a real Cauchy--Riemann operator on $V$.

The construction of $\BlowUpInvariant$ is reminiscent of the construction in \cite[\S 8.6]{McDuff2012}.
In fact, various aspects are simpler in the present situation.
However, a substantial novel complication arises because $\sO_{\P V}(-2)$ is non-compact:
sequences of $J$--holomorphic sections might escape to infinity.
A combination of  geometric measure theory (again) and a twist on Radó's theorem in complex analysis
are used to prove that this phenomenon can only occur if the real Cauchy--Riemann operator associated with $J$ fails to be $2$--rigid.

\subsection{The blow-up: \texorpdfstring{$\sO_{\P V}(-2)$}{O(-2)}}
\label{Sec_FibrewiseBlowUp}

Here is a summary of
the construction of the blow-up of $\check p \co V/\set{\pm 1} \to C$,
and some of its properties relevant to the further discussion.

\begin{definition}
  \label{Def_FibrewiseBlowUp}
  The blow-up of $\check p \co V/\set{\pm 1} \to C$ is constructed as follows:
  \begin{enumerate}
  \item    
    The \defined{projectivisation of $V$} is the $\CP^1$--bundle
    \begin{equation*}
      \varpi \co \bP V \coloneqq V^\times/\C^\times \to C.
    \end{equation*}
    Here $V^\times \coloneqq V \setminus 0$ and $\C^\times$ acts with weight $1$.
  \item
    Consider the complex line bundle
    \begin{equation*}
      \varsigma \co \sO_{\bP V}(-2) \coloneqq \paren{V^\times \times \C}/\C^\times \to \bP V
    \end{equation*}
    with $\C^\times$ acting with weight $(1,-2)$.
    The \defined{blow-up of $\check p \co V/\set{\pm 1} \to C$} is the fibre bundle
    \begin{equation*}
      q \coloneqq \varpi \circ \varsigma \co \sO_{\P V}(-2) \to C.
    \end{equation*}
  \item
    The \defined{blow-down map} is the continuous map
    $\beta \co \sO_{\bP V}(-2) \to V/\set{\pm 1}$
    defined by
    \begin{equation*}
      \beta([v,\lambda]) \coloneqq [\pm \sqrt{\lambda} v].
      \qedhere
    \end{equation*}
  \end{enumerate}
\end{definition}

Of course,
the blow-down map $\beta$
satisfies $q = \check{p} \circ \beta$, and
induces a diffeomorphism
\begin{equation*}
  \sO_{\P V}(-2)^\times \coloneqq \sO_{\P V}(-2)\setminus 0
  \iso
  V^\times/\set{\pm 1}.
\end{equation*}
This shall be regarded as an identification.

\medskip

If $I \in \SpaceOfHomogeneousAlmostComplexStructures(V)$ is $\C^\times$--invariant,
then it is integrable,
and descends to a complex structure $\check I$ on $V^\times/\set{\pm 1}$;
moreover: $\check I$ extends to a complex structure $\tilde I$ on $\sO_{\P V}(-2)$.
(If $J \in \SpaceOfHomogeneousAlmostComplexStructures(V)$ is not $\C^\times$--invariant,
then $\check J$ does not extend from $V^\times/\set{\pm 1}$ to $\sO_{\P V}(-2)$.)

Choose a Hermitian metric on $V$.
Define $\rho^2 \in C^\infty(V)$ by $\rho^2(v) \coloneqq \abs{v}^2$.
Denote the induced smooth maps on $V^\times/\set{\pm 1}$ and $\sO_{\P V}(-2)$ by $\check \rho^2$ and $\tilde\rho^2$ respectively.

\begin{prop}
  \label{Prop_FibrewiseKahlerForm}
  The complex fibre bundle $q \co \sO_{\P V}(-2) \to C$ admits the following family of fibrewise Kähler forms:
  \begin{enumerate}
  \item
    \label{Prop_FibrewiseKahlerForm_Existence}
    Let $t \geq 0$.
    Define $f \in C^\infty\paren{\R_{>0}}$ by
    \begin{equation*}      
      f_t(u) \coloneqq \sqrt{u^2 + t^4} + t^2\log u - t^2\log \paren{\sqrt{u^2 + t^4} + t^2}
    \end{equation*}
    and set  
    \begin{equation*}
      \Omega_{I,t}
      \coloneqq
      -\frac14 \rd \sqparen{ \rd (f_t \circ \rho^2) \circ I}
      \in \Omega^2(V^\times).
    \end{equation*}
    $\Omega_{I,t}$ is $\set{\pm 1}$--invariant and descends to $\check\Omega_{I,t} \in \Omega^2(V^\times/\set{\pm})$;
    moreover:
    $\check\Omega_{I,t}$ extends to a closed $2$--form
    \begin{equation*}
      \tilde\Omega_{I,t} \in \Omega^2\paren{\sO_{\bP V}(-2)}.
    \end{equation*}
  \item
    \begin{enumerate}
    \item 
      \label{Prop_FibrewiseKahlerForm_T>0Kahler}
      For every $t > 0$,
      $\tilde \Omega_{I,t}$ is a fiberwise Kähler form with respect to $\tilde I$;
      that is:
      for every $x \in C$,
      $\tilde \Omega_{I,t}|_{\sO_{\bP V_x}(-2)}$ is a Kähler form with respect to $\tilde I|_{\sO_{\bP V_x}(-2)}$;
      moreover:
      $\tilde \Omega_{I,t}|_{\P V_x}$ is a positive multiple of the Fubini--Study form on $\P V_x$.
    \item
      \label{Prop_FibrewiseKahlerForm_T=0Standard}
      $\Omega_{I,0}$ is a fibrewise Kähler form with respect to $I$ on $V^\times$;
      moreover:
      the fibrewise Kähler metric is the one induced by the Hermitian metric on $V$.
    \end{enumerate}
  \item
    \label{Prop_FibrewiseKahlerForm_Scaling}
    For every $t \geq 0$ and $R > 0$,
    \begin{equation*}
      R^{-2} \cdot R^*\Omega_{I,t} = \Omega_{I,t/R}.
    \end{equation*}
  \end{enumerate}
\end{prop}

\begin{proof}
  \autoref{Prop_FibrewiseKahlerForm_Existence} and \autoref{Prop_FibrewiseKahlerForm_T>0Kahler} are verified by straight-forward computations.
  Indeed, this is nothing but (a family version of) the construction due to   \citet{Eguchi1979}; cf.~\cite{Lye2023:CalabiEguchiHanson}.

  \autoref{Prop_FibrewiseKahlerForm_T=0Standard} and \autoref{Prop_FibrewiseKahlerForm_Scaling} are obvious.
\end{proof}

\medskip

Here are two observations regarding the topology of $\sO_{\P V}(-2)$,
which are required in the following.

\begin{prop}
  \label{Prop_FirstChernClassOfVerticalTangentBundle}
  The \defined{vertical tangent bundle}
  $T^\ver\sO_{\P V}(-2) \coloneqq \ker Tq \subset T\sO_{\P V}(-2)$
  satisfies
  \begin{equation*}
    c_1(T^\ver\sO_{\P V}(-2)) = q^*c_1(V).
  \end{equation*}
\end{prop}

\begin{proof}
  Since
  \begin{equation*}
    \varsigma^*\sO_{\P V}(-2) \into T^\ver\sO_{\P V}(-2) \onto \varsigma^*T^\ver\P V
  \end{equation*}
  is exact,
  \begin{equation*}
    c_1(T^\ver\sO_{\P V}(-2)) = \varsigma^*\paren[\big]{c_1(T^\ver\P V) - 2c_1(\sO_{\P V}(1))}.
  \end{equation*}
  Since the (relative) Euler sequence
  \begin{equation*}
    \sO_{\P V} \into \sO_{\P V}(1) \otimes_\C \varpi^*V \onto T^\ver\P V
  \end{equation*}
  is exact,  
  \begin{equation*}
    c_1(T^\ver\P V) = \rk_\C V \cdot c_1(\sO_{\P V}(1)) + \varpi^*c_1(V).
  \end{equation*}
  Since $\rk_\C V = 2$,
  these combine to prove the assertion.
\end{proof}

\begin{prop}
  \label{Prop_HomologyClassOfSections}
  If $s \co C \to \sO_{\P V}(-2)$ is a section of $q \co \sO_{\P V}(-2) \to C$,
  then $s_*[C] \in \rH_2\paren{\sO_{\P V}(-2);\Z}$ is determined by the \defined{degree}
  \begin{equation*}
    \deg(s) \coloneqq \Inner{s^*\varsigma^*c_1\paren{\sO_{\bP V}(1)},[C]} \in \Z;
  \end{equation*}
  indeed:
  \begin{equation*}
    s_*[C] = S + \paren{\deg V + \deg(s)} \cdot F
    \in
    \rH_2\paren{\P V;\Z} \iso \rH_2\paren{\sO_{\P V}(-2);\Z}
  \end{equation*}
  with $S \coloneqq \mathrm{PD}\sqparen{c_1(\sO_{\P V}(1))}$ and $F$ denoting the  homology class of a fibre of $\P V$.
\end{prop}

\begin{proof}
  $\rH_2\paren{\P V;\Z} = \Span{S,F}$ with
  \begin{equation*}
    S \cdot S = -\deg V, \quad
    S \cdot F = 1, \qandq
    F \cdot F = 0.
  \end{equation*}
  Since
  \begin{equation*}
    s_*[C]\cdot S = \deg(s) \qandq
    s_*[C]\cdot F = 1,
  \end{equation*}
  the assertion follows.  
\end{proof}

\begin{remark}
  \label{Rmk_DegreeGenus}
  Let $s \co C \to \sO_{\bP V}(-2)$ be a section of $q$.
  Define $\pi \co \tilde C \to C$ by the pullback diagram
  \begin{equation*}
    \begin{tikzcd}
      \tilde C \ar{r}{\pi} \ar[hook]{d} & C \ar[hook]{d}{\beta \circ s} \\
      V \ar[two heads]{r} & V/\set{\pm 1}.
    \end{tikzcd}
  \end{equation*}
  If $s$ is transverse to $\bP V \subset \sO_{\bP V}(-2)$,
  then $\pi \co \tilde C \to C$ is a double cover branched over $s^{-1}(\bP V)$.
  Therefore, by the Riemann--Hurwitz formula,
  \begin{equation*}
    g(\tilde C)
    = 2g(C) - 1 + \abs{s^{-1}(\bP V)}.
  \end{equation*}
  If, moreover, $s$ intersects $\P V$ positively,
  then $\abs{s^{-1}(\bP V)} = \deg(s)$.
  In this sense,
  $\deg(s)$ relates to the genus of $\tilde C$.
\end{remark}


\subsection{Admissible almost complex structures on $\sO_{\P V}(-2)$}
\label{Sec_AlmostComplexStructuresOnBlowup}

Here is a description of the subspace of
$\sJ(\sO_{\P V}(-2))$,
the space of $C^K$ almost complex structures on $\sO_{\P V}(-2)$ employed in the construction of $\BlowUpInvariant$.

\begin{prop}
  \label{Prop_BlowUpAlmostComplexStructures}
  Let $\omega_C$ be a Kähler form on $C$.
  There are
  a subspace
  $\SpaceOfBlowUpAlmostComplexStructures \subset \sJ(\sO_{\P V}(-2))$,
  the structure of a Banach manifold on $\SpaceOfBlowUpAlmostComplexStructures$, and
  a submersion
  \begin{equation*}
    \fd_\infty
    \co
    \SpaceOfBlowUpAlmostComplexStructures
    \to
    \SpaceOfRealCauchyRiemannOperators(V)
  \end{equation*}
  such that the following hold:  
  \begin{enumerate}
  \item
    \label{Prop_BlowUpAlmostComplexStructures_Properties}
    Every $J \in \SpaceOfBlowUpAlmostComplexStructures$ is \defined{admissible};
    that is: the following hold:
    \begin{enumerate}
    \item
      \label{Prop_BlowUpAlmostComplexStructures_Properties_Fibred}
      The map $q \co \sO_{\bP V}(-2) \to C$ is $J$--holomorphic.
    \item
      \label{Prop_BlowUpAlmostComplexStructures_Properties_End}
      If $\fd \coloneqq \fd_\infty(J)$,
      then
      \begin{equation*}
        J = \check J_\fd
        \quad\text{on}\quad
        \set{ \tilde\rho^2 \geq 1 } = \set{ \check\rho^2 \geq 1 } \subset \sO_{\P V}(-2)^\times = V^\times/\set{\pm 1}.
      \end{equation*}
    \item
      \label{Prop_BlowUpAlmostComplexStructures_Properties_Tamed}           
      There is a $\Lambda = \Lambda(J) > 0$ such that
      for every $R > 0$
      the restriction of
      \begin{equation*}
        \Omega_{\tilde I,1} + \Lambda (1+R^2) \cdot q^* \omega_C
      \end{equation*}
      to $\set{ \tilde\rho^2 < 2R^2 } \subset \sO_{\P V}(-2)$ is symplectic and tames $J$.
      Here $I \coloneqq J_{\delbar}$ with $\delbar$ denoting the Dolbeault operator induced by $\fd$.
    \end{enumerate}
  \item
    \label{Prop_BlowUpAlmostComplexStructures_FibreIsContractible}
    For every $\fd \in \SpaceOfRealCauchyRiemannOperators(V)$,
    $\fd_\infty^{-1}(\fd)$ is contractible; in particular: non-empty.
  \item
    \label{Prop_BlowUpAlmostComplexStructures_IsAmple}
    For every $J \in \SpaceOfBlowUpAlmostComplexStructures$ there is an $\epsilon = \epsilon(J) > 0$ such that
    if
    \begin{equation*}
      \fa \in B_\epsilon(0) \subset C_c^K\Gamma\paren{\sO_{\P V}(-2),\overline{\End}_\C(T\sO_{\P V}(-2),J)}
    \end{equation*}
    satisfies $Tq \circ \fa = 0$ and vanishes on $\set{ \rho^2 \geq 1 } \subset \sO_{\P V}(-2)$,
    then
    \begin{equation*}
      J(\fa) \coloneqq \paren{\one + \tfrac12 J\fa} J \paren{\one + \tfrac12 J\fa}^{-1} \in \SpaceOfBlowUpAlmostComplexStructures;
    \end{equation*}
    moreover:
    $\fa \mapsto J(\fa)$ is smooth.
  \end{enumerate}
\end{prop}

\begin{proof}
  Let $\fd \in \SpaceOfRealCauchyRiemannOperators(V)$.
  Denote by $\delbar$ the Dolbeault operator induced by $\fd$.
  Set $I \coloneqq J_{\delbar}$ and $\fa \coloneqq J_{\delbar}$.
  Let $\chi \in C^\infty\paren{[0,\infty),0}$ with $\chi|_{[0,1/2]} = 0$ and $\chi|_{[1,\infty)} = 1$.
  The almost complex structure
  \begin{equation*}
    J_\fd^\chi \coloneqq I + \paren{\chi \circ \rho^2} \cdot \fa
  \end{equation*}
  is $\set{\pm 1}$--invariant, and
  $\check J_\fd^\chi$ extends to an almost complex structure $\tilde J_\fd^\chi$ on $\sO_{\P V}(-2)$.

  By construction,
  $\tilde J_\fd^\chi$ satisfies \autoref{Prop_BlowUpAlmostComplexStructures_Properties_Fibred} and \autoref{Prop_BlowUpAlmostComplexStructures_Properties_End} with $\fd_\infty(\tilde J_\fd^\chi) \coloneqq \fd$.
  Moreover, it satisfies
  \autoref{Prop_BlowUpAlmostComplexStructures_Properties_Tamed};
  indeed:
  for $\Lambda \gg_\fd 1$,
  \begin{equation*}
    \sqparen{\Omega_{\tilde I,1} + \Lambda \paren{1+\rho^2} \cdot q^* \omega_C}(v,\tilde J_\fd^\chi v) > 0
  \end{equation*}
  for every non-zero $v \in T\sO_{\bP V}(-2)$;
  cf.~\cite[Proof of Lemma 3.6~(5)]{Doan2018a}.

  Evidently,
  for $\epsilon \ll_\fd 1$ and $\fa$ as in \autoref{Prop_BlowUpAlmostComplexStructures_IsAmple},
  $\tilde J_\fd^\chi(\fa)$ satisfies \autoref{Prop_BlowUpAlmostComplexStructures_Properties}.
  The choice of $\epsilon = \epsilon(\fd)$ can be arranged to depend continuously on $\fd \in \SpaceOfRealCauchyRiemannOperators(V)$. 
  Set
  \begin{equation*}
    \SpaceOfBlowUpAlmostComplexStructures
    \coloneqq
    \set{
      \tilde J_\fd^\chi(\fa)
      :
      \fd \in \SpaceOfRealCauchyRiemannOperators(V),
      \fa ~\text{as in \autoref{Prop_BlowUpAlmostComplexStructures_IsAmple} with}~
      \epsilon = \epsilon(\fd)
    }
  \end{equation*}
  and define $\fd_\infty(\tilde J_\fd^\chi(\fa)) \coloneqq \fd$.
  There is an obvious Banach manifold structure on $\SpaceOfBlowUpAlmostComplexStructures$ with respect to which the above assertions hold.
\end{proof}

For the remainder of this section \emph{choose} an incarnation of
$\fd_\infty
\co
\SpaceOfBlowUpAlmostComplexStructures
\to
\SpaceOfRealCauchyRiemannOperators(V)$
as in \autoref{Prop_BlowUpAlmostComplexStructures}.


\subsection{The space of pseudo-holomorphic sections}
\label{Sec_TheSpaceOfPseudoHolomorphicSections}

Set
\begin{equation*}
  \SpaceOfBlowUpAlmostComplexStructures^\times
  \coloneqq
  \SpaceOfBlowUpAlmostComplexStructures\setminus \fd_\infty^{-1}(\sW).
\end{equation*}
The chambered invariant $\BlowUpInvariant$ is extracted from the following space of pseudo-holomorphic sections.

\begin{prop}
  \label{Prop_SpaceOfPseudoHolomorphicSections}
  Let $p > 2$.
  Denote by $\sB \coloneqq W^{1,p}\Gamma\paren{C,\sO_{\bP V}(-2)}$ the Banach manifold of $W^{1,p}$~sections of the fibre bundle $q \co \sO_{\bP V}(-2) \to C$.
  Denote by 
  $\sE \to \SpaceOfBlowUpAlmostComplexStructures^\times \times \sB$
  the Banach space bundle whose fibre over $(J,s) \in \SpaceOfBlowUpAlmostComplexStructures^\times \sB$ is $L^p\Omega^{0,1}\paren{C,s^*T^{\ver}{\sO_{\bP V}(-2)}}$.
  The following hold:
  \begin{enumerate}
  \item
    \label{Prop_SpaceOfPseudoHolomorphicSections_Transverse}
    The section $\delbar$ of $\sE$ defined by $\delbar(J,s) \coloneqq \delbar_Js$ intersects the zero section transversely;
    in particular, $\sM \coloneqq \delbar^{-1}(0)$ is a Banach submanifold.
  \item
    \label{Prop_SpaceOfPseudoHolomorphicSections_Index}
    The map
    $\pi \co \delbar^{-1}(0) \to \SpaceOfBlowUpAlmostComplexStructures^\times$ is Fredholm of index $0$.
  \end{enumerate}
\end{prop}

\begin{proof}
  If $(J,s) \in \delbar^{-1}(0)$,
  then $\im s$ must intersect $\set{\tilde\rho^2< 1} \subset \sO_{\bP V}(-2)$;
  otherwise: $s$ defines a $\check J_\fd$--holomorphic section of $V^\times/\set{\pm 1}$ for $\fd \coloneqq \fd_\infty(J)$,
  contradicting $J_\fd \notin \sW$.  
  Therefore and by \autoref{Prop_BlowUpAlmostComplexStructures}~\autoref{Prop_BlowUpAlmostComplexStructures_IsAmple},
  the transversality argument from the proof of \cite[Theorem 8.3.1]{McDuff2012} carries over to prove \autoref{Prop_SpaceOfPseudoHolomorphicSections_Transverse}.

  The Snake Lemma implies that $\ind T_{J,s}\pi$ agrees with the index of a real Cauchy--Riemann operator
  \begin{equation*}
    \fd \co
    W^{1,2}\Gamma\paren{C,s^*T^\ver\sO_{\P V}(-2)}
    \to
    L^2\Omega^{0,1}\paren{C,s^*T^\ver\sO_{\P V}(-2)}.
  \end{equation*}
  By Riemann--Roch and \autoref{Prop_FirstChernClassOfVerticalTangentBundle}, and 
  since
  $\rk_\C T^\ver\sO_{\bP V}(-2) = \rk_\C V = 2$ and
  $2 \deg V + \rk_\C V \cdot \chi(C) = 0$,
  \begin{align*}
    \ind \fd
    &=
      2 \deg \paren{s^*T^\ver\sO_{\bP}(-2)} + \rk_\C V \cdot \chi(C) \\
    &=
      2(\rk_\C V-2) \cdot \deg(s) = 0.
  \end{align*}
  This proves \autoref{Prop_SpaceOfPseudoHolomorphicSections_Index}.
\end{proof}

The index bundle of $\pi \co \sM \to \SpaceOfBlowUpAlmostComplexStructures^\times$ is oriented in the standard fashion; cf.~\cite[Theorem 8.3.1]{McDuff2012}.


\subsection{\texorpdfstring{$(\pi,\deg)$}{(pi,deg)} is essentially proper}
\label{Sec_PiDegIsEssentiallyProper}

\begin{prop}
  \label{Prop_PiDegIsEssentiallyProper}
  There are a Banach manifold $\sS$ and a Fredholm map $\sigma \co \sS \to \SpaceOfBlowUpAlmostComplexStructures^\times$ of index at most $-2$
  such that
  \begin{equation*}
    (\pi,\deg) \co \sM\setminus\pi^{-1}(\im\sigma) \to \paren{\SpaceOfBlowUpAlmostComplexStructures^\times\setminus \im\sigma} \times \Z
  \end{equation*}
  is proper.
\end{prop}

There are two potential source of non-properness of $(\pi,\deg)$:
the formation of bubbles and
the non-compactness of $\sO_{\bP V}(-2)$.
Here is how to deal with the former.

\begin{prop}
  \label{Prop_PiDegMaxRho2IsEssentiallyProper}
  Define $\max \tilde\rho^2 \co \sM \to [0,\infty)$ by $\max \tilde\rho^2(s) \coloneqq \max_C \tilde\rho^2 \circ s$.
  There are a Banach manifold $\sS$ and a Fredholm map $\sigma \co \sS \to \SpaceOfBlowUpAlmostComplexStructures^\times$ of index at most $-2$
  such that
  \begin{equation*}
    (\pi, \deg, \max\tilde\rho^2) \co \sM\setminus\pi^{-1}(\im\sigma) \to \paren{\SpaceOfBlowUpAlmostComplexStructures^\times\setminus \im\sigma} \times \Z \times [0,\infty)
  \end{equation*}
  is proper.
\end{prop}

\begin{proof}
  Let $d \in \Z$ and $R > 0$.
  Let $(J_n,s_n) \in \sM^\N$ with $J_n \to J_\infty$, $\deg s_n = d$ and $\max_C \rho^2 \circ s \leq R$.
  By \autoref{Prop_BlowUpAlmostComplexStructures}~\autoref{Prop_BlowUpAlmostComplexStructures_Properties_Tamed},
  for every $n \in \N\cup\set{\infty}$
  there is a symplectic form $\omega_n$ taming $J_n$ on $\set{ \tilde\rho^2 \leq R } \subset \sO_{\bP V}(-2)$ with $\omega_n \to \omega_\infty$.  
  Therefore, by Gromov compactness,
  a subsequence of $(s_n)$ converges (a priori) to a stable nodal $J_\infty$--holomorphic map
  \begin{equation*}
    s_\infty \vee \bigvee_{\alpha \in T} b_\alpha \co C \vee \bigvee_{\alpha \in T} \CP^1 \to \sO_{\bP V}(-2).
  \end{equation*}
  Here $s_\infty$ is a $J_\infty$--holomorphic section of $q \co \sO_{\P V}(-2) \to C$,
  and every bubble $b_\alpha$ of the bubble tree $T$ maps into a fibre of $q$.

  The appearance of a non-trivial bubble tree is a codimension two phenomenon.
  This is a consequence of the discussion in \cite[\S 8.4 and \S 8.5]{McDuff2012}.
  The salient point is that the existence of a (simple) $J$--holomorphic sphere in a fibre of $q$ is a codimension zero phenomenon and, therefore,
  such a sphere intersecting a $J$--holomorphic section of $q$ is a codimension two phenomenon.
\end{proof}

\begin{prop}
  \label{Prop_MaxRho2IsBounded}
  Let $d \in \Z$.
  If $(J_n,s_n) \in \sM$ with $\deg s_n = d$ and $J_n \to J$,
  then
  \begin{equation*}
    \sup_{n \in \N} \max_C \tilde\rho^2 \circ s_n < \infty.
  \end{equation*}
\end{prop}

\begin{proof}
  If not,
  then, after passing to a subsequence,
  \begin{equation*}
    R_n^2 \coloneqq \max_C \tilde\rho^2 \circ s_n \to \infty.
  \end{equation*}
  The following five steps lead to a contradiction.
  
  \begin{step}
    \label{Pf_Prop_MaxRho2IsBounded_Step1}
    There is a proper $\check J_\fd$--holomorphic map $u \co C^\times \to V^\times/\set{\pm 1}$ satisfying
    \begin{equation*}
      \max_{C^\times} \check\rho^2 \circ u = 1
    \end{equation*}
    and such that,
    after passing to a subsequence,
    for every $\alpha \in \Omega_c^2(V^\times/\set{\pm 1})$
    \begin{equation*}
      \int_{C^\times} u^*\alpha
      =
      \lim_{n \to \infty}
      \int_C s_n^*(R_n^{-1})^*\alpha.
    \end{equation*}
  \end{step}

  For every $n \in \N$,
  $T_n \in \Hom(\Omega_c^2(V^\times/\set{\pm 1}),\R)$ defined by
  \begin{equation*}
    T_n(\alpha)
    \coloneqq
    \int_C s_n^*(R_n^{-1})^*\alpha
  \end{equation*}
  is a $(R_n)_*J_n$--holomorphic cycle in $V^\times/\set{\pm 1}$.
    
  By    \autoref{Prop_BlowUpAlmostComplexStructures}~\autoref{Prop_BlowUpAlmostComplexStructures_Properties_Tamed},
  for $\Lambda \gg 1$, $n \in \N$,
  \begin{equation*}
    \omega_n
    \coloneqq
      R_n^{-2} R_n^*\sqparen{\Omega_{\tilde I_n,1} + \Lambda (1+R_n^2) \cdot q^* \omega_C}
  \end{equation*}
  is symplectic and tames $(R_n)_*J_n$
  on $\set{ \tilde\rho^2 < 2 } \subset \sO_{\P V}(-2)$.
  By \autoref{Prop_BlowUpAlmostComplexStructures}~\autoref{Prop_BlowUpAlmostComplexStructures_Properties_End},
  $(R_n)_*J_n$ converges to $\check J_\fd$ with $\fd \coloneqq \fd_\infty(J_\infty)$ on compact subsets of $V^\times/\set{\pm 1}$.
  By
  \autoref{Prop_FibrewiseKahlerForm}~\autoref{Prop_FibrewiseKahlerForm_T=0Standard} and \autoref{Prop_FibrewiseKahlerForm_Scaling},
  $\omega_n$ converges to a symplectic form $\omega_\infty$ on compact subsets of $\set{ \check\rho^2 < 2 }$;
  moreover: $\omega_\infty$ tames $\check J_\fd$.
  Since
  \begin{equation*}
    \sup_{n \in \N} \bM(T_n) \leq \sup_{n \in \N} \int_C s_n^*(R_n^{-1})^*\omega_n < \infty,
  \end{equation*}
  by \autoref{Thm_FedererFlemming},
  after passing to a subsequence,
  $(T_n)$ converges to a $\check J_\fd$--holomorphic cycle $T \in \Hom(\Omega_c^2(V^\times/\set{\pm 1}),\R)$.

  Since $\max_{\supp T_n} \check\rho^2 = 1$,
  by the monotonicity formula \cite[Lemma 5.6]{Doan2018a},
  $\max_{\supp T} \check\rho^2 = 1$.
  The assertion follows from \autoref{Lem_EveryPseudoHolomorphicCycleArisesFromAPseudoHolomorphicMap}.

  \begin{step}
    \label{Pf_Prop_MaxRho2IsBounded_Step2}
    The holomorphic map $\iota \coloneqq \check p \circ u \co C^\times \to C$ is open.
  \end{step}

  Since $\iota$ is holomorphic,
  if its fails to be open,
  then it is constant on a component of $C^\times$.
  Therefore,
  by \autoref{Pf_Prop_MaxRho2IsBounded_Step1},
  $u$ exhibits a bounded proper holomorphic map from this component to $\C^2/\set{\pm 1}$.
  This contradicts the maximum principle.

  \begin{step}
    \label{Pf_Prop_MaxRho2IsBounded_Step3}
    The holomorphic map $\iota \coloneqq \check p \circ u \co C^\times \to C$ is an embedding;
    in particular: $\iota$ identifies $C^\times$ with the open subset $\im \iota \subset C$.
  \end{step}

  By \autoref{Pf_Prop_MaxRho2IsBounded_Step2},
  $\iota \co C^\times \to \im \iota$ is a finite branched covering map.
  If $z \in \im \iota$ is not a branch point,
  then it admits an open neighborhood 
  $U \subset \im\iota$ such that
  $\iota \co \iota^{-1}(U) \to U$ is a covering map and
  $K \coloneqq \overline{\iota^{-1}(U)} \subset V^\times/\set{\pm 1}$ is compact.  
  Choose $\eta \in \Omega_c^2(U)$ with $\int_U \eta = 1$,
  and $\chi \in C_c^\infty(V^\times/\set{\pm 1})$ with $\chi|_K = 1$.
  By \autoref{Pf_Prop_MaxRho2IsBounded_Step1} and since $s_n$ is a section of $q$,
  \begin{align*}
    \deg_z \iota
    &=
      \int_{C^\times} u^*\check p^*\eta
      = \int_{C^\times} u^*\paren{\chi \cdot \check p^*\eta}
      = T(\chi \cdot \check p^*\eta) \\
    &=
      \lim_{n \to \infty} \int_C s_n^*(R_n^{-1})(\chi \cdot \check p^*\eta) 
      =
      \lim_{n \to \infty} \int_C s_n^*(R_n^{-1})^*q^*\eta \\
    &=
      \int_C \eta = 1.
  \end{align*}
  Therefore, $\deg \iota = 1$ and $\iota$ is an embedding.

  \begin{step}
    \label{Pf_Prop_MaxRho2IsBounded_Step4}
    The map $\check \rho^2 \circ u \co C^\times\to (0,1]$ extends to a continuous function $C \to [0,\infty]$ vanishing on $B \coloneqq C \setminus C^\times$.
  \end{step}

  Since $u$ is proper,
  for every $\epsilon > 0$,
  $K_\epsilon \coloneqq \set{ \check \rho^2 \circ u \geq \epsilon } \subset C^\times \subset C$ is compact;
  hence:
  $B \cup \set{ \check \rho^2 \circ u < \epsilon } = C \setminus K_\epsilon$ is open.
  
  \begin{step}
    \label{Pf_Prop_MaxRho2IsBounded_Step5}
    $\fd$ fails to be $2$--rigid: the desired contradiction.
  \end{step}

  Set $B \coloneqq C\setminus C^\times$.
  By the discussion in \autoref{Sec_ThreePerspectivesOn2Rigidity},
  $u$ defines a Euclidean line bundle $\fl$ over $C \setminus B$ and
  a non-zero $s \in \ker \fd^\fl$.
  If $B$ was finite,
  then $\fl$ is (or extends to) a ramified Euclidean line bundle over $C$;
  therefore: $\fd$ fails to be $2$--rigid.
  Indeed, $B$ is finite because $\abs{s}$ extends to a continuous map on $C$ vanishing on $B$.
  This is proved as \autoref{Prop_ZeroLocusOfTwistedHolomorphicSectionIsFinite} below.
\end{proof}

\begin{proof}[Proof of \autoref{Prop_PiDegIsEssentiallyProper}]
  This is an immediate consequence of \autoref{Prop_PiDegMaxRho2IsEssentiallyProper} and \autoref{Prop_MaxRho2IsBounded}.
\end{proof}


\subsection{A twist on Radó's theorem}
\label{Sec_Rado}

\begin{prop}
  \label{Prop_ZeroLocusOfTwistedHolomorphicSectionIsFinite}
  Let $\fd$ be a real Cauchy--Riemann operator on $V$.
  Let $B \subset C$ be a closed subset.
  Let $\fl \to C \setminus B$ be a Euclidean line bundle.
  If $s \in \Gamma(C \setminus B, V \otimes \fl)$
  is non-vanishing,
  satisfies $\fd^\fl s = 0$, and
  $\abs{s}$ extends to a continuous map on $C$ vanishing on $B$,
  then $B$ is finite.
\end{prop}

Since no a priori control on the size of $B$ is assumed,
this cannot be derived via the usual removable singularity methods
(at least those known to the authors).
Instead,
the proof relies on the following classical result in complex analysis.

\begin{theorem}[{\citet{Rado1924}, \citet[Satz 1]{BehnkeStein1951}, \citet{Cartan1952}; \cites{Erhard1956}[Theorem 12.14]{Rudin1987}}]
  \label{Thm_RadoBehnkeSteinCartan}
  Let $C$ be a Riemann surface.
  If $f \co C \to \C$ is a continuous map which is holomorphic on $C \setminus f^{-1}(0)$,
  then $f$ is holomorphic.
  \qed
\end{theorem}

\begin{proof}[Proof of \autoref{Prop_ZeroLocusOfTwistedHolomorphicSectionIsFinite}]  
  If $\fd=\delbar$ is a Dolbeault operator,
  then $s \otimes_\C s \in \Gamma(C\setminus B,V \otimes_\C V)$ is holomorphic,
  and extends to a continuous section of $V \otimes_\C V$ vanishing on $B$.
  By \autoref{Thm_RadoBehnkeSteinCartan},
  $s \otimes_\C s$ is holomorphic.
  Therefore, $\abs{s}^{-1}(0) = \paren{s \otimes_\C s}^{-1}(0)$ is finite.
    
  If $\fd$ is a real Cauchy--Riemann operator,
  then the assertion follows from the above and the Carleman similarity principle \cite[\S 2.3]{McDuff2012}.
  Decompose $\fd = \delbar + \fa $ as in \autoref{Rmk_RealCauchyRiemanOperator~>DolbeaultOperator} and 
  define $\tilde a \in L^\infty(C,\Hom_\C(V,V \otimes_\C \overline{K}_C))$ by
  \begin{equation*}
    \tilde \fa(x) \coloneqq
    \begin{cases}
      \fa s \cdot \abs{s}^{-2}\Inner{s,\cdot} & \textnormal{if } x \notin B, \\
      0 & \textnormal{if } x \in B.
    \end{cases}
  \end{equation*}
  By construction $\paren{\delbar + \tilde\fa}^\fl s = 0$.
  Let $p > 2$.
  By \cite[Step 3 in the proof of Theorem 2.3.5]{McDuff2012},
  for every $x \in C$ there are $\delta > 0$ and a $g \in W^{1,p}\paren{B_\delta(x),\End_\C(V)}$ satisfying $g (\delbar+\tilde \fa) g^{-1} = \delbar$;
  in particular: $\delbar^\fl (g s) = 0$.
  
  Therefore, $\abs{s}^{-1}(0) \supset B$ is discrete;
  hence: finite.
\end{proof}


\subsection{Construction of $\BlowUpInvariant$}
\label{Sec_ConstructionOfBlowUpInvariant}

The construction of $\BlowUpInvariant$ is now straight-forward;
cf.~\autoref{Ex_InvariantsFromProperFredholmMaps}.
By
\autoref{Prop_SpaceOfPseudoHolomorphicSections} and
\autoref{Prop_PiDegMaxRho2IsEssentiallyProper},
if $J \in \SpaceOfBlowUpAlmostComplexStructures \setminus \fd_\infty^{-1}(\sW)$
is a regular value of $\pi$ and $J \notin \im \sigma$,
then for every $d \in \Z$
\begin{equation*}
  \BlowUpModuliSpace(J,d) \coloneqq (\pi,\deg)^{-1}(J,d)
\end{equation*}
is a compact oriented $0$--manifold.
By Sard--Smale, the subset of these $J$ is comeager.
Also by \autoref{Prop_SpaceOfPseudoHolomorphicSections},
\autoref{Prop_PiDegMaxRho2IsEssentiallyProper},
and Sard--Smale,
the oriented bordism class
\begin{equation*}
  [\BlowUpModuliSpace(J,d)] \in \Omega_0^{\SO}
\end{equation*}
only depends of the path-component of $\SpaceOfBlowUpAlmostComplexStructures \setminus \fd_\infty(\sW)$.
Since $\Omega_0^{\SO} = \Z$,
therefore, there is an invariant
\begin{equation*}
  \widetilde\BlowUpInvariant \in \rH^0(\SpaceOfBlowUpAlmostComplexStructures \setminus \fd_\infty(\sW);\Z[[x]])
\end{equation*}
with
\begin{equation*}
  \widetilde\BlowUpInvariant(J)
  = \sum_{d \in \N_0} \# \BlowUpModuliSpace(J,d) \cdot x^d.
\end{equation*}

Since $\fd_\infty$ has path-connected fibres,
\begin{equation*}
  \fd_\infty^*
  \co
  \rH^0(\SpaceOfRealCauchyRiemannOperators(V) \setminus \sW;\Z[[x]])
  \to
  \rH^0(\SpaceOfBlowUpAlmostComplexStructures \setminus \fd_\infty^{-1}(\sW);\Z[[x]])
\end{equation*}
is an isomorphism.
Therefore,
$\widetilde\BlowUpInvariant$ descends to the desired $\BlowUpInvariant$.

\begin{remark}
  If $d \in -\N$,
  then $[\BlowUpModuliSpace(J,d)] = 0$.
  This can be proved by a suitable choice of $J$.
\end{remark}

\begin{remark}
  There is a variant of the above construction for $\rk_\C V > 2$.
  The essential difference is that $\pi$ in \autoref{Prop_SpaceOfPseudoHolomorphicSections}~\autoref{Prop_SpaceOfPseudoHolomorphicSections_Index} is Fredholm of index $2(\rk_\C V-2) \cdot \deg(s)$ at $(J,d)$.
  As a consequence,
  the variant of $\BlowUpInvariant$ is constructed using $\BlowUpModuliSpace(J,0)$ only and, therefore, $\Z$--valued.
\end{remark}



\section{Chambered invariants from gauge theory}
\label{Sec_ChamberedInvariantsFromQuaternionicVortexEquations}

This section constructs chambered invariants of real Cauchy–Riemann operators by counting solutions to the \defined{quaternionic vortex equations}: gauge-theoretic equations on Riemann surfaces associated with quaternionic representations. 
These equations arise from a dimensional reduction of generalized Seiberg–Witten equations introduced in \cites{Taubes1999b, Pidstrygach2004} and studied in \cites{Taubes2012, Taubes2016, Haydys2014, Doan2017a, Doan2017c, Walpuski2019}.

We focus on the \defined{$(r,k)$ ADHM vortex equations} associated with the quaternionic representations appearing in the ADHM construction of instantons on $\R^4$. 
Non-compactness phenomena for the $(r,k) = (2,1)$ and $(r,k) = (1,2)$ versions of the equations are closely related to $2$–rigidity of real Cauchy–Riemann operators. 
These equations are used to construct chambered invariants from \autoref{Thm_ADHMCurvesInvariants} and \autoref{Thm_ADHMGaugeInvariants}.


\subsection{Quaternionic vortex equations}
\label{Sec_QuaternionicVortexEquation}

\begin{definition}
	Let $G$ be a compact connected Lie group. 
	Let $V$ a quaternionic vector space equipped with a Euclidean inner product with respect to which unit quaternions act by isometries. 
	Denote by $\Sp(V)$ the group of quaternion-linear isometries of $V$. 
A \defined{quaternionic representation} of $G$ on $V$ is a homomorphism $\rho \colon G \to \Sp(V)$. 
\end{definition}

Let $\fg$ be the Lie algebra of $G$. 
By abuse of notation we denote the action of $\fg$ on $V$ by $\rho$.
Choose a $G$--invariant inner product of $\fg$.

\begin{definition}
	Let $\rho \colon G \to \Sp(V)$ be a quaternionic representation.
	Define moment maps $\mu_i, \mu_j, \mu_k \colon V \to \fg$  by
	\begin{gather*}
	  \langle \mu_i(\phi), \xi \rangle =  \langle i\rho(\xi)\phi,\phi \rangle, \\
	  \langle \mu_j(\phi), \xi \rangle =  \langle j\rho(\xi)\phi,\phi \rangle, \\
	  \langle \mu_k(\phi), \xi \rangle =  \langle k\rho(\xi)\phi,\phi \rangle
	\end{gather*}
	for all $\phi \in V$ and $\xi \in \fg$.
	The $\Sp(1)$--equivariant \defined{hyperkähler moment map} is
	\begin{equation*}
	  \mu = (\mu_i,\mu_j,\mu_k) \colon V \to \fg \otimes \R^3.
	\end{equation*}
	The \defined{hyperkähler quotient} of $V$ by $G$ is the hyperkähler variety defined by
	\begin{equation*}
	  X = \mu^{-1}(0) / G. \qedhere
	\end{equation*}
\end{definition}

The complex structures $i, j, k$ and the Euclidean metric induce three non-degenerate $2$--forms $\omega_i, \omega_j, \omega_k$ on $V$. 
From now on, we think of $V$ as a complex vector space with a distinguished complex structure given by multiplication by $i$.
The remaining complex structures $j$ and $k$ are equivalent to an $i$--bilinear $2$--form
\begin{equation*}
  \Omega = \omega_j + i \omega_k.
\end{equation*}
With respect to the splitting $\R^3 = \R\oplus \C$,  
\begin{equation*}
  \mu = \mu_\R \oplus \mu_\C
\end{equation*}
where $\mu_\R = \mu_i$ is the \defined{real moment map} and $\mu_\C = \mu_j + i\mu_k$ is the \defined{complex moment map}.
The latter is the moment map for the action of the complex group $G_\C$ on the complex symplectic vector space $(V,i,\Omega)$.

In what follows, consider the following situation.

\begin{hypothesis}
  \label{Hyp_SpinGRepresentation}
  Let $\rho \colon G \to \Sp(V)$ be a quaternionic representation. 
  Assume that the center of $G$ contains a subgroup isomorphic to $\{\pm 1 \}$ such that $\rho(-1) = -\id$. 
\end{hypothesis}

In the situation of \autoref{Hyp_SpinGRepresentation}, let 
\begin{equation*}
  \Spin^G(2) = \Spin(2)\times_{\{\pm 1 \}} G,
\end{equation*}  
so that $\rho$ extends to a unitary representation
\begin{equation*}
  \rho \colon \Spin^G(2) \to \U(V) 
\end{equation*}
by
\begin{equation*}
  \rho( [\lambda, g]) v = \lambda \rho(g) v.
\end{equation*}
We have isomorphisms $\Spin(2) \cong \U(1)$ and $\SO(2) \cong \U(1)$ under which $\Spin(2) \to \SO(2)$ is given by $z \mapsto z^2$. 
Let $L = \C$ be the standard representation of $\U(1)$, so that the induced representation of $\Spin(2)$ is $L^2$. 
The maps $\Omega$ and $\mu_\C$ are $\Spin^G(2)$--equivariant maps
\begin{equation*}
  \Omega \colon \Lambda^2_\C V \to L^2 \qandq \mu_\C \colon V \to L^2 \otimes_\C \fg^c.
\end{equation*}

Let $C$ be a surface equipped with a Riemannian metric. 
Let $\fw \to C$ be a principal $\Spin^G(2)$ bundle together with an isomorphism of the frame bundle of $K_C$, the canonical bundle of $C$, with the principal $\U(1)$ bundle induced by $\fw$ and homomorphism $\Spin^G(2) \to \U(1)$.
Denote by $P$ the $G/\{\pm 1 \}$--principal bundle associated with $\fw$ and the projection $\Spin^G(2) \to G/\{\pm 1 \}$.
Let
\begin{equation*}
  \bV = \fw \times_{\Spin^G(2)} V
\end{equation*}
be the associated vector bundle with complex structure induced by $i$. 
The algebraic structures associated with $\rho$ gives rise to the following bundle structures: the real and complex moment maps
\begin{equation*}
  \mu_\R \colon \bV \to \Ad P \quad\text{and}\quad \mu_\C \colon \bV \to K_C \otimes_\C \Ad P_\C,
\end{equation*}
where $\Ad P_\C$ is the complexification of $\Ad P$, 
and the complex symplectic form
\begin{equation*}
  \Omega \colon \Lambda^2_\C \bV \to K_C,
\end{equation*}
which induces an isomorphism $\SerreOperator \colon \bV \cong \bV^\dagger = \bV^* \otimes_\C K_C$ as in \autoref{Sec_SelfAdjointRealCauchyRiemannOperators}. 

For every connection $A$ on $P$ there is a unique connection on $\fw$ such that the induced connections on $K_C$ and $P$ are the Levi--Civita connection and $A$. 
Thus, for every $A$ we have the induced Cauchy--Riemann operator
\begin{equation*}
  \delbar_A \colon \Gamma(\bV) \to \Omega^{0,1}(C,\bV).
\end{equation*}
This operator is self-adjoint with respect to $\Omega$ in the sense of \autoref{Def_RealCauchyRiemannOperator_SelfAdjoint}:
\begin{equation*}
  \delbar_A^\dagger = \delbar_A^\SerreOperator
\end{equation*}
or equivalently
\begin{equation*}
  \int_C \Omega(\delbar_A \Phi \wedge \Psi) = \int_C \Omega(\Phi \wedge \delbar_A \Psi) \quad\text{for all } \Phi,\Psi \in \Gamma(\bV).
\end{equation*}

\begin{definition}
  The \defined{quaternionic vortex equation} for a connection $A$ on $P$ and a section $\Psi \in \Gamma(\bV)$ is
  \begin{equation}
    \label{Eq_QuaternionicVortex}
    \begin{split}
      \delbar_A\Psi =0, \\
      \mu_\C(\Psi) = 0, \\
      F_A + \ast \mu_\R(\Psi) = 0.
    \end{split}
  \end{equation}
  Let $\sA(P)$ be the space of connections on $P$.
  The gauge group $\sG(P)$ of $P$ acts on $\sA(P)\times \Gamma(\bV)$ preserving the quaternionic vortex equations.
\end{definition}

\begin{remark}
  \label{Rem_ChernSimonsFunctional}
  Solutions to the quaternionic vortex equations are critical points of a holomorphic functional on $\sA(P) \times \Gamma(\bV)$. 
  The space of connections $\sA(P)$ is an affine space over $\Omega^1(C,\Ad P) = \Omega^{0,1}(C, \Ad P_\C)$ where $P_\C$ is the complexification of $P$. 
  This makes the configuration space $\sA(P) \times \Gamma(\bV)$ into an infinite-dimensional Kähler manifold with a Hamiltonian action of the gauge group $\sG(P)$.
  The moment map for this action is
  \begin{gather*}
    m \colon \sA(P)\times\Gamma(\bV) \to \Gamma(\Ad P), \\
    m(A,\Psi) = \ast F_A + \mu_\R(\Psi).
  \end{gather*}
  We also have the action of the complex gauge group $\sG_\C(P)$ which preserves the complex structure but not the symplectic form on $\sA(P) \times \Gamma(\bV)$. 
  The holomorphic Chern--Simons functional 
  \begin{gather*}
    \sF \colon \sA(P) \times\Gamma(\bV) \to \C, \ \\
    \sF(A,\Psi) = \int_C \Omega( \delbar_A \Psi \wedge \Psi)
  \end{gather*}
  is invariant under the action of $\sG_\C(P)$. 
  The equation for critical points is
  \begin{gather*}
    \delbar_A \Psi = 0, \\
    \mu_\C(\Psi) = 0.
  \end{gather*}
  Therefore, the quaternionic vortex equation can be interpreted as the critical point equation for $\sF$, together with the moment map equation $m(A,\Psi) = 0$. 
  The Hitchin--Kobayashi correspondence implies that the $\sG(P)$ orbits of solutions correspond to $\sG_\C(P)$ orbits of solutions to the critical point equation satisfying an appropriate stability condition.
\end{remark}

\subsection{Generalizations}

The main results on counting solutions to quaternionic vortex equations \autoref{Thm_ChamberedInvariantsFromADHM12} and 
\autoref{Thm_ChamberedInvariantsFromADHM21}, concern certain modifications to these equations.
The first modification is to allow what physicists refer to as flavor symmetry. 

\begin{hypothesis}
  \label{Hyp_FlavorSymmetry}
  Let $H$ be a compact Lie group.
  Assume that $G$ is a normal subgroup of $H$ and that $\rho \colon G \to \Sp(V)$ extends to a quaternionic representation $\rho \colon H \to \Sp(V)$ satisfying \autoref{Hyp_SpinGRepresentation}. 
\end{hypothesis}

In this situation, $G$ is called the \defined{gauge symmetry group} and $I = H/G$ the \defined{flavor symmetry group}.
As before, extend $\rho \colon H \to \Sp(V)$ to a unitary representation
\begin{equation*}
  \rho \colon \Spin^H(2) \to \U(V).
\end{equation*}
Let $\fw \to C$ be a principal $\Spin^H(2)$ bundle together with an isomorphism between the frame bundle of $K_C$ and the principal $\U(1)$ bundle induced by $\fw$ and homomorphism $\Spin^H(2) \to \U(1)$. 
Define 
\begin{equation*}
  \bV = \fw \times_{\Spin^H(2)} V.
\end{equation*}
We have a short exact sequence
\begin{equation}
  \label{Eq_GaugeFlavorExactSequence}
  \begin{tikzcd}
    1 \ar{r} & G \ar{r} & H \ar{r} & I \ar{r} & 1
  \end{tikzcd}
\end{equation}
so that $\fw$ induces a principal $I$--bundle $Q \to C$.
Fix a connection on $Q$. 
If $H = G \times I$ is a product, then there is a $G/\{\pm 1\}$ principal bundle $P \to C$ associated with $\fw$. 
In that case, given a connection $A \in \sA(P)$, there is a unique connection on $\fw$ which induces $A$ on $P$, the fixed connection on $I$, and the Levi--Civita connection on $K_C$. 
This connection induces a twisted Dolbeault operator
\begin{equation*}
  \delbar_A \colon \Gamma(\bV) \to \Omega^{0,1}(C, \bV),
\end{equation*}
and we can consider the quaternionic vortex equation \autoref{Eq_QuaternionicVortex} with respect to this operator.

In general, if \autoref{Eq_GaugeFlavorExactSequence} does not split, such a bundle $P$ might not exist. 
However, we can still consider the space of $\Spin^H(2)$ connections on $\fw$ which induce the given connection on $Q$ and the Levi--Civita connection on $K_C$. 
By abuse of notation, denote this space by $\sA(P)$ as before.
Moreover, the associated bundle of Lie algebras $\Ad(P)$ and the group of $G$--gauge transformations $\sG(P)$ still exist. 
With this adjustment of notation, we define $\delbar_A$ and consider \autoref{Eq_QuaternionicVortex} in the general case.

The second modification is to allow zeroth order perturbations of the Cauchy--Riemann operator.
For the discussion of compactness, it is important to restrict ourselves to perturbations which are compatible with the hyperkähler quotient construction. 

\begin{definition}
  \label{Def_ZerothOrderPerturbation}
  Let $\HOMT(\bV,\SerreOperator)$ be the vector spaces modeling the affine space of $\SerreOperator$--self-adjoint real Cauchy--Riemann operators on $\bV$, see \autoref{Prop_RealCauchyRiemanOperator_SelfAdjoint_AffineSpace}. 
   Denote by $\HOMT(\bV,\rho)\subset\HOMT(\bV,\SerreOperator)$ the subbundle of those $a \in \HOMT(\bV,\SerreOperator)$ which commute with $G$--automorphisms of $\bV$ and for every $\Phi \in \bV$ with $\mu(\Phi) =  0$, 
  \begin{equation*}
    a \Phi \in \ker\rd_\Phi\mu \cap (\rho(\fg)\Phi)^\perp. 
  \end{equation*}
  A \defined{homogeneous perturbaton} is a section $\Upsilon \in \Gamma(\HOMT(\bV,\rho))$. 
  Given such $\Upsilon$, define
  \begin{equation*}
    \delbar_{A,\Upsilon} = \delbar_A + \Upsilon \colon \Gamma(\bV) \to \Omega^{0,1}(C,\bV). \qedhere
  \end{equation*}
\end{definition}

\begin{remark}
  We can interpret elements of $\HOMT(\bV,\rho)$ as linear vector fields on $V$ which descend to a  vector field on the hyperkähler quotient $\mu^{-1}(0)/G$. 
\end{remark}

\begin{remark}
  In the situation of \autoref{Hyp_FlavorSymmetry}, varying the $I$--connection on the flavor symmetry bundle $Q$ can be incorporated into a homogeneous perturbation. 
\end{remark}

The third modification is to allow inhomogeneous perturbations.
As in classical Seiberg--Witten theory, such perturbations are useful for ruling out reducible solutions

\begin{definition}
	\label{Def_InhomogeneousPerturbation}
	An \defined{inhomogeneous perturbation} is a pair $\eta = (\eta_\C,\eta_\R)$ where
  \begin{enumerate}
    \item $\eta_\C \colon \sA(P)\times\Gamma(\bV) \to \Omega^{1,0}(C,\Ad P_\C)$ is an $\sG(P)$--equivariant map satisfying
    \begin{equation*}
      \delbar_A (\eta_\C(A,\Psi))= 0,
    \end{equation*}
    \item $\eta_\R \colon \sA(P) \times\Gamma(\bV) \to \Omega^0(C,\Ad P)$ is an $\sG(P)$--equivariant map,
  \end{enumerate} 
  such that the linearizations of $\eta$ and $\tau$ are pseudo-differential operators of order zero. 
\end{definition}

For example, $\eta_\C$ and $\eta_\R$ can be holonomy perturbations, or can be independent of $(A,\Psi)$ and take values in the center of $\fg$.  

\begin{definition}
  \label{Def_PerturbedQuaternionicVortex}
  Assume \autoref{Hyp_SpinGRepresentation} and \autoref{Hyp_FlavorSymmetry}.
  Fix a Riemannian metric on $C$ and a connection on the flavor symmetry bundle $Q$. 
  Let $\Upsilon$ and $\eta$ be homogeneous and inhomogeneous perturbations.
  The \defined{$(\Upsilon,\eta)$--pertubed quaternionic vortex equation} for $A \in \sA(P)$ and $\Psi \in \Gamma(\bV)$ is
  \begin{equation}
    \label{Eq_PerturbedQuaternionicVortex}
    \begin{split}
      \delbar_{A,\Upsilon}\Psi =0, \\
      \mu_\C(\Psi) = \eta_\C(A,\Psi), \\
      \ast F_A +  \mu_\R(\Psi) = \eta_\R(A,\Psi).
    \end{split}
  \end{equation}
\end{definition}

\begin{remark}
  The assumptions that $\Upsilon$ is $\SerreOperator$--self-adjoint and $\eta_\C$ satisfies $\delbar_A \eta_\C = 0$ are technical and probably not necessary.
  However, without these assumptions \autoref{Eq_PerturbedQuaternionicVortex} has to be replaced by more general equations \autoref{Eq_PerturbedQuaternionicVortexZeta}  to make the equations elliptic, see \autoref{Prop_ExtraFieldVanishes} below.
  It is likely that compactness theory for generalized Seiberg--Witten equations can be extended to \autoref{Eq_PerturbedQuaternionicVortexZeta} but we will not consider this question here. 
\end{remark}



\subsection{Moduli spaces of quaternionic vortices}

\begin{definition}
  \label{Def_PerturbedQuaternionicVortexEquation}
  In the situation of \autoref{Def_PerturbedQuaternionicVortex}, the \defined{moduli space of $(\Upsilon,\eta)$--perturbed quaternionic vortices} is
  \begin{equation*}
    \sM(\Upsilon,\eta)= \{ (A,\Psi) \in\sA(P)\times\Gamma(\bV) : (A,\Psi) \text{ satisfies \autoref{Eq_PerturbedQuaternionicVortex}} \} / \sG(P).
  \end{equation*}
  Equip $\sM(\Upsilon,\eta)$ with the topology of $C^\infty$ convergence. 
\end{definition}

\begin{remark}
   \label{Rem_CkRegularity}
   In practice, it is convenient to consider perturbations $\Upsilon$, $\eta_\R$, $\eta_\C$ of lower regularity. 
   In that case, we consider solutions to \autoref{Eq_PerturbedQuaternionicVortex} of lower regularity, such as $C^K$ for some $K \in\N$, and equip the moduli space with the $C^K$ topology.
\end{remark}

Equation \autoref{Eq_PerturbedQuaternionicVortex} arises from gauge theory over the $3$--manifold $C\times S^1$ (this is explained in \autoref{Sec_CompactnessQuaternionicVortexEquation}). 
Therefore, deformation theory of solutions to  \autoref{Eq_PerturbedQuaternionicVortex} should be controlled by an elliptic complex of index zero.
The first and third equation in \autoref{Eq_PerturbedQuaternionicVortex} already form an elliptic system modulo gauge.
However, the index of the linearization of this elliptic is not zero.
To obtain correct deformation theory with index zero, introduce the following extension of  \autoref{Eq_PerturbedQuaternionicVortex}.
Denote by 
\begin{equation*}
	\overline\SerreOperator \colon \bV \to \bV\otimes_\C K_C^*
\end{equation*}
the complex anti-linear map obtained by composing $\SerreOperator \colon \bV \to \bV^*\otimes_\C K_C$ with the anti-linear isomorphism $\bV^*\otimes_\C K_C \to \bV \otimes_\C K_C^*$ given by the Hermitian inner product.
Consider the following equations for a triple $(A,\Psi,\zeta)$ where $\zeta\in\Gamma(\Ad P_\C)$:
\begin{equation}
	\label{Eq_PerturbedQuaternionicVortexZeta}
  \begin{split}
    (\delbar_{A,\Upsilon} + \overline{\SerreOperator}(\rho(\zeta)\Psi) =0, \\
    \del_A \zeta + \mu_\C(\Psi) = \eta_\C(A,\Psi), \\
    \ast F_A +  \mu_\R(\Psi) = \eta_\R(A,\Psi).
  \end{split}
\end{equation}
The next proposition, which follows from a standard calculation in Seiberg--Witten theory, see \autoref{Prop_VanishingExtraField}, shows that the additional field $\zeta$ can be ignored.

\begin{prop}
  \label{Prop_ExtraFieldVanishes}
  If $(A,\Psi,\zeta)$ is a solution to \autoref{Eq_PerturbedQuaternionicVortexZeta}, then 
  \begin{equation*}
    \delbar_A \zeta = 0 \qandq \rho(\zeta)\Psi = 0.
  \end{equation*}
  In particular, if the stabilizer of $(A,\Psi)$ in $\sG$ is trivial, then $\zeta=0$. 
\end{prop}

By linearizing \autoref{Eq_PerturbedQuaternionicVortexZeta} at $(A,\Psi,0)$ modulo the action of $\sG(P)$, we associate with the solution an elliptic deformation complex:
\begin{equation*}
  \begin{tikzcd}
    \Omega^0(\Ad P) \ar{r} & \Omega^0(\Ad P^\C) \oplus \Omega^1(\Ad P) \oplus \Gamma(\bV) \ar{r} &  \Omega^2(\Ad P) \oplus \Omega^{0,1}(\Ad P^\C) \oplus \Omega^{0,1}(\bV).
  \end{tikzcd}
\end{equation*} 
We have real isomorphisms
\begin{equation*}
  \Omega^{0,1}(C,\C) \cong \Omega^1(C,\R) \qandq \Omega^0(C,\C) \cong \Omega^0(C,\R)\oplus\Omega^2(C,\R),
\end{equation*} 
with respect to which the map
\begin{equation*}
  \rd + \rd^* \colon \Omega^1(C,\R) \to \Omega^0(C,\R)\oplus\Omega^2(C,\R)
\end{equation*}
is identified with 
\begin{equation*}
  \delbar^* \colon \Omega^{0,1}(C,\C) \to \Omega^0(C,\C).
\end{equation*}
We also use the real isomorphism $\overline\SerreOperator \colon \Gamma(\bV) \cong \Omega^{0,1}(\bV)$.
Under these identifications, the operator associated with the deformation complex is
\begin{gather*}
  L_{A,\Psi} \colon \Omega^{0,\bullet}(\Ad P_\C) \oplus \Gamma(\bV) \to  \Omega^{0,\bullet}(\Ad P_\C) \oplus \Gamma(\bV), \\
  L_{A,\Psi} =  
  \begin{pmatrix}
    \delbar_A + \delbar_A^* & \Gamma_\Psi^* \\
    \Gamma_\Psi  & \overline\SerreOperator^{-1} \delbar_{A,\Upsilon}
  \end{pmatrix}
  + T_{A,\Psi}\eta
\end{gather*}  
where $\Omega^{0,\bullet} = \Omega^0 \oplus \Omega^{0,1}$, the map $\Gamma_\Psi \colon \Omega^{0,\bullet}(\Ad P_\C) \to \bV $ is defined by
\begin{equation*}
  \Gamma_\Psi(\xi+a) = \rho(\xi)\Psi + \overline\SerreOperator^{-1}(\rho(a)\Psi) \quad\text{for } \xi \in \Ad P_\C \qandq a \in K_C^*\otimes\Ad P_\C,
\end{equation*}
and $T_{A,\Psi}\eta$ is the derivative of $\eta(A,\Psi)$ at $(A,\Psi)$. 
The operator $L_{A,\Psi}$ is a compact perturbation of the self-adjoint elliptic operator
\begin{equation*}
  \begin{pmatrix}
    \delbar_A + \delbar_A^* &  0 \\
     0 & \overline\SerreOperator^{-1} \delbar_{A,\Upsilon}
  \end{pmatrix}.
\end{equation*}
Therefore, $L_{A,\Psi}$ is elliptic and $\ind(L_{A,\Psi}) = 0$. 

\begin{remark}
  An alternative way to obtain the deformation operator is to consider the equation for citical points of the complex Chern--Simons functional modulo the complex gauge group $\sG^\C(P)$.
  (We assume here that $\Upsilon$ is $\C$--linear so that the complex Chern--Simons functional is $\sG^\C(P)$--invariant.)
  A finite-dimensional analog is the deformation complex associated with a critical point of a holomorphic Morse function invariant under the action of a complex Lie group. 
  From this point of view, the deformation complex is
  \begin{equation*}
    \begin{tikzcd}
      \Omega^0(\Ad P_\C) \ar{r} & \Omega^{0,1}(\Ad P_\C) \oplus \Gamma(\bV) \ar{r} & \Omega^{1,0}(\Ad P_\C) \oplus \Omega^{0,1}(\bV) \ar{r} &  \Omega^{1,1}(\Ad P_\C).
    \end{tikzcd}
  \end{equation*}
  The first map is the linearization of the $\sG_\C(P)$ action, the second map is the linearization of the equation with a complex co-gauge condition, and the remaining two maps are complex dual to the first two over $\C$.
  Observe that $\Omega^{1,0}$ is dual to $\Omega^{0,1}$ and $\Omega^0$ is dual to $\Omega^{1,1}$ using the wedge product and integration, and similarly $\Omega^{0,1}(\bV)$ is dual to $\Gamma(\bV)$ by the pairing
  \begin{equation*}
    \alpha \in \Omega^{0,1}(\bV), \quad \beta \in \Gamma(\bV) \mapsto \int_\Sigma \Omega(\alpha\wedge\beta).
  \end{equation*}
  With respect to these identifications, the above complex is self-dual over $\C$.
  Under the isomorphisms
  \begin{equation*}
   \Omega^0(\Ad P)\oplus\Omega^2(\Ad P) \cong \Omega^0(\Ad P_\C), \quad
   \Omega^1(\Ad P) \cong \Omega^{0,1}(\Ad P_\C), \quad
   \Gamma(\bV) \cong \Omega^{0,1}(\bV)
  \end{equation*}
  the operator associated with the above complex agrees with $L_{A,\Psi}$. 
\end{remark}

The following is a standard application of the implicit function theorem and the existence of slices for the action of the gauge group.

\begin{definition}
  A solution $(A,\Psi)$ to \autoref{Eq_PerturbedQuaternionicVortex} is \defined{irreducible} if the stabilizer of $(A,\Psi)$ in $\sG$ is trivial, and \defined{unobstructed} if $\coker L_{A,\Psi} = \{0\})$ (equivalently, $\ker L_{A,\Psi} = \{0\}$). 
\end{definition}

\begin{prop}
  Let $(A,\Psi)$ be a solution to \autoref{Eq_PerturbedQuaternionicVortex} and let $G_{A,\Psi}$ be the stabilizer of $(A,\Psi)$ in $\sG(P)$.
  The operator $L_{A,\Psi}$ is $G_{A,\Psi}$--equivariant and there exist
  \begin{enumerate}
    \item an open neighborhood $U$ of $[A,\Psi]$ in $\sM(\Upsilon,\eta)$,
    \item a $G_{A,\Psi}$--equivariant open neihborhood $U$ of $0$ in $\ker L_{A,\Psi}$,
    \item a $G_{A,\Psi}$--equivariant smooth map $\kappa \colon U \to U$
  \end{enumerate}
  such that $U$ is homeomorphic to $\kappa^{-1}(0)/G_{A,\Psi}$. 
  In particular, if every $[A,\Psi] \in \sM(\Upsilon,\eta)$ is irreducible and unobstructed, then $\sM(\Upsilon,\eta)$ is a discrete space.
\end{prop}

The determinant line bundle of the family of operators $L_{A,\Psi}$ over $\sA(P)\times\Gamma(\bV)$ is $\sG(P)$--equivariantly trivial because the $\sG(P)$--equivariant homotopy
\begin{equation*}
  L_{A,\Psi; t} =  
  \begin{pmatrix}
    \delbar_A + \delbar_A^* & t\Gamma_\Psi^* \\
    t\Gamma_\Psi  & \overline\SerreOperator^{-1}\delbar_A + t \Upsilon
  \end{pmatrix}
  + t T_{A,\Psi}\eta
\end{equation*}  
connects $L_{A,\Psi} = L_{A,\Psi;1}$ with $L_{A,\Psi;0}$, which is the sum of a $\C$--linear and $\C$--antilinear operator. 
For every irreducible unobstructed solution $[A,\Psi] \in \sM(\Upsilon,\eta)$, define $\sign(A,\Psi) \in \{ -1, 1\}$ by orientation transport from $\det L_{A,\Psi;0}$ to $\det L_{A,\Psi;1}$.

\begin{definition}
  Suppose that $\sM(\Upsilon,\eta)$ is compact and all solutions in $\sM(\Upsilon,\eta)$ are irreducible and unobstructed.
  Define
  \begin{equation}
    \label{Eq_CountOfQuaternionicVortices}
    \# \sM(\Upsilon,\eta) = \sum_{[A,\Psi] \in \sM(\Upsilon,\eta)} \sign(A,\Psi).
  \end{equation}
\end{definition}

General Fredholm theory allows us to extend this definition to the case when the moduli space is compact but possibly obstructed; see \cite[Theorem 2.7]{Cieliebak2002}, \cite[Proposition 14]{Brussee1996}.

\begin{prop}
  \label{Prop_ChamberedInvariantsFromGaugeTheory}
  For every $(\Upsilon,\eta)$ such that $\sM(\Upsilon,\eta)$ is compact and consists of irreducible solutions we can associate then $\# \sM(\Upsilon,\eta) \in \Z$ in such a way that:
  \begin{enumerate}
  		\item If all solutions in $\sM(\Upsilon,\eta)$ are unobstructed, then $\# \sM(\Upsilon,\eta)$ agrees with \autoref{Eq_CountOfQuaternionicVortices}.
  		\item If $(\Upsilon_t,\eta_t)_{t\in[0,1]}$ is a continuous path such that $\sM(\Upsilon_t,\eta_t)$ is compact and consists of irreducible solutions for all $t\in[0,1]$, then
  		\begin{equation*}
  			\# \sM(\Upsilon_0,\eta_0) = \# \sM(\Upsilon_1,\eta_1).
		\end{equation*} 
  \end{enumerate}
\end{prop}

In general, the moduli space $\sM(\Upsilon,\eta)$ may be non-compact.
This phenomenon was studied in \cite{Taubes2013,Haydys2014, Walpuski2019}.
The known compactness results assume the following algebraic property of the quaternionic representation $\rho \colon G \to V$.

\begin{hypothesis}
	\label{Hyp_Compactness}
	Given $\Phi \in V$ define $\Gamma_\Psi \colon \Im\H \otimes \fg \to V$ by $\Gamma_\Psi(q\otimes \xi) = q \rho(\xi) V$.
	Suppose that there exist constants $\delta, c > 0$ such that for every $\Psi \in V$ with $|\Psi| = 1$ and $|\mu(\Psi) \leq \delta$
	\begin{equation*}
		|\mu(\Psi)| \leq c |\Gamma_\Psi \mu(\Psi)|.
	\end{equation*} 
\end{hypothesis}

\begin{definition}
	\label{Def_LimitingConfiguration}
	Let $\Upsilon$ be a homogeneous perturbation. 
	An \defined{$\Upsilon$--limiting configuration} is a triple $(B,A,\Psi)$ consisting of 
	\begin{itemize}
		\item a nowhere dense closed subset $B \subset C$, 
		\item a $G$--connection $A$ on $P|_{C\setminus B}$, 
		\item a section $\Psi \in \Gamma(C\setminus B, \bV)$ with $\| \Psi \|_{L^2} =1$ and $\| \nabla_A \Psi \|_{L^2} < \infty$,
	\end{itemize}
	 such that
	\begin{gather*}
		\delbar_{A,\Upsilon}\Psi = 0, \\
		\mu_\C(\Psi) = 0, \\
		\mu_\R(\Psi) = 0,
	\end{gather*} 
	and $|\Psi|$ extends to a Hölder continuous function on $C$ such that $|\Psi|^{-1}(0) = B$. 
\end{definition}

The following is an adaptation of a compactness theorem for generalized Seiberg--Witten equations \cite{Walpuski2019}.
The proof is postponed until \autoref{Sec_CompactnessQuaternionicVortexEquation}. 

\begin{theorem}
	\label{Thm_CompactnessQuaternionicVortexEquations}
	Assume that the quaternionic representation $\rho \colon G \to \Sp(V)$ satisfies \autoref{Hyp_Compactness}.
	Let $(A_n,\Psi_n)$ be a sequence of solutions to the $(\Upsilon,\eta)$--perturbed quaternionic vortex equation such that $|\eta(A_n,\Psi_n)|$, $|\nabla_{A_n} \eta(A_n,\Psi_n)|$ are uniformly bounded. 
	\begin{enumerate}
		\item
		If $\limsup_{n\to\infty} \| \Psi_n \|_{L^2} < \infty$, then after passing to a subsequence and applying gauge transformations $(A_n,\Psi_n)$ converges in $C^\infty$ to a solution $(A,\Psi)$. 
		\item If $\limsup_{n\to\infty} \| \Psi_n \|_{L^2} = \infty$, then there exists an $\Upsilon$--limiting configuration $(B,A,\Psi)$ such that after passing to a subsequence and applying gauge transformations over $C\setminus B$
		\begin{enumerate}
			\item $|\Psi_n| / \| \Psi_n\|_{L^2}$ converges to $|\Psi|$ in the $C^{0,\alpha}$ topology over $C$ for some $\alpha \in (0,1)$,
			\item $A_n$ converges to $A$ in the weak $W^{1,2}$ topology over compact subsets of $C\setminus B$,
			\item $\Psi_n / \|\Psi_n\|_{L^2}$ converges to $\Psi$ in the weak $W^{2,2}$ topology over compact subsets of $C\setminus B$.
		\end{enumerate} 
	\end{enumerate}
\end{theorem}



\subsection{ADHM vortex equations}
\label{Sec_ADHMVortexEquations}

The moduli space $\cM_{r,k}$ of framed instantons on $\R^4$ with structure group $\SU(r)$ and charge $k$ is a hyperkähler orbifold.
The ADHM construction exhibits $\cM_{r,k}$ as a finite-dimensional hyperkähler quotient of a quaternionic representation of $\U(k)$. 
In this section we discuss the quaternionic vortex equations associated with this representations.
These equations were proposed by Haydys--Walpuski in relation to gauge theory on special holonomy manifolds \cite{Haydys2014}; for related discussion see \cite{Doan2017d}.

Let $G = \U(k)$ and
\begin{equation*}
  V = (\H\otimes_\C \End(\C^k)) \oplus (\H\otimes_\C \Hom_\C(\C^r, \C^k)).
\end{equation*}
Define a quaternionic structure on $V$ by left quaternionic multiplication on $\H$. 
Define a quaternionic representation $\rho \colon H \to \Sp(V)$ by the adjoint action of $G$ on $\End(\C^k)$ and the standard action on $\C^k$, and by $K$ acting by right quaternionic multiplication on the first summand, and trivially on the second summand. 
The quaternionic moment map for the $G$ action is
\begin{gather*}
  \mu \colon V \to \R^3 \otimes \fu(k), \\
  \mu(\xi,\Psi) = [\xi\wedge\xi]^+ + (\Psi \Psi^*)_0.
\end{gather*}
The first term combines the wedge product and the Lie bracket on $\fu(k)$ with the projection $\Lambda^2 \R^4 \to \Lambda^+ \R^4 = \R^3$.
The second term involves projecting $\Psi \Psi^*$ on $\su(2) \subset \End(\R^4)$ and identification $\su(2) \iso \Lambda^4 \R^4 = \R^3$ using the Clifford multiplication. 
Let $\Psi = (\alpha,\beta^*)$ under identification $\H = \C^2$, where $\alpha \in \Hom_\C(\C^r, \C^k)$ and $\beta \in \Hom_\C(\C^k,\C^r)$. 
The complex and real moment maps (multiplied by $i$) are
\begin{gather*}
  \mu_\C(\xi,\Psi) = [\xi \wedge \xi] + \alpha \beta,  \\
  \mu_\R(\xi,\Psi) = [\xi, \xi^*] + \alpha\alpha^* - \beta^*\beta,
\end{gather*}
where in the first formula we identify $\Lambda^2 \C^2 = \C$.

\begin{theorem}[Atiyah--Hitchin--Drinfeld--Manin]
  The hyperkähler quotient of $V$ by $G$ is isomorphic to $\cM_{r,k}$ if $r\geq 2$ and to $\Sym^k \H$ if $r=1$. 
\end{theorem}

The ADHM representation has a natural flavor symmetry as in \autoref{Hyp_FlavorSymmetry}.
Let 
\begin{equation*}
  K = \Sp(1)\times \SU(r) \qandq H = G\times K.
\end{equation*}
The action of $G$ extends to an action of $\Spin^H(2)$ with the action of $K$ preserving the quaternionic structure and commuting with $G$.
The action of $\Sp(1)$ on $V$ is by right quaternionic multiplication $q \cdot v \mapsto v q^{-1}$. 

The resulting quaternionic vortex equation \autoref{Eq_QuaternionicVortex} is concretely described as follows.
We start with the following data:
\begin{itemize}
\item a Riemann surface $C$,
\item a Hermitian vector bundle $N \to C$ of rank two, equipped with a $\U(2)$ connection,
\item a covariantly constant isomorphism $\Lambda^2_\C N \cong K_C$,
\item a Hermitian vector bundle $E \to C$ of rank $r$ and trivial determinant line bundle, equipped with an $\SU(r)$ connection, and
\item a Hermitian vector bundle $H \to C$ of rank $k$.
\end{itemize}
The equation involves a $\U(k)$ connection $A$ on $U$ and a section $\Psi$ of
\begin{equation}
  \label{Eq_ADHMSpinorBundle}
  \bV = (N\otimes_\C \End H) \oplus (\Hom_\C(E,H) (\oplus (\Hom_\C(H,E) \otimes K). 
\end{equation}
If $\Psi = (\xi, \alpha, \beta)$, where
\begin{equation*}
  \xi \in \Gamma(N \otimes\End H), \quad \alpha \in \Gamma(\Hom_\C(E,H)), \quad \beta \in \Omega^{1,0}(C, \Hom_\C(H,E))),
\end{equation*}
then the (unperturbed) $(r,k)$ ADHM vortex equation is
\begin{equation}
  \label{Eq_ADHMVortexEquation}
  \begin{split}
    \delbar_A \alpha = 0, \quad \delbar_A \beta = 0, \quad \delbar_A \xi &= 0, \\
    [\xi \wedge \xi] + \alpha\beta &= 0, \\
    i \ast F_A + [\xi, \xi^*] + \alpha\alpha^* - \ast\beta^*\beta  &= 0.
  \end{split}
\end{equation}
In the second equation, $[\xi\wedge\xi]$ combines the wedge product $N\otimes N \to \Lambda^2_\C N$ with the commutator on $\End H$, followed by the isomorphism $\Lambda^2_\C N \cong T^*C$.
In the third equation, $[\xi,\xi^*]$ combines the Hermitian inner product on $N$ with the commutator on $\End H$, so that it takes place in $\Gamma(\End H)$.
The equation is equivariant with respect to the action of the group $\sG(H)$ of unitary gauge transformations of $H$.

Our goal is to consider a perturbation of \autoref{Eq_ADHMVortexEquation} of the form \autoref{Def_PerturbedQuaternionicVortexEquation}, for a suitable class of $\Upsilon$ and $\eta$, and use them to define chambered invariants of $(\Upsilon,\eta)$.
We expect that this is possible for all values of $(r,k)$.
However, the known compactness results for quaternionic vortex equations require the technical assumption from \autoref{Hyp_Compactness}, which is satisfied only for $(r,k) = (1,2)$ and $(r,k)=(k,2)$.
We focus on these two cases.

\subsection{The $(1,2)$ ADHM vortex equation}
\label{Sec_12ADHMVortexEquation}

In this section we define chambered invariants of real Cauchy--Riemann operators using the ADHM vortex equation \autoref{Eq_ADHMVortexEquation} with $r=1$ and $k=2$.

Let $C$ be a compact Riemann surface.
Let $H$ be a Hermitian vector bundle of rank $2$. 
Let $N \to C$ be a Hermitian vector bundle of rank $2$ equipped with an isomorphism $\SerreOperator \colon N \cong N^\dagger$, where $N^\dagger = N \otimes_\C K_C^*$.
Denote by $\SpaceOfRealCauchyRiemannOperators(N,\SerreOperator)$ the space of real Cauchy--Riemann operator which are self-adjoint with respect to $\SerreOperator$.
Every such operator can be written as
\begin{equation*}
  \fd = \delbar_N + \Upsilon
\end{equation*}	
where $\delbar_N$ is a fixed self-adjoint Cauchy--Riemann operator on $N$ and $\Upsilon \in\Gamma(\HOMT(N ,\SerreOperator))$, see \autoref{Prop_RealCauchyRiemanOperator_SelfAdjoint_AffineSpace}.
For every $\Upsilon$ as above and $\U(2)$ connection $A$ on $H$ denote by
\begin{equation*}
  \delbar_{A,\Upsilon} = \delbar_{N,A} + \Upsilon \otimes \id \colon \Gamma(N\otimes_\C \End_\C H) \to \Omega^{0,1}(C, N\otimes_\C \End_\C H)  
\end{equation*}
the induced real Cauchy--Riemann operator on $N \otimes_\C \End_\C H$.

\begin{definition}
  Let $\Upsilon \in \Gamma(\HOMT(N ,\SerreOperator))$ and $\eta = (\eta_\C,\eta_\R) \in \rH^{1,0}(C) \times \R$. 
  The \defined{$(\Upsilon,\eta)$-perturbed $(1,2)$ ADHM vortex equation} for
  \begin{equation*}
    A \in \sA(H), \quad \xi \in \Gamma(N\otimes_\C \End_\C H), \quad \alpha \in \Gamma(H), \quad \beta \in \Omega^{1,0}(H^*)
  \end{equation*} 
  reads
  \begin{equation}
    \label{Eq_ADHM12}
    \begin{split}
      \delbar_A \alpha = 0, \quad \delbar_A \beta = 0, \quad \delbar_{A,\Upsilon} \xi &= 0, \\
      [\xi \wedge \xi] + \alpha\beta &= \eta_\C \otimes_\C \id_H, \\
      i \ast F_A + [\xi, \xi^*] + \alpha\alpha^* - \ast\beta^*\beta  &= (\eta_\R + \pi\deg(H))\id_H. 
    \end{split}\qedhere
  \end{equation}
\end{definition}

\begin{remark}
  This is the ADHM vortex equation \autoref{Eq_ADHMVortexEquation} for $(r,k) = (1,2)$ with homogeneous and inhomogeneous pertubations, as in \autoref{Eq_PerturbedQuaternionicVortex}.
  We consider here the subset of homogeneous perturbations given by sections of $\HOMT(N,\SerreOperator)$ rather than the full $\HOMT(\bV,\rho)$.  
\end{remark}

Let $\sW \subset \SpaceOfRealCauchyRiemannOperators(N,\SerreOperator)$ be the wall of failure of $2$-rigidity defined in \autoref{Sec_ThreePerspectivesOn2Rigidity}. 

\begin{theorem}
  \label{Thm_ChamberedInvariantsFromADHM12}
  Let $\fd = \delbar_N + \Upsilon$ and $\eta \in \rH^{1,0}(C)\times\R$.
  Suppose that $\fd \in\SpaceOfRealCauchyRiemannOperators(N,S)\setminus \sW$ and $\eta \notin \{0\}\times (\pi/2)\Z$. 
  The signed count of solutions $\#\ADHMCurvesModuliSpace(\fd,\eta, d)$ to the $(\Upsilon,\eta)$-perturbed $(1,2)$ ADHM equation for the bundle $H$ with $\deg(H) = d$ is well-defined, as in \autoref{Prop_ChamberedInvariantsFromGaugeTheory}, and defines an element
  \begin{itemize}
  \item $\#\ADHMCurvesModuliSpace(\cdot,d) \in \rH^0(  \SpaceOfRealCauchyRiemannOperators(N,\SerreOperator) \setminus \sW; \Z)$ if the genus of $C$ is positive,
  \item $\#\ADHMCurvesModuliSpace(\cdot,d) \in  \rH^0( (\SpaceOfRealCauchyRiemannOperators(N,\SerreOperator) \setminus \sW) \times (\R\setminus (\pi/2)\Z); \Z)$ if the genus of $C$ is zero.
  \end{itemize}
\end{theorem}

We begin with a description of limiting configurations for the $(1,2)$ ADHM vortex equation (in the sense of \autoref{Def_LimitingConfiguration}).

\begin{prop}
  \label{Prop_ADHM12LimitingConfiguration}
  Let $\fd = \delbar_N + \Upsilon$. 
  If $(B,A,\Psi)$ is an $\Upsilon$-limiting configuration for the $(1,2)$ ADHM vortex equation \autoref{Eq_ADHM12}, then there exist  a Euclidean line bundle $\fl$ over $C\setminus B$, a bundle inclusion
  \begin{equation*}
    \rho \colon \fl\otimes \C \hookrightarrow (\End_\C H ) |_{C\setminus B},
  \end{equation*}
  and sections $\xi_0 \in \Gamma(C\setminus B, N\otimes \fl)$ and $\xi_1 \in \Gamma(C,N)$  such that
  \begin{enumerate}
  \item $\nabla_A\rho = 0$,  
  \item $\xi_0$ satisfies $\fd^\fl(\xi_0) = 0$ and is of class $W^{1,2}$ on $C\setminus B$,
  \item $|\xi_0|$ extends to a Hölder continuous function on $C$ such that $|\xi_0|^{-1}(0) = B$, 
  \item $\xi_1$ satisfies $\fd( \xi_1) = 0$ on $C$,
  \item $\xi = \rho(\xi_0) + \xi_1 \otimes_C \id_H$.
  \end{enumerate}
  Moreover, the set $B$ is finite. 
\end{prop}

\begin{proof}
  Defne $\hat\xi_0$ and $\xi_1$ to be the trace-free and trace part of $\xi$, so that $\Psi = (\hat\xi_0 + \xi_1, \alpha,\beta)$. 
  Equation $\mu(\Psi) = 0$ implies that $\alpha=\beta=0$ and $\hat\xi_0$ has two eigenvalues $\pm \lambda \in N$. 
  (This is how one proves that the hyperkähler quotient of the ADHM$_{1,2}$ representation is the symmetric product $\Sym^4 \R^4$.)
  
  Locally over $C\setminus B$ we can choose eigenvectors $\Psi_\pm$ in $H$ such that $\hat\xi \Psi_\pm = \pm \lambda \Psi_\pm$ and $|\Psi_\pm| = 1$. 	
  Trace-free endomorphisms which are diagonal with respect to the basis $\Psi_-,\Psi_+$ define a line subbundle of $\End_\C H$, which is the image of $\rho \colon \fl \otimes \C \to \End_\C H$ for some Euclidean line bundle $\fl$.
  By construction, $\hat\xi_0 = \rho(\xi_0)$, and equation $(\delbar_{N,A}+\Upsilon)(\hat\xi_0) = 0$ is equivalent to $\fd^\fl(\xi_0) = 0$.
  It follows from \autoref{Prop_ZeroLocusOfTwistedHolomorphicSectionIsFinite} that $B$ is finite.
  Since $\xi_1$ is a $W^{1,2}$ section of $N$ defined over $C\setminus B$ and satisfying $\fd(\xi_1)=0$, the unique continuation principle implies that $\xi_1$ extends to all of $C$.
  The other propertiers of $\xi_0$ and $\xi_1$ follow from the definition of a limiting configuration. 
\end{proof}

\begin{proof}[Proof of \autoref{Thm_ChamberedInvariantsFromADHM12}]
  First, we show that if $\eta_\C \neq 0$, then the moduli space $\sM(\fd,\eta,d)$ consists of irreducible solutions.
  Suppose that $(A,\xi,\alpha,\beta)$ is a reducible solution to \autoref{Eq_ADHM12}.
  By considering possible stabilizers of $A$ in the gauge group $\sG$, and how these stabilizers act on $(\xi,\alpha,\beta)$, we see that there are two possibilities.
  The first is that $(\alpha,\beta) = (0,0)$; taking the trace of the second equation shows that this is impossible.
  The second possibility is that there exists an orthogonal splitting $H = L_1 \oplus L_2$ which is covariantly constant with respect to $A$ and such that $\xi(L_i) \subset L_i$ and $\alpha \in \Gamma(L_1)$, $\beta \in \Omega^{1,0}(L_1^*)$. 
  In that case, $[\xi\wedge\xi] = 0$ and $\alpha\beta$ is a bundle homomorphism $H \to K_C \otimes_\C H$ of rank one, so we cannot have $\alpha\beta = \eta_\C \otimes \id_H$.
  Therefore, every solution is irreducible.
  
  (The genus zero case is proved similarly. 
  If $\eta_\C = 0$ and $\eta_\R \neq 0$, then $(\alpha,\beta) \neq (0,0)$ by integrating the third equation.
  On the other hand, if $E = L_1 \oplus L_2$ and $A = A_1 \oplus A_2$ as above, then one of $\alpha$ or $\beta$ must vanish by the second equation.
  The third equation gives us $i \ast F_{A_2} = \eta_\R + \pi\deg(H)$, which contradicts with $\eta_\R \notin (\pi/2) \Z$.)

  To prove compactness, assume that $(A_n,\Psi_n)$ is a sequence of solutions, with $\Psi_n = (\xi_n,\alpha_n,\beta_n)$. 
  By 	\autoref{Thm_CompactnessQuaternionicVortexEquations}, if $\limsup_{n\to\infty} \| \Psi_n \|_{L^2} < \infty$, then after passing to a subsequence and applying gauge transformations $(A_n,\Psi_n)$ converges to a solution.
  Suppose, by contradiction, that $\lim_n \| \Psi_n \|_{L^2} = \infty$.
  By \autoref{Thm_CompactnessQuaternionicVortexEquations} there exists a limiting configuration $(B,A,\Psi)$.
  By \autoref{Prop_ADHM12LimitingConfiguration}, $B$ is finite and there exists a Euclidean line bundle $\fl$ over $C\setminus B$ together with a non-zero $W^{1,2}$ section $\zeta_0$ satisfying $\fd^\fl(\zeta_0)$.
  This contradicts the assumption that $\fd$ is $2$-rigid. 
  Therefore, $\limsup_{n\to\infty} \| \Psi_n \|_{L^2} < \infty$ for any sequence of solutions and $\sM(\fd,\eta,d)$ is compact. 
  This means that we are in a situation in which  \autoref{Prop_ChamberedInvariantsFromGaugeTheory} can be applied.

  If $C$ has positive genus, any two choices $\eta, \eta' \in (\rH^{1,0}(C) \oplus \R) \setminus \{(0,0)\}$ can be connected by a path avoiding $(0,0)$ which shows that $\#\ADHMCurvesModuliSpace(\fd, \eta,d) = \#\ADHMCurvesModuliSpace(\fd',\eta',d)$. 
  Given a path $(\fd_t)_{t\in[0,1]}$ such that $\fd_t \notin \sW$, we can lift it to a path $(\Upsilon_t, \eta_t)$ as above.
  The preceding discussion implies that $\sM(\fd_t,\eta_t,d)$ is compact and consists of irreducible solutions for every $t\in [0,1]$.
  Therefore, \autoref{Prop_ChamberedInvariantsFromGaugeTheory} implies that 
  \begin{equation*}
    \#\ADHMCurvesModuliSpace(\fd_0,\eta_0,d) =  \#\ADHMCurvesModuliSpace(\fd_1,\eta_1, d),
  \end{equation*}
  so that it is constant on connected components of $\SpaceOfRealCauchyRiemannOperators(N,\SerreOperator) \setminus \sW$.
  The genus zero case is similar except that $\rH^{1,0}(C)=0$ and we need to take $\eta_\R \notin (\pi/2)\Z$. 
\end{proof}

\subsection{The $(2,1)$ ADHM vortex equation}
\label{Sec_ADHM21}

In this section we define chambered invariants of real Cauchy--Riemann operators using the ADHM vortex equation \autoref{Eq_ADHMVortexEquation} with $(r,k) = (2,1)$
The construction of perturbations in this case is somewhat more involved than in the case $(r,k) = (1,2)$ discussed earlier, as the algebra of the $(2,1)$ ADHM representation has a less direct geometric interpretation.
For that reason, it is convenient to write \autoref{Eq_ADHMVortexEquation} in a different form than in the previous section, emphasising the role of spinors and Dirac operators. 

We begin with a discussion of spin and spin$^c$ Dirac operators in dimension two.
For a brief summary of spinors in low dimensions and their relation to quaternions see \autoref{Sec_Spinors}. 
Let $C$ be a Riemann surface and let $E \to C$ be a Hermitian vector bundle with structure group $\SU(2) = \Sp(1)$ equipped with an $\SU(2)$ connection.
A spin structure on $C$ is equivalent to a Hermitian line bundle $S^+$ and an isometry $S^+ \otimes_\C S^+ \cong K_C$. 
The spinor bundle is
\begin{equation*}
  S = S^+ \oplus S^- \quad\text{with } S^- = S^+ \otimes_\C K^* = (S^+)^*. 
\end{equation*}
The Levi-Civita connection on $K_C$ induces a spin connection on $S$. 
Two different spinor bundles $S$ and $S'$ are isomorphic as unitary bundles but as bundles with connections $S' = S \otimes \fl$ where $\fl$ is a real line bundle with a flat $\{ \pm 1 \}$--connection. 
Set
\begin{equation*}
  \bS = E\otimes_\C S \qandq \bS^\pm = E \otimes_\C S^\pm
\end{equation*}
and consider the spin Dirac operator twisted by the connection on $E$
\begin{gather*}
  D \colon \Gamma(\bS) \to \Gamma(\bS), \\
  D = 
  \begin{pmatrix}
    0 & \delbar_E^* \\
    \delbar_E & 0 
  \end{pmatrix},
\end{gather*}
where $\delbar_E$ is the Cauchy--Riemann operator twisted by the connection on $E$
\begin{equation*}
  \delbar_E \colon \Gamma(\bS^+) \to \Omega^{0,1}(\bS^+) = \Gamma(\bS^-).
\end{equation*}
The Dirac operator $D$ is $\C$-linear and self-adjoint.
Since $\bS$ is the complex tensor product of two quaternionic line bundles, it inherits a real structure
\begin{equation*}
  \tau \colon \bS \to \bS,
\end{equation*}
which is a $\C$-antilinear isomorphism satisfying $\tau^2 = \id$, defined as the tensor product of the $\C$-antilinear multiplication by $j$ on $E$ and $S$. 
The Dirac operator $D$ commutes with $\tau$ and preserves the spliting into $\pm 1$ eigenspaces of $\tau$.
Let 
\begin{equation*}
  \bS_\R = \Re(E\otimes_\C S)
\end{equation*}
be the $+1$ eigenspace.
Denote the induced self-adjoint $\R$-linear operator by
\begin{equation*}
  D_\R \colon \Gamma(\bS_\R) \to \Gamma(\bS_\R).
\end{equation*}

\begin{definition}
  Let $\End_\tau(\bS)$ be the subbundle of self-adjoint $\C$-linear endomorphisms which commute with $\tau$.
  For every $\Upsilon \in \Gamma(\End_\tau(\bS))$ set
  \begin{equation*}
    D_\Upsilon = D + \Upsilon \colon \Gamma(\bS) \to \Gamma(\bS)
  \end{equation*}
  and denote by 
  \begin{equation*}
    D_{\Upsilon,\R} \colon \Gamma(\bS_\R) \to \Gamma(\bS_\R).
  \end{equation*}
  the induced real operator. 
\end{definition}

\begin{remark}
  The bundle $\End_\tau(\bS)$ does not depend on the choice of the spin structure.
  Therefore, $\Upsilon \in \Gamma(\End_\tau(\bS))$ defines a perturbation of the spin Dirac operator for every spin structure.
\end{remark}

The next proposition follows from linear algebra discussed in \autoref{Sec_Spinors}.

\begin{prop}
  \label{Prop_FromDiracToCauchyRiemann}
  Define an operator $\fd_\Upsilon$ by the commutative diagram
  \begin{equation*}
    \begin{tikzcd}
      \Gamma(\bS_\R) \ar{r}{D_{\Upsilon,\R}} \ar{d} &  \Gamma(\bS_\R) \ar{d} \\
      \Gamma(\bS^+) \ar{r}{\fd_\Upsilon} & \Gamma(\bS^-)
    \end{tikzcd}
  \end{equation*}
  where the vertical arrows are induced by the projections $\bS \to \bS^\pm$. 
  Then $\fd_\Upsilon$ is a real Cauchy--Riemann operator on $\bS^+$ which is self-adjoint with respect to the isomorphism 
  \begin{equation*}
    \SerreOperator \colon \bS^+ \to (\bS^+)^* \otimes_\C K_C
  \end{equation*}
  induced by $S^+ \otimes_\C S^+ \cong K_C$ and $\Lambda^2_\C E \cong \underline\C$. 
  Conversely, every $\SerreOperator$-self-adjoint real Cauchy--Riemann operator on  $\bS^+$ is of this form for a unique $\Upsilon \in \Gamma(\End_\tau(\bS))$. 
\end{prop}

\begin{proof}
  For $\Upsilon=0$ this follows from the relation between the Dirac operator and the Cauchy--Riemann operator in dimension $2$, see \autoref{Prop_DiracOperator2D}.
  The general case then follows from the isomorphism $\End_\tau(\bS) \cong \Hom(\bS^+,\bS^-)$ described in \autoref{Prop_RealSpinorsProjection}.
\end{proof}

From now on, we will implicitly identify the space $\SpaceOfRealCauchyRiemannOperators(\bS^+,\SerreOperator)$ of $\SerreOperator$-self-adjoint real Cauchy--Riemann operators on $\bS^+$ with the space of Dirac operators on $\bS_\R$ of the form $D_{\Upsilon,\R}$.
This is an affine space modeled on $\Gamma(\End_\tau(\bS))$.

\begin{remark}
  The section $\Upsilon$ can be recovered from $\fd_\Upsilon$ as follows.
  Let $\fd_\Upsilon = \delbar_E + \fn$ for $\fn \in \Gamma(\Hom(\bS^+, \bS^-))$. 
  Let $\fn = \fn_c + \fn_a$ be the decomposition into $\C$-linear and $\C$-antilinear parts.
  Denote by 
  \begin{equation*}
    \overline\gamma \colon \bS \to \bS
  \end{equation*}
  the $\C$-antilinear isomorphism which combines multiplication by $j$ on $E$ with the isomorphism $S^\pm \cong S^\pm$ given by the Hermitian metric. 
  Then $\Upsilon$ is expressed in terms of $\fn$ by
  \begin{equation*}
    \Upsilon = 
    \begin{pmatrix}
      -\overline\gamma \fn_a & \overline\gamma  \fn_c \overline\gamma   \\
      \fn_c & - \fn_a\overline\gamma 
    \end{pmatrix}. \qedhere
  \end{equation*}
\end{remark}

Next, we discuss spin$^c$ structures on $C$; such a structure corresponds to the choice of a Hermitian line bundle $\det(W) \to C$ of even degree.
If $W = W^+\oplus W^-$ is the spinor bundle, then 
\begin{equation*}
  W^- = K^*\otimes_\C W^+ \qandq \det(W) = K^*\otimes_\C (W^+)^2.
\end{equation*}
A $\U(1)$ connection $A$ on $\det(W)$ induces a unique spin$^c$ connection on $W$ such that the induced connection on $\det(W)$ agrees with $A$. 
Set
\begin{equation*}
  \bW = E \otimes_\C W \qandq \bW^\pm = E\otimes_\C W^\pm.
\end{equation*}
Given such $A$ and a connection on $E$, the corresponding twisted Dirac operator on $\bW$ is
\begin{gather*}
  D_A \colon \Gamma(\bW) \to \Gamma(\bW), \\
  D_A = 
  \begin{pmatrix}
    0 & \delbar_A^* \\
    \delbar_A & 0
  \end{pmatrix},
\end{gather*}
where
\begin{equation*}
  \delbar_A \colon \Gamma(\bW^+) \to \Omega^{0,1}(C,\bW^+) = \Gamma(\bW^-)
\end{equation*}
is the Cauchy--Riemann operator twisted by $A$ and the connection on $E$.
This operator can be perturbed in the following way.

\begin{definition}
  Fix a spin structure on $C$ with spinor bundle $S$.
  Since $W$ is obtained from $S$ by tensoring by a complex line bundle, a section $\Upsilon$ of $\End_\tau(\bS)$ can be interpreted as a section of the bundle of $\C$-linear self-adjoint endomorphisms of $\bW$. 
  Given a connection $A$ on $\det W$, let
  \begin{equation*}
    D_{A,\Upsilon} = D_A + \Upsilon \colon \Gamma(\bW) \to \Gamma(\bW). \qedhere
  \end{equation*}
\end{definition}

\begin{definition}
  \label{Def_ADHM21VortexEquation}
  Let $C$ be a Riemann surface, let $E \to C$ be a complex vector bundle with structure group $\SU(2)$, and let $W \to C$ be a spin$^c$ structure. 
  Let $\Upsilon \in \Gamma(\End_\tau(\bS))$ and $\eta = (\eta_\C,\eta_\R)$ where $\eta \in \rH^{1,0}(C)$ and $\eta_\R \in \R$.
  The \defined{$(\Upsilon,\eta)$-perturbed $(2,1)$ ADHM vortex equation} for  
  \begin{equation*}
    A \in \sA(\det W), \qandq \Psi = (\alpha,\beta) \in \Gamma(\bW),  
  \end{equation*}
  reads
  \begin{equation}
    \label{Eq_ADHM21}
    \begin{split}
      D_{A,\Upsilon}  \Psi &= 0, \\
      \mu_\C(\Psi) &= \eta_\C, \\
      \ast F_A + i\mu_\R(\Psi) &= \eta_\R + 2\pi\deg(W),
    \end{split}
  \end{equation}	
  where
  \begin{equation*}
    \mu_\C(\Psi) = \alpha^*\beta, \qandq \mu_\R(\Psi) = i(|\alpha|^2 - |\beta|^2).
    \qedhere
  \end{equation*}
\end{definition}

\begin{remark}
  These equations are equivalent to a perturbation of the ADHM vortex equation  \autoref{Eq_ADHMVortexEquation} for $r=2$, $k=1$.
  The equations for $\xi$ in \autoref{Eq_ADHMVortexEquation} decouple so that $\xi$ can be ignored.
  The remaining equations can be identified with \autoref{Eq_ADHM21} by taking $W^+ = S^+ \otimes H$, replacing $E$ by $E^*$, and $\beta$ by $\beta^*$. 
  The point of that last conjugation is that $\H$ can be considered as $\C^2$ in two different ways by left- or right-multiplication by $i$, so that the associated bundle $\bS$ has two natural complex structures; the description used in \autoref{Eq_PerturbedQuaternionicVortex} corresponds to one of them, and in \autoref{Eq_ADHM21} to the other.
  As for perturbations, by \autoref{Prop_FromDiracToCauchyRiemann} we can consider $\Gamma(\End_\tau(\bS))$ as a subspace of the space of homogeneous perturbations introduced in \autoref{Def_ZerothOrderPerturbation}. 
  We use Dirac operators as they interact with the algebraic structure of the $(2,1)$ ADHM representation in a more natural way than Cauchy--Riemann operators.
\end{remark}

Let $\sW \subset \SpaceOfRealCauchyRiemannOperators(\bS^+, \SerreOperator)$ be the wall of failure of $2$-rigidity 
defined in \autoref{Sec_ThreePerspectivesOn2Rigidity}.  

\begin{theorem}
  \label{Thm_ChamberedInvariantsFromADHM21}
  Let $\fd \in  \SpaceOfRealCauchyRiemannOperators(\bS^+, \SerreOperator) \setminus \sW$.
  Let $D_{\Upsilon,\R}$ be the Dirac operator corresponding to $\fd$ under \autoref{Prop_FromDiracToCauchyRiemann}.
  Let $\eta \in \rH^{1,0}(C) \times \R \setminus \{ (0,0)\}$.
  The signed count of solutions $\#\ADHMGaugeModuliSpace(\fd,\eta,d)$ to the $(\Upsilon,\eta)$-perturbed $(2,1)$ ADHM equation for the spin$^c$ structure with $\deg W^+ = d$ is well-defined, as in \autoref{Prop_ChamberedInvariantsFromGaugeTheory}, and defines an element
  \begin{itemize}
  \item $\#\ADHMCurvesModuliSpace(\cdot,d) \in \rH^0(  \SpaceOfRealCauchyRiemannOperators(\bS^+,\SerreOperator) \setminus \sW; \Z)$ if the genus of $C$ is positive,
  \item $\#\ADHMCurvesModuliSpace(\cdot, d) \in \rH^0((\SpaceOfRealCauchyRiemannOperators(\bS^+,\SerreOperator) \setminus \sW) \times (\R \setminus \{0\}); \Z)$ if the genus of $C$ is zero.
  \end{itemize}
\end{theorem}

Since the proof uses \autoref{Thm_CompactnessQuaternionicVortexEquations}, we need to understand limiting configurations for the $(2,1)$ ADHM vortex equation. 

\begin{prop}
  \label{Prop_ADHM21LimitingConfiguration}
  If $(B,A,\psi)$ is an $\Upsilon$-limiting configuration for the $(2,1)$  ADHM vortex equation \autoref{Eq_ADHM21}, then there exist  a Euclidean line bundle $\fl$ over $C\setminus B$ and an isomorphism
  \begin{equation*}
    \rho \colon W|_{C\setminus B} \to S|_{C\setminus B} \otimes \fl
  \end{equation*}
  such that $\nabla_A\rho =0$ and the section $\Psi_\R = (\id_E \otimes_ \C \rho)\Psi$ has the following properties:
  \begin{enumerate}
  \item $\Psi_\R$ takes values in $\bS_\R \otimes\fl$,
  \item $\Psi_\R$ is of class $W^{1,2}$ on $C\setminus B$,
  \item $|\Psi_\R|$ extends to a Hölder continuous function on $C$ such that $|\Psi_\R|^{-1}(0) = B$, 
  \item $D_\Upsilon^\fl \Psi_\R = 0$.
  \end{enumerate}
  Moreover, the set $B$ is finite. 
\end{prop}

\begin{remark}
  The bundle $S \otimes \fl$ is the spinor bundle of a spin structure on $C\setminus B$.
  Therefore, a limiting configuration induces a reduction of the spin$^c$ structure $W$ to a spin structure outside $B$. 
  In particular, $\deg W$ is equal to a signed count of points in $B$.
  This is a $2$-dimensional analog of \cite[Theorem 4]{Haydys2019}. 
\end{remark}

\begin{proof}
  This is a special case of a theorem of Haydys describing limiting configurations of generalized Seiberg--Witten equations \cite{Haydys2013}; see also \cite[Section 2.3]{Doan2017a}. 
  Let $L$ be the complex line bundle satisfying $W = S \otimes L$.
  Trivialize $L$ locally over a contractible subset $U \subset C\setminus B$ so that $\Psi$ is identified with a section of $\bS$. 
  Since the hyperkähler quotient of $\C^2\otimes_\C \H$ by $\U(1)$ is $\Re(\C^2\otimes_\C \H) / \Z_2$, there is a gauge transformation of $L|_U$ which maps $\Psi$ to $\Re(\bS)$.
  This gauge transformation is unique up to $\pm 1$. 
  Change the trivialization by that gauge transformation.
  In the new trivialization, let
  \begin{equation*}
    A = \rd - a^* + a \quad\text{with } a \in \Omega^{0,1}(U),
  \end{equation*}
  so that
  \begin{equation*}
    D_{A,\Upsilon} = D_\Upsilon + \gamma(a)
    \quad
    \text{with } 
    \gamma(a) =
    \begin{pmatrix}
      0 & a^* \\
      a & 0 
    \end{pmatrix}.
  \end{equation*}
  Since $D_\Upsilon$ preserves $\Re(\bS)$ and $\gamma(a)$ maps $\Re(\bS)$ to $\Im(\bS)$, equation $D_{A,\Upsilon}\Psi = 0$ implies that $a=0$. 
  Therefore, locally $A$ agrees with the spin connection on $S$. 
  Globally, by patching these local trivializations together, we obtain a real line bundle $\fl \to C\setminus B$ with monodromy $\pm 1$, together with an isomorphism $\rho \colon W|_{C\setminus B} \to S|_{C\setminus B}\otimes\fl$ which is covariantly constant with respect to $A$ and such that $\Psi_\R = (\id_E\otimes_\C\rho)\Psi$ is real.
  The other three properties follow from the analogous properties of $\Psi$. 
  It follows from \autoref{Prop_ZeroLocusOfTwistedHolomorphicSectionIsFinite} that $B$ is finite.
\end{proof}

\begin{proof}[Proof of \autoref{Thm_ChamberedInvariantsFromADHM21}]
  The proof is the same as that of \autoref{Thm_ChamberedInvariantsFromADHM12}.
  Given $(\Upsilon,\eta)$ as above, we need to show that the moduli space of solutions $\ADHMCurvesModuliSpace(\fd,\eta,d)$ is compact and consists of irreducible solutions, in which case $\#\ADHMCurvesModuliSpace(\fd,\eta,d) \in \Z$ is well-defined and invariant under deformations $(\fd_t,\eta_t)$ with the same properties. 
  If $(A,\Psi)$ is a solution to \autoref{Eq_ADHM21}, it follows from the second and third equation (by integrating) that for $\eta \neq 0$ we have $\Psi \neq 0$.
  Therefore, every solution in $\#\ADHMCurvesModuliSpace(\fd,\eta,d)$ is irreducible. 
  To prove that $\ADHMCurvesModuliSpace(\fd,\eta,d)$ is compact if $\fd \notin \sW$ we proceed in the same way as in the proof of \autoref{Thm_ChamberedInvariantsFromADHM12}, using \autoref{Thm_CompactnessQuaternionicVortexEquations} and  \autoref{Prop_ADHM12LimitingConfiguration} which relates limiting configurations to the $(2,1)$ ADHM vortex equations to elements of $\ker\fd^\fl$ for some ramified Euclidean line bundle $\fl$. 
\end{proof}



\appendix
\section{Proof of the index formula for $\fd^\fl$}
\label{Sec_ProofOfIndexFormula}

Here is the extension of \autoref{Thm_Index} promised in \autoref{Rmk_Index}.

\begin{theorem}
  \label{Thm_Index+}
  Let $C$ be a closed connected Riemann surface.
  Let $V$ be a complex vector bundle over $C$.
  If
  $\fd$ is a real Cauchy--Riemann operator on $V$ and  
  $\fl$ is a ramified Euclidean line bundle over $C$,
  then
  $\fd^\fl \co W^{1,2}\Gamma\paren{\mathring{C},V\otimes\fl} \to L^2\Omega^{0,1}\paren{\mathring{C},V\otimes\fl}$ is Fredholm with
  \begin{equation*}
    \ind \fd^\fl = 2\deg V + \rk_\C V \cdot \chi(C) - \# \Br(\fl) \cdot \rk_\C V.
  \end{equation*}
\end{theorem}

The discussion in \autoref{Sec_IndexFormulaAndResidues} already shows that $\fd^\fl$ is Fredholm.
The proof of the index formula relies on the following.

\begin{prop}
  \label{Prop_X}
  Let $\fl$ is a ramified Euclidean line bundle over $C$.
  \begin{enumerate}
  \item
    \label{Prop_X_L}
    $\mathring{\sL} \coloneqq \fl \otimes \C$ be a holomorphic line bundle over $\mathring{C}$ (with Dolbeault operator obtained by twisting the trivial Dolbeault operator by $\fl$).
  \item
    \label{Prop_X_Beta}
    The unique(!) orientation-preserving isometry $\fl^2 \iso \ubR$ determines a $\mathring{\beta} \in \rH^0\paren{\mathring{C},\mathring{\sL}^2}$.
  \item
    \label{Prop_X_Extension}
    $\mathring{\sL}$ extends to a holomorphic line bundle $\sL$ over $C$ such that:
    $\mathring{\beta}$ extends to $\beta \in \rH^0(C,\sL^2)$, and
    $\beta$ is transverse to the zero section;
    moreover: $\sL$ is unique up to unique isomorphism.
  \item
    \label{Prop_X_Twist}
    If $\delbar$ is a Dolbeault operator on $V$ and $\sV$ denotes the corresponding holomorphic vector bundle,
    then
    \begin{equation*}
      \ker \delbar^\fl
      \iso \rH^0\paren{C,\sV \otimes_\C \sL(-\Br(\fl)}
      \iso \rH^0\paren{C,\sV \otimes_\C \sL^*}.
    \end{equation*}
  \end{enumerate}
\end{prop}

\begin{proof}
  \autoref{Prop_X_L} and \autoref{Prop_X_Beta} are obvious.

  It suffices to prove \autoref{Prop_X_Extension} for the ramified Euclidean line bundle $\fm$ over $D$ from the proof of \autoref{Prop_RiemannSurface_DegreeTwo_+-}.
  The isomorphism 
  $\tau \co \fm \otimes\C \iso \ubC$ of complex line bundles defined by
  $\tau\paren{(z,w) \otimes \lambda} \coloneqq \paren{z, \lambda w}$ is an isomorphism of holomorphic line bundles.
  Identifying $\fm\otimes\C = \ubC$ and $\ubC^{\otimes 2} = \ubC$,
  $\mathring \beta(z) = z \in \rH^0\paren{\mathring{D},\ubC}$.
  Evidently, $\ubC$ and $\mathring{\beta}$ extend from $\mathring{D}$ to $D$.

  To prove \autoref{Prop_X_Twist},
  observe that
  if $s \in \ker \paren{\delbar^\fl \co W^{1,2}\Gamma\paren{\mathring{D},\fl\otimes\C } \to L^2\Omega^{0,1}\paren{\mathring{D},\fl\otimes\C}}$,
  then 
  $s(z) = z^{1/2} \otimes \paren{z^{1/2}\cdot f(z)}$
  with $f \co D \to \C$ holomorphic;
  therefore: $\tau_*s(z) = zf(z)$.
  Consequently, the obvious
  inclusion $\ker \delbar^\fl \incl \rH^0\paren{C,\sV \otimes_\C \sL}$ factors through an isomorphism  
  \begin{equation*}
    \ker \delbar^\fl \iso \rH^0\paren{C,\sV \otimes_\C \sL(-\Br(\fl))}.
  \end{equation*}
  Since $\beta$ vanishes transversely in $\Br(\fl)$,
  $\sL^2 \iso \sO_\C(\Br(\fl))$.
  Therefore,
  \begin{equation*}
    \sL(-\Br(\fl)) \iso \sL \otimes_\C (\sL^2)^* \iso \sL^*.
    \qedhere
  \end{equation*}  
\end{proof}

\begin{proof}[Proof of \autoref{Thm_Index+}: index formula]
  By \autoref{Prop_RealCauchyRiemanOperator_AffineSpace},
  it suffices to prove the index formula if $\fd = \delbar$ is a Dolbeault operator.
  Denote by $\sV$ the holomorphic vector bundle corresponding to $(V,\delbar)$.
  By Riemann--Roch and with $\sL$ as in \autoref{Prop_X}~\autoref{Prop_X_Extension},
  \begin{equation*}
    \dim \rH^0(C,\sV\otimes_\C \sL^*) - \dim \rH^1(C,\sV \otimes_\C \sL^*)
    =
    2 \deg V + \rk_\C V \cdot \chi(C) - \# \Br(\fl) \cdot \rk_\C V.
  \end{equation*}

  By \autoref{Prop_X}~\autoref{Prop_X_Twist},
  $\ker \delbar^\fl \iso \rH^0\paren{C,\sV \otimes \sL^*}$.
  By Serre duality,
  \begin{equation*}
    \rH^1(C,\sV \otimes_\C \sL^*)
    = \rH^0(C,\sV^\dagger \otimes_\C \sL^* (\Br(\fl) )
  \end{equation*}
  with $\sV^\dagger \coloneqq \sV^* \otimes_\C K_C$.
  $\rH^0(C,\sV^\dagger \otimes_\C \sL^* (\Br(\fl) )$ consists of meromorphic sections of $\sV^\dagger \otimes_\C \sL^*$ with at most simple poles at $\Br(\fl)$.
  Therefore,
  there is an inclusion
  \begin{equation*}
    \rH^0(C,\sV^\dagger \otimes_\C \sL^* (\Br(\fl) )
    \incl \ker \delbar_{L^2}^{\dagger,\fl}.
  \end{equation*}
  In fact, \autoref{Prop_Residue_Construction} implies that this map is an isomorphism.
  This proves the index formula.
\end{proof}

\begin{remark}
  \label{Rmk_Prop_X_Converse}
  \autoref{Prop_X} has a partial converse:
  if $\sL$ is a holomorphic bundle over $C$ and $\beta \in \rH^0(C,\sL^2)$ is transverse to the zero section,
  then 
  \begin{equation*}
    \fl \coloneqq \set{ (x,v) \in \sL : v^2 \in \R_{\geq 0} \cdot \beta(x) }
  \end{equation*}
  with the Euclidean inner product $\Inner{(x,v),(x,w)} \coloneqq \frac{vw}{\beta(x)}$ is a ramified Euclidean line bundle over $C$ with $\Br(\fl) = Z(\beta)$.
\end{remark}

\begin{remark}
  Here is a description of holomorphic maps $\pi \co \tilde C \to C$ of degree two (alternative to \autoref{Prop_RiemannSurface_DegreeTwo_+-} and familiar from algebraic geometry):
  \begin{enumerate}
  \item
    If $\sL$ is a holomorphic bundle over $C$ and $\beta \in \rH^0(C,\sL^2)$ is transverse to the zero section,
    then
    \begin{equation*}
      \tilde C \coloneqq \set{ (x,v) \in \sL : v^2 = \beta(x) }
    \end{equation*}
    is a Riemann surface, and
    $\pi \co \tilde C \to C$ defined by $\pi(x,v) \coloneqq x$ is a holomorphic map of degree two.       
    A moment's thought reveals that
    \begin{equation}
      \label{Ex_Pi*OC}
      \pi_*\sO_{\tilde C} = \sO_C \oplus \sL^*.
    \end{equation}
    The summands are (anti-)invariant under the action of the involution $\tau$ of $\tilde C$ defined by $\tau(x,v) \coloneqq (x,-v)$.
  \item
    If $\pi \co \tilde C \to C$ is a holomorphic map of degree two,
    then it is a branched double cover.
    In particular, there is an involution $\tau$ of $\tilde C$ which swaps the sheets of $\pi$.
    Consider the anti-invariant subbundle
    \begin{equation*}
      \sL^* \coloneqq \paren{\pi_*\sO_{\tilde C}}^-.
    \end{equation*}
    Set $\sL \coloneqq (\sL^*)^*$.
    The holomorphic section $\beta \in \rH^0(C,\sL^2)$ defined by
    \begin{equation*}
      \beta(x)(\lambda^2) \coloneqq \paren{\lambda(\tilde x)}^2 \in \C
      \quad\textnormal{for}\quad \tilde x \in \pi^{-1}(x)
    \end{equation*}
    is transverse to the zero section.
    Moreover,
    the map $\phi \co \tilde C \to \set{ (x,v) \in \sL : v^2 = \beta(x) }$ defined by $\phi(\tilde x) \coloneqq (\pi(\tilde x),\ev_{\tilde x})$
    is biholomorphic.
  \end{enumerate}
  \autoref{Prop_X} and \autoref{Rmk_Prop_X_Converse}
  bridge the above discussion and \autoref{Prop_RiemannSurface_DegreeTwo_+-}.
  Moreover,  
  \autoref{Ex_Pi*OC} provides a further (perspective on the) proof of \autoref{Prop_X}~\autoref{Prop_X_Twist}:
  \begin{equation*}
    \ker \delbar^\fl
    \iso \paren{\ker \pi^*\delbar}^-
    \iso \rH^0\paren{\tilde C,\pi^*\sV}^{-}
    \iso \rH^0\paren{C,\sV \otimes_\C \sL^*}.
    \qedhere
  \end{equation*}
\end{remark}


\section{Local wall-crossing formulae}
\label{Sec_LocalWallCrossing}

The following explains the analogue of \autoref{Ex_WallCrossingHypersurface} if $\sW \subset \sP$ is a proper wall.

\begin{theorem}
  \label{Thm_LocalWallCrossingProperWall}
  Let $\sP$ be a Banach manifold.
  Let $\sW \subset \sP$ be a proper wall in $\sP$.
  Let $\pi \co \sE \to \sP$, $\nu \co \sN \to \sE$ and $\sigma \co \sS \to \sE$ be as in \autoref{Def_ProperWall}.
  Set
  \begin{equation*}
    \sW^* \coloneqq \sW \setminus \sB
    \qwithq
    \sB \coloneqq \paren*{ \im (\pi\circ\nu) \cup \im (\pi\circ\sigma) \cup \pi(\Crit{\pi}) }.
  \end{equation*}
  Denote by $\fo$ the local system on $\sW^*$ induced by the local system $\tilde\fo$ of orientations of $\det T\pi$ on $\sE$.
  The following hold:
  \begin{enumerate}
  \item
    \label{Thm_LocalWallCrossingProperWall_Sharp}        
    There is a homomorphism
    \begin{equation*}
      \sharp \co \rH^0(\sW^*;\Val \otimes_\Z \fo) \to \Hom(\rH_1(\sP,\sP\setminus\sW),\Val) \iso \rH^1(\sP,\sP\setminus\sW;\Val)
    \end{equation*}
    such that \autoref{Eq_Sharp} holds for every $C^1$ path $\bp \co [0,1] \to \sP$ is a $C^1$ with $\bp(0),\bp(1) \notin \sW$ and transverse to $\pi$, $\pi \circ \nu$, and $\pi \circ \sigma$.
    (The sum in \autoref{Eq_Sharp} is finite because of \autoref{Prop_ProperWall}.)
  \item
    \label{Thm_LocalWallCrossingProperWall_Delta}
    There is an injective homomorphism
    \begin{equation*}
      \SpaceOfWallCrossingFormulae(\sP,\sW;\Val)
      \stackrel{\Delta_{\loc}}{\into}
      \rH^0(\sW^*;\Val \otimes_\Z \fo)
    \end{equation*}
    such that the following diagram commutes
    \begin{equation*}    
      \begin{tikzcd}
        \SpaceOfWallCrossingFormulae(\sP,\sW;\Val) \ar{r}{\Delta_\loc} \ar[swap]{rd}{\Delta} &
        \rH^0(\sW^*;\Val \otimes_\Z \fo) \ar{d}{\sharp}  \\
        & \rH^1(\sP,\sP\setminus\sW;\Val).
      \end{tikzcd}
    \end{equation*}
  \item
    \label{Thm_LocalWallCrossingProperWall_Exact}
    $\Delta_\loc$ and $\sharp$ assemble into the exact sequence
    \begin{equation*}
      \SpaceOfWallCrossingFormulae(\sP,\sW;\Val) \stackrel{\Delta_{\loc}}{\into} \rH^0(\sW;\Val \otimes_\Z \fo) \stackrel{\sharp}{\to} \rH^1(\sP;\Val).
      \qedhere
    \end{equation*}    
  \end{enumerate}
\end{theorem}

\begin{remark}
  \label{Rmk_LocalWallCrossingProperWall_ComputeDelta}
  The homomorphism $\Delta_\loc$ can be computed as follows.
  Let $p \in \sW^*$.
  Choose a $C^1$ path $\bp \co [-1,1] \to \sP$ transverse to $\pi$, $\pi \circ \nu$, and $\pi \circ \sigma$.
  There is an $\epsilon \in (0,1]$ such that $\bp^{-1}(\sW) \cap [-\epsilon,\epsilon] = 0$ and 
  \begin{equation*}
    \Delta_\loc([\ChamberedInvariant])(p)
    = \Delta([\ChamberedInvariant])([\bp])
    = \ChamberedInvariant(\bp(\epsilon)) - \ChamberedInvariant(\bp(-\epsilon)).
    \qedhere
  \end{equation*}
\end{remark}

\begin{remark}
  \label{Rmk_LocalWallCrossingProperWall_Sheaf}
  An exercise in sheaf theory constructs a sheaf $\sF$ on $\sW$ and an exact sequence with $\rH^0(\sW;\sF)$ in place of $\rH^0(\sW^*;\Val\otimes_\Z \fo)$ without any hypothesis on $\sW$.
  From this perspective, the point of \autoref{Thm_LocalWallCrossingProperWall} is to give a manageable description of this wall-crossing sheaf $\sF$ assuming $\sW$ is a proper wall in $\sP$.
\end{remark}

\autoref{Thm_LocalWallCrossingProperWall} would follow from \autoref{Ex_WallCrossingHypersurface} if $\sB$ were closed.
Since $\sB$ might not be closed,
the proof of \autoref{Thm_LocalWallCrossingProperWall} requires the following preparations.

\begin{lemma}
  \label{Lem_W*Neigborhood}
  Assume the situation of \autoref{Thm_LocalWallCrossingProperWall}.
  Let $p \in \sW^*$.
  Set $\set{e} \coloneqq \pi^{-1}(p)$.
  There are
  open neighborhoods $p \in U \subset \sP$, $e \in V \subset \sE$ such that:
  \begin{enumerate}
  \item
    \label{Lem_W*Neigborhood_Embedding}
    $\pi|_V \co V \to U$ is an embedding.
  \item
    \label{Lem_W*Neigborhood_Dense}
    Every connected component of $\pi(V)$ intersects $\sW^*$.
  \item
    \label{Lem_W*Neigborhood_AlmostEverything}
    $\pi^{-1}(U)\setminus V \subset \im \nu \cup \im \sigma$.
  \end{enumerate}
\end{lemma}

\begin{proof}
  Since $T_e\pi \co T_e\sE \to T_p\sP$ is injective,
  there are open neighborhoods $p \in U \subset \sP$, $e \in V \subset \sE$ such that \autoref{Lem_W*Neigborhood_Embedding} and \autoref{Lem_W*Neigborhood_Dense} holds.

  If \autoref{Lem_W*Neigborhood_AlmostEverything} cannot be arranged to hold by shrinking $U$ and $V$,
  then there is a sequence $(e_n)$ in $\sE \setminus (\im \nu \cup \im \sigma)$ which does not converge to $e$ but with $(\pi(e_n))$ converging to $p$.
  Since $\pi|_{\sE \setminus (\im\nu \cup \im\sigma)} \co \sE\setminus(\im\nu \cup \im\sigma) \to \sP \setminus \paren{\im \paren{\pi \circ \nu} \cup \im \paren{\pi\circ\sigma}}$ is proper,
  after passing to a subsequence,
  $(e_n)$ converges to $e' \in \sE$.
  Since $\pi(e') = p$ and $\pi|_{\sE \setminus (\im\nu \cup \im\sigma)}$ is injective,
  $e' = e$: a contradiction.
\end{proof}

\begin{prop}
  \label{Prop_LocalWallCrossingProperWall_2}
  Suppose that $U,V$ satisfy
  \autoref{Lem_W*Neigborhood_Embedding},
  \autoref{Lem_W*Neigborhood_Dense},
  \autoref{Lem_W*Neigborhood_AlmostEverything}
  from \autoref{Lem_W*Neigborhood}.
  Set $U^* \coloneqq U \setminus \sB$.
  The following hold:
  \begin{enumerate}
  \item
    \label{Prop_LocalWallCrossingProperWall_2_NegligibleIsomorphism}
    The homomorphisms $\rH^0(U;\Val) \to \rH^0(U^*;\Val)$ and $\rH^0(U\setminus\pi(V);\Val) \to \rH^0(U^*\setminus\sW^*;\Val)$ are isomorphisms.
  \item
    \label{Prop_LocalWallCrossingProperWall_2_XIsomorphism}
    There is an isomorphism
    \begin{equation*}
      t \co \coker\paren{\rH^0(U;\Val) \to \rH^0\paren{U\setminus \pi(V);\Val}}
      \iso
      \rH^0(V,\Val\otimes_\Z \tilde\fo).
    \end{equation*}
  \item
    \label{Prop_LocalWallCrossingProperWall_2_Injective}
    The homomorphism
    $\rH^0(V,\Val\otimes_\Z \tilde \fo) \to \rH^0(\sW^*\cap U;\Val\otimes_\Z \fo)$
    is injective.
  \end{enumerate}
  Moreover,
  if
  $U_1,V_1$ and $U_2,V_2$ satisfy
  \autoref{Lem_W*Neigborhood_Embedding}, \autoref{Lem_W*Neigborhood_Dense}, \autoref{Lem_W*Neigborhood_AlmostEverything}
  from \autoref{Lem_W*Neigborhood},
  then the following hold:
  \begin{enumerate}[resume]
  \item
    \label{Prop_LocalWallCrossingProperWall_2_Also}        
    $U_{12} \coloneqq U_1\cap U_2$, $V_{12} \coloneqq V_1 \cap V_2$ satisfy
    \autoref{Lem_W*Neigborhood_Embedding}, \autoref{Lem_W*Neigborhood_Dense}, \autoref{Lem_W*Neigborhood_AlmostEverything}
    from \autoref{Lem_W*Neigborhood}.
  \item
    \label{Prop_LocalWallCrossingProperWall_2_Commute}        
    The diagram
     \begin{equation*}
      \begin{tikzcd}
        \rH^0(\sP\setminus\sW;\Val) \ar{r}\ar{d} & \rH^0(U_1\setminus \pi(V_1);\Val) \ar{d} \ar{r} & \rH^0(V_1;\Val\otimes_\Z \tilde\fo) \ar{dd} \\
        \rH^0(U_2\setminus \pi(V_2);\Val) \ar{d} \ar{r} &  \rH^0(U_{12} \setminus \pi(V_{12});\Val) \ar{rd} \\
        \rH^0(V_2;\Val\otimes_\Z \tilde\fo) \ar{rr} & & \rH^0(V_{12};\Val\otimes_\Z \tilde\fo)
      \end{tikzcd}
    \end{equation*}
    commutes.
  \end{enumerate}
\end{prop}

\begin{proof}
  \autoref{Prop_LocalWallCrossingProperWall_2_NegligibleIsomorphism} follows from the Sard--Smale Theorem.

  \autoref{Prop_LocalWallCrossingProperWall_2_XIsomorphism} is the elementary instance of the Thom isomorphism mention in \autoref{Ex_WallCrossingHypersurface}~\autoref{Ex_WallCrossingHypersurface_Thom}.

  \autoref{Prop_LocalWallCrossingProperWall_2_Injective} holds by construction.
  
  \autoref{Prop_LocalWallCrossingProperWall_2_Also} and \autoref{Prop_LocalWallCrossingProperWall_2_Commute} are obvious.  
\end{proof}

\begin{proof}[Proof of \autoref{Thm_LocalWallCrossingProperWall}]
  Set $\sP \coloneqq \sP \setminus \sB$.
  By the Sard--Smale Theorem,
  $\rH^0(\sP;\Val) \to \rH^0(\sP^*;\Val)$ is an isomorphism.
  Therefore,
  \begin{equation*}
    \SpaceOfWallCrossingFormulae(\sP,\sW;\Val)
    \iso
    \SpaceOfWallCrossingFormulae(\sP^*,\sW^*;\Val).
  \end{equation*}
  Let $\set{ U_\alpha \subset \sP : \alpha \in A}$ be a set of open subsets with $\sW^* \subset \bigcup_{\alpha \in A} U_\alpha$.
  Set $U_\alpha^* \coloneqq U_\alpha \setminus \sB$.
  Since $\sW$ is closed, $\set{\sP\setminus\sW} \cup \set{ U_\alpha \subset \sP : \alpha \in A}$ is an open cover of $\sP^*$.
  A diagram chase in  
  \begin{equation*}
    \begin{tikzcd}[column sep=small]
      \rH^0(\sP^*;\Val) \ar[hook]{r} \ar[hook]{d} & \rH^0(\sP^*\setminus\sW^*;\Val) \ar[two heads]{r} \ar[hook]{d} & \SpaceOfWallCrossingFormulae(\sP^*,\sW^*;\Val) \ar{d}{\tilde \Delta_\loc} \\
      \prod\limits_{\alpha \in A} \rH^0(U_\alpha^*;\Val) \ar[hook]{r} \ar{d} & \prod\limits_{\alpha \in A}\rH^0(U_\alpha^*\setminus\sW^*;\Val) \ar[two heads]{r} \ar{d} & \prod\limits_{\alpha \in A} \coker\paren{\rH^0(U_\alpha^*;\Val) \to \rH^0(U_\alpha^*\setminus\sW^*;\Val)} \\
      \prod\limits_{\alpha,\beta \in A} \rH^0(U_{\alpha\beta}^*;\Val) \ar[hook]{r} & \prod\limits_{\alpha,\beta \in A}\rH^0(U_{\alpha\beta}^* \setminus\sW^*;\Val)
    \end{tikzcd}
  \end{equation*}
  reveals that $\tilde\Delta_\loc$ is injective.
  
  By \autoref{Lem_W*Neigborhood},
  it is possible to choose $\set{U_\alpha : \alpha \in A }$ and $\set{ V_\alpha : \alpha \in A}$ such that  
  $U_\alpha$, $V_\alpha$ satisfy
  \autoref{Lem_W*Neigborhood_Embedding}, \autoref{Lem_W*Neigborhood_Dense}, and \autoref{Lem_W*Neigborhood_AlmostEverything}.
  Set $W_\alpha \coloneqq \sW^* \cap U_\alpha$.
  By \autoref{Prop_LocalWallCrossingProperWall_2},
  this induces an injective linear map
  \begin{equation*}
    \SpaceOfWallCrossingFormulae(\sP^*,\sW^*;\Val)
    \incl
    \prod_{\alpha \in A} \rH^0(W_\alpha;\Val \otimes_\Z \fo)
  \end{equation*}
  which lifts along
  $\rH^0(\sW^*;\Val \otimes_\Z \fo) \incl \prod_{\alpha \in A} \rH^0(W_\alpha;\Val \otimes_\Z \fo)$.  
\end{proof}



\section{Spinors and quaternions}
\label{Sec_Spinors}

This appendix summarizes algebraic properties of spinors in dimension $3$ used in \autoref{Sec_ChamberedInvariantsFromQuaternionicVortexEquations}. 
 
Let $S$ be the spinor representation of $\Spin(3)$.
We identify $S$ with the space of quaternions $\H$ and $\Spin(3)$ with the group of unit quaternions $\Sp(1) \subset \H$ in the following way.
The double cover  $\Sp(1) \to \SO(3)$ is given by the right action of $\Sp(1)$ on imaginary quaternions $\Im\H = \R^3$ by $q(x) = \bar{q}^{-1}xq$. 
The right action of $\Sp(1)$ on $S=\H$ is by right quaternionic multiplication.
It preserves three complex structures given by left multiplication by $i,j,k$.
The Clifford multiplication $\gamma \colon \R^3 \to \End(S)$ is given by
\begin{equation*}
	\gamma(e_1)x = xi, \quad \gamma(e_2)x = xj, \quad \gamma(e_3)x = xk.
\end{equation*} 
With these conventions, the Clifford multiplication by the volume form is
\begin{equation*}
	\gamma(e_1\wedge e_2\wedge e_3) = \gamma(e_1)\gamma(e_2)\gamma(e_3) = kji = 1.
\end{equation*}
Fix the complex structure on $\H$ given by left multiplication by $i$. 
We have a complex isomorphism
\begin{gather*}
	\C^2 \to \H \\
	(z,w) \mapsto z + wj.
\end{gather*}
With respect to this identification, left multiplication by $j$ is
\begin{equation*}
	j(z,w) = (-\bar w, \bar z),
\end{equation*}
and the action of $\Sp(1)$ agrees with the standard right 
representation of $\SU(2)$. 
The Clifford multiplication is given by the Pauli matrices
\begin{equation*}
	\gamma(e_1) =
	\begin{pmatrix}
	i & 0 \\
	0 & -i
	\end{pmatrix}, 
	\quad
	\gamma(e_2) = 
	\begin{pmatrix}
	0 & -1 \\
	1 & 0 
	\end{pmatrix}
	\quad
	\gamma(e_3) = 
	\begin{pmatrix}
		0 & i \\
		i & 0
	\end{pmatrix}.
\end{equation*} 
The Dirac operator acting on maps $\R^3 \to \C^2$ is
\begin{equation*}
	D = \gamma(e_1)\partial_1   + \gamma(e_2)\partial_2 + \gamma(e_3)\partial_3  
\end{equation*}
We can relate the spinors in dimension $3$ to spinors in dimension $2$ by considering the subgroup $\U(1) \subset \Sp(1)$ of unit complex numbers. 
The right action of $\U(1)$ on $\H = \C^2$ is
\begin{equation*}
	\lambda(z,w) = (\lambda z, \lambda^{-1}w) \quad\text{for } \lambda \in \U(1). 
\end{equation*}
The restriction of the right action of $\Sp(1)$ on $\Im\H$ to $\U(1)$ is given by
\begin{equation*}
	\lambda( x_1 i + x_2 j + x_3 k) = x_1 i + \lambda^{-2}x_2 j + \lambda^{-2} x_3 k,
\end{equation*}
so if we identify
\begin{gather*}
	\R\oplus\C = \Im\H, \\
	(t,z) \mapsto ti + zj,
\end{gather*}
then the action of $\U(1)$ is
\begin{equation*}
	\lambda(t,z) = (t, \lambda^{-2}z),
\end{equation*}

\begin{prop}
	\label{Prop_DiracOperator2D}
	Writing a map $\R\times\C \to \C^2$ as a pair of maps $\R\times\C \to \C$ and using coordinates $t$ on $\R$ and $z = x+iy$ on $\C$, the Dirac operator is
	\begin{equation*}
		D = 
		\begin{pmatrix}
			i\partial_t & - \partial_x + i \partial_y \\
			\partial_x + i \partial_y & -i\partial_t
		\end{pmatrix}
		=
		\begin{pmatrix}
			i\partial_t & - \partial_z \\
			\bar\partial_z & -i\partial_t
		\end{pmatrix}.
	\end{equation*}
\end{prop}

The following facts about real structures are used in the discussion of the $(2,1)$ ADHM vortex equation in \autoref{Sec_ADHM21},
Let $E$ and $S$ be quaternionic Hermitian vector spaces of quaternionic dimension one.
Their complex tensor product has a real structure $\tau \colon E \otimes S \to E \otimes S$, that is a complex antilinear map such that $\tau^2 = \id$, given by
\begin{equation*}
	\tau = j_E \otimes j_S,
\end{equation*}
where $j_E$ and $j_S$ are multiplication by $j$ on $E$ and $S$. 
Let 
\begin{equation*}
	\Re(E\otimes S) = \{ \psi \in E\otimes S : \tau(\psi) = \psi \}.
\end{equation*}
be the real part of $E\otimes S$.
As a real vector space $E\otimes S$ decomposes into
\begin{equation*}
	E\otimes S = \Re(E\otimes S) \oplus i\Re(E\otimes S).
\end{equation*}
Let $S = S^+ \oplus S^-$ for a complex line $S^+$ and $S^- = (S^+)^*$, with the quaternionic structure given by
\begin{gather*}
	j_S \colon S^+ \oplus S^- \to S^+\oplus S^- \\
	j_S(\alpha,\beta) = (-\beta^*,\alpha^*)
\end{gather*}
with the star denoting the metric dual $S^\pm \to S^\mp$ with respect to the Hermitian inner product.
Extend $j_E \colon E \to E$ to complex antilinear maps
\begin{equation*}
	 \overline\gamma \colon E\otimes S^\pm \to E\otimes S^\mp
\end{equation*}
by taking tensor product with the metric dual.
They satisfy $\overline\gamma^2 = -\id$. 
The real structure is 
\begin{gather*}
	\tau \colon (E\otimes S^+) \oplus (E\otimes S^-) \to (E\otimes S^+) \oplus (E\otimes S^-) \\
	\tau=
	\begin{pmatrix}
		0 & -\overline\gamma \\
		\overline\gamma & 0.
	\end{pmatrix}
\end{gather*}

\begin{prop}
	\label{Prop_RealSpinorsProjection}
	~
	\begin{enumerate}
	\item 
	We have
	\begin{align*}
		\Re(E\otimes S) &= \{ (\alpha,\beta) \in (E\otimes S^+) \oplus (E\otimes S^-) : \beta = \overline\gamma(\alpha) \}, \\
		&= \{ (\alpha,\beta) \in (E\otimes S^+) \oplus (E\otimes S^-) : \alpha = -\overline\gamma(\beta) \}, 
	\end{align*}
	and the projections $\Re(E\otimes S) \to E \otimes S^\pm$ are isomorphisms of $\R$--vector spaces.
	\item  
	A $\C$--linear map $\Upsilon \colon E \otimes S \to E \otimes S$ which commutes with $\tau$ is uniquely determined by the induced $\R$--linear map
	\begin{equation*}
		\Upsilon_\R \colon E\otimes S^+ \cong \Re(E\otimes S) \to \Re(E\otimes S) \cong E \otimes S^-.
	\end{equation*}
	Moreover, $\Upsilon$ is self-adjoint with respect to the Hermitian inner product on $E\otimes S$ if and only if $\overline\gamma\Upsilon_\R$ is self-adjoint with respect to the Euclidean inner product on $E \otimes S^+$.  	
	\end{enumerate}	
\end{prop}

\begin{proof}
	Ageneral map $P \colon E \otimes S \to E\otimes S$ can be written as
	\begin{equation*}
		\Upsilon = 
		\begin{pmatrix}
			\Upsilon_+^+ & \Upsilon_-^+ \\
			\Upsilon_+^- & \Upsilon_-^- 
		\end{pmatrix}
	\end{equation*}
	with respect to decomposition $S = S^+\oplus S^-$. 
	Such a map preserves $\tau$ if and only if
	\begin{equation*}
		\Upsilon_+^+  \overline\gamma = \overline\gamma   \Upsilon_-^- \qandq \Upsilon_-^+   \overline\gamma = - \overline\gamma \Upsilon_+^-. 
	\end{equation*}
	The induced real operator is
	\begin{gather*}
		\Upsilon_\R \colon E \otimes S^+ \to E \otimes S^-, \\
		\Upsilon_\R = \Upsilon_+^- + \Upsilon_-^-\overline\gamma.
	\end{gather*}
	If $\Upsilon$ is $\C$--linear then it can be reconstructed from $\Upsilon_\R$ in the following way.
	Decompose $\Upsilon_\R \colon E \otimes S^+ \to E \otimes S^-$ into $\C$--linear and $\C$--antilinear parts $\Upsilon_\R = \Upsilon_c + \Upsilon_a$.
	Then
	\begin{equation*}
		\Upsilon = 
		\begin{pmatrix}
			-\overline\gamma \Upsilon_a & \overline\gamma \Upsilon_c \overline\gamma  \\
			\Upsilon_c & - \Upsilon_a\overline\gamma 
		\end{pmatrix}. \qedhere
	\end{equation*} 
\end{proof}


\section{Compactness for generalized Seiberg--Witten equations}
\label{Sec_CompactnessQuaternionicVortexEquation}

The quaternionic vortex equation is a dimensional reduction of the three-dimensional generalized Seiberg--Witten equation over $M=C\times S^1$ \cite[Section 3]{Doan2017}. 
We will prove \autoref{Thm_CompactnessQuaternionicVortexEquations} by extending the compactness theorem for generalized Seiberg--Witten equations \cite{Walpuski2019}. 
Unfortunately, this extension requires discussing some technical details of  \cite{Walpuski2019}; the method of the proof still works but some of the estimates and statements have to be adjusted.
Rather than rewriting the full proof, we will assume that the reader is familiar with  \cite{Walpuski2019} and comment on the necessary modifications. 
In particular, we will use the notation introduced in \cite[Section 1.1]{Walpuski2019} and  cite propositions, theorems, equations, or definitions with a number x.xx from \cite{Walpuski2019} simply as [WZ x.xx].

Throughout this section, let $\rho \colon G \to \Sp(V)$ be a quaternionic representation as in \autoref{Sec_QuaternionicVortexEquation}. 
Moreover, we assume \autoref{Hyp_SpinGRepresentation}, \autoref{Hyp_FlavorSymmetry}, and \autoref{Hyp_Compactness}.
Let $M$ be a closed oriented $3$-manifold.
Given a spin$^H$ structure on $M$, we associate with it a principal $G$-bundle $P \to M$ and a vector bundle $\bV \to M$ with fiber $V$, similarly to the $2$-dimensional construction in \autoref{Sec_QuaternionicVortexEquation}.
A connection $A$ on $P$ defines the induced Dirac operator
\begin{equation*}
	D_A \colon \Gamma(\bV) \to \Gamma(\bV)
\end{equation*}
and we have a quaternionic moment map
\begin{equation*}
	\mu \colon \bV \to \Lambda^2 T^*M \otimes \ad P.
\end{equation*}
For details, see \cite[Section 1.1]{Walpuski2019}.

\begin{definition}
	The \defined{generalized Seiberg--Witten equation} for a triple
	\begin{equation*}
		A \in \sA(P), \quad \Psi \in \Gamma(\bV), \quad \zeta \in \Gamma(\ad P)
	\end{equation*}
	reads
	\begin{equation*}
		\begin{split}
			&(D_A+\rho(\zeta))\Psi = 0, \\
	   		&F_A + \ast \rd_A \zeta = \mu(\Psi) .
		\end{split}
	\end{equation*}
\end{definition}

The following perturbation of the equation is relevant to the discussion of quaternionic vortex equations.

\begin{definition}
	Let $\sG$ be the gauge group of $P$. 
	Suppose that $\Upsilon \co \sA(P) \times \Gamma(\bV) \to \End(\bV)$ is $\sG$-invariant,  $\eta \co \sA(P) \times \Gamma(\bV) \to \Omega^2(\ad P)$ is $\sG$-equivariant, and 
	\begin{enumerate}
	\item $\Upsilon(A,\Psi)$ is self-adjoint and commutes with the action of the gauge group,
	\item $\rd_A \eta(A,\Psi) = 0$,
	\item $| \Upsilon(A,\Psi) |, | \nabla_A \Upsilon(A,\Psi) |, | \eta(A,\Psi) | , |\nabla_A\eta(A,\Psi)|$ are uniformly bounded. 
	\end{enumerate}
	(Note that $\nabla_A \Upsilon = \nabla_{A_0}\Upsilon$ for any fixed connection $A_0$ because $\Upsilon$ commutes with $\ad P$.)
	
	The \defined{$(\Upsilon,\eta)$--perturbed generalized Seiberg--Witten equation} for $(A,\Psi,\zeta)$ reads
	\begin{equation}
	\label{Eq_GeneralizedSW}
	\begin{split}
		&(D_A+\Upsilon(A,\Psi)+\rho(\zeta))\Psi = 0, \\
	   &F_A + \ast \rd_A \zeta + \eta(A,\Psi) = \mu(\Psi) .
	\end{split}
	\end{equation}
\end{definition}

It is necessary to include the field $\zeta$ to make the equation elliptic modulo the action of the gauge group. 
However, the next proposition shows that without loss of generality we may assume that $\zeta=0$.

\begin{prop}
  \label{Prop_VanishingExtraField}
If $(A,\Psi,\zeta)$ satisfies \autoref{Eq_GeneralizedSW}, then
\begin{equation*}
	\rho(\zeta)\Psi = 0 \qandq \rd_A \zeta = 0.
\end{equation*}
In particular, if $(A,\Psi)$ is irreducible, then $\zeta = 0$. 
\end{prop}

\begin{proof}
Apply $\rd_A$ to the second equation.
By the Bianchi identity and \cite[Proposition B.4]{Doan2017a},
\begin{equation*}
	\rd_A^* \rd_A \zeta - \rho^*( (D_A\Psi)\Psi^* ) = 0.
\end{equation*}
Therefore, by the first equation,
\begin{equation*}
	\rd_A^* \rd_A \zeta + \rho^*( (\rho(\zeta)\Psi)\Psi^* + (\Upsilon\Psi)\Psi^*) = 0.
\end{equation*}
Taking the inner product with $\zeta$ and integrating yields
\begin{equation*}
	\| \rd_A \zeta \|_{L^2}^2 + \| \rho(\zeta)\Psi \|_{L^2}^2 + \langle \Upsilon\Psi, \rho(\zeta)\Psi \rangle_{L^2} = 0.
\end{equation*}
The third term is zero because $\rho(\zeta)\Upsilon$ is a skew-adjoint operator:
\begin{equation*}
	(\rho(\zeta)\Upsilon)^* = \Upsilon^*\rho(\zeta)^* = - \Upsilon\rho(\zeta) = - \rho(\zeta)\Upsilon.  \qedhere
\end{equation*} 
\end{proof}

Set $\zeta =0$ in \autoref{Eq_GeneralizedSW}. 
If $\Psi \neq 0$, then after rescaling the equation by $\epsilon = \| \Psi \|_{L^2}^{-1}$ we can write it as an equation for the triple $(A,\Psi,\epsilon)$:
\begin{equation}
\label{Eq_GeneralizedSWSimplified}
\begin{split}
	&(D_A+\Upsilon(A,\Psi))\Psi = 0, \\
   &\epsilon^2(F_A + \eta(A,\Psi)) = \mu(\Psi) , \\
   &\| \Psi \|_{L^2} = 1.
\end{split}
\end{equation}
The following generalizes the compactness theorem \cite[Theorem 1.28]{Walpuski2019} to the perturbation \autoref{Eq_GeneralizedSWSimplified} of the generalized Seiberg--Witten equation. 

\begin{prop}
If $(A_n,\Psi_n,\epsilon_n)$ is a sequence of solutions of \autoref{Eq_GeneralizedSWSimplified}, then, after passing to a subsequence and applying gauge transformations, $(A_n,\Psi_n,\epsilon_n)$ converges in the $C^\infty$ topology to a solution $(A,\Psi,\epsilon)$  of \autoref{Eq_GeneralizedSWSimplified}. 
\end{prop}

\begin{theorem}
\label{Thm_SWCompactness}
If $(A_n,\Psi_n,\epsilon_n)$ is a sequence of solutions to \autoref{Eq_GeneralizedSWSimplified} such that Hypothesis 1.22 from \cite{Walpuski2019} holds and $\epsilon_n \to 0$, then there exist a closed, nowhere dense subset $Z \subset M$, a connection $A$ over $M\setminus Z$, and section a $\Psi \in \Gamma(M\setminus Z, \bV)$ with the following properties.
\begin{enumerate}
	\item $A$ and $\Psi$ satisfy
	\begin{gather*}
		(D_A + \Upsilon(A,\Psi))\Psi = 0, \\
		\mu(\Psi) = 0, \\
		\| \Psi \|_{L^2} = 1. 
	\end{gather*}
	\item The function $|\Psi|$ extends to a Hölder continuous on $M$ such that $Z = |\Psi|^{-1}(0)$. 
	\item After passing to a subsequence and applying gauge transformations over $M\setminus Z$,
	\begin{enumerate}
		\item $|\Psi_n|$ converges to $|\Psi|$ in the $C^{0,\alpha}$ topology over $M$ for some $\alpha \in (0,1)$,
		\item $A_n$ converges to $A$ in the weak $W^{1,2}$ topology over every compact subset of $M\setminus Z$,
		\item $\Psi_n$ converges to $\Psi$ in the weak $W^{2,2}$ topology over every compact subset of $M\setminus Z$.
	\end{enumerate} 
\end{enumerate}
\end{theorem}

\autoref{Thm_SWCompactness} implies a compactness theorem for quaternionic vortex equations. 

\begin{proof}[Proof of \autoref{Thm_CompactnessQuaternionicVortexEquations} assuming \autoref{Thm_SWCompactness}]

Let $M = S^1 \times C$.
Without loss of generality assume that solutions $(A_n,\Psi_n)$ on $C$ are irreducible solutions such that $\lim_{n \to \infty} \| \Psi_n \|_{L^2} = \infty$.
Set $\hat \Psi_n = \Psi_n / \| \Psi_n \|_{L^2}$.
Consider the pull-back of $(A_n, \hat\Psi_n)$ to $S^1 \times C$ (which, by abuse of notation, we denote by the same symbols) as a sequence of solutions to
\autoref{Eq_GeneralizedSWSimplified} with $\epsilon_n = 1 / \| \Psi_n \|_{L^2}$
These solutions satisfy the $S^1$-invariance property
\begin{equation}
	\label{Eq_CircleInvariance}
	\nabla_t (A - A_0) = 0 \qandq \nabla_t \hat\Psi = 0,
\end{equation}
where $t$ is the coordinate on $S^1$ and $A_0$ is a fixed connection pulled back from $C$.  

By \autoref{Thm_SWCompactness}, there exist a closed, nowhere dense subset $\sZ \subset M$, a connection $A$ over $M\setminus \sZ$, and a section $\Psi \in \Gamma(M\setminus \sZ, \bV)$ with the following properties:
\begin{enumerate}
		\item $A$ and $\Psi$ satisfy
		\begin{gather*}
			\bD_{A,\Upsilon} \Psi = 0, \\
			\mu(\Psi) = 0, \\
			\| \Psi \|_{L^2} = 1,
		\end{gather*}
		where
	\begin{equation*}
		\bD_{A,\Upsilon} = i\nabla_t + D_A + \Upsilon(A,\Psi)
	\end{equation*}
	is the $3$-dimensional Dirac operator acting on sections of $\bV$ over $M$. 
		\item The function $|\Psi|$ extends to a Hölder continuous on $M$ such that $\sZ = |\Psi|^{-1}(0)$. 
		\item After passing to a subsequence and applying gauge transformations over $M\setminus \sZ$,
		\begin{enumerate}
			\item $|\Psi_n|$ converges to $|\Psi|$ in the $C^{0,\alpha}$ topology over $M$ for some $\alpha \in (0,1)$,
			\item $A_n$ converges to $A$ in the weak $W^{1,2}$ topology over every compact subset of $M\setminus \sZ$,
			\item $\Psi_n$ converges to $\Psi$ in the weak $W^{2,2}$ topology over every compact subset of $M\setminus \sZ$.
		\end{enumerate} 
	\end{enumerate}
	The weak convergence implies that for every $\hat a$ and $\hat\psi$,
	\begin{gather*}
		\langle \nabla_t (A-A_0), \hat a \rangle_{L^2} = \lim_{n\to \infty} \langle \nabla_t (A_n-A_0), \hat a \rangle_{L^2} = 0, \\
		\langle \nabla_t \Psi, \hat\psi \rangle_{L^2} = \lim_{n\to\infty} \langle \nabla_t \Psi_n, \hat\psi \rangle_{L^2} = 0,
	\end{gather*}
	so that $(A,\Psi)$ satisfy \autoref{Eq_CircleInvariance} and that $\sZ = Z \times S^1$ for a closed subset $Z \subset C$. 
	After applying a gauge transformation which puts $A$ in a temporal gauge, $(A,\Psi)$ are pulled back from a configuration on $C$. 
\end{proof}

In the remaining part of this section we discuss how to adapt the proof of \cite[Theorem 1.28]{Walpuski2019} to the perturbed equation. 
The main tool in the proof of \cite[Theorem 1.28]{Walpuski2019} is the Lichnerowicz--Weitzenböck formula.
The perturbation terms $\Upsilon = \Upsilon(A,\Psi)$ and $\eta = \eta(A,\Psi)$ contribute to new terms in this formula. 

\begin{prop}
\label{Prop_Weitzenbock}
For every $(A,\Psi)$, 
\begin{equation*}
	(D_A+\Upsilon)^2\Psi = \nabla_A^*\nabla_A\Psi + \gamma(F_A)\Psi + \fR\Psi + D_A(\Upsilon\Psi) + \Upsilon(D_A+\Upsilon)\Psi,
\end{equation*}
where $\fR$ is the zeroth order operator defined in \cite[Definition2.1]{Walpuski2019} and depending on the Riemann curvature of $M$. 
In particular, if $(A,\Psi)$ is a solution to \autoref{Eq_GeneralizedSWSimplified}, then
\begin{equation*}
\nabla_A^*\nabla_A\Psi - \gamma(\eta)\Psi + \epsilon^{-2}\gamma(\mu(\Psi))\Psi + \fR\Psi + D_A(\Upsilon\Psi) = 0,
\end{equation*}
\end{prop}

\begin{proof}
We have
\begin{equation*}
	0=(D_A+\Upsilon)^2\Psi = D_A^2\Psi + D_A(\Upsilon\Psi) + \Upsilon(D_A+\Upsilon)\Psi  
\end{equation*}
and the formula follows from the standard Weitzenböck formula for the Dirac operator and the second equation in \autoref{Eq_GeneralizedSWSimplified}, see \cite[Proposition 2.2]{Walpuski2019}. 
\end{proof}

The new term involving $\eta$ is harmless.
The term involving $\Upsilon$ is more problematic at it involes the first derivative of $\Psi$. 
It can be estimated pointwise using Young's inequality
\begin{equation}
\label{Eq_EstimateNaive}
|D_A(\Upsilon\Psi)| \lesssim \delta|\nabla_A\Psi|^2 + \delta^{-1}|\Psi|^2
\end{equation}
for any $\delta > 0$. 
Alternatively, using integration by parts, we can trade the integral of the derivative term for a boundary integral.

\begin{prop}
\label{Eq_EstimateByParts}
If $(A,\Psi)$ satisfies 
\begin{equation*}
(D_A+\Upsilon)\Psi = 0,
\end{equation*}
then for every open set $U \subset M$ with smooth boundary
\begin{equation*}
	\int_U \langle D_A(\Upsilon\Psi), \Psi \rangle \lesssim \int_U |\Psi|^2 + \int_{\partial U} |\Psi|^2.
\end{equation*}
\end{prop}

\begin{proof}
Let $e_1,e_2,e_3$ be a local orthonormal frame of $TM$ and set $\nabla_{A,i} = \nabla_{A,e_i}$. 
\begin{align*}
	\int_U \langle D_A (\Upsilon\Psi), \Psi \rangle &= - \sum_i \int_U \langle \nabla_{A,i} (\Upsilon\Psi), \gamma(e_i) \Psi \rangle \\
	&= \sum_i \int_U \langle \Upsilon\Psi, \nabla_{A,i}(\gamma(e_i)\Psi)\rangle - \sum_i \int_{\partial U} \langle \Upsilon\Psi,  \gamma(e_i)\Psi \rangle (e_i \lrcorner \vol) \\
	&= \int_U \langle \Upsilon\Psi, D_A \Psi + R\Psi \rangle - \sum_i \int_{\partial U} \langle \Upsilon\Psi,  \gamma(e_i)\Psi \rangle (e_i \lrcorner \vol) \\
	&\lesssim \int_U  |\Psi|^2 + \int_{\partial U} |\Psi|^2,
\end{align*}
where $R = \gamma(\sum_i \nabla_i e_i)$ is a Riemannian curvature term and in the last line we use $D_A \Psi = - \Upsilon \Psi$ and $|\Upsilon| \leq C$. 
\end{proof}

In what follows, let $(A,\Psi, \epsilon)$ be a solution of \autoref{Eq_GeneralizedSWSimplified}.
In the case $\Upsilon=0$ and $\eta=0$, section 2 of \cite{Walpuski2019} establishes a priori bounds on $A$ and $\Psi$ in terms of the frequency function defined below.

\begin{definition}
Let $r_0$ be the injectivity radius of $M$.
Given $x \in M$ and $r \in (0,r_0]$ let $B_r(x) \subset M$ be the geodesic ball of radius $r$ centered at $x$.
Define
\begin{gather*}
	m_x(r) = \frac{1}{4\pi r^2} \int_{\partial B_r(x)} |\Psi|^2, \\
	D_x(r) = \frac{1}{4\pi r} \int_{B_r(x)} |\nabla_A\Psi|^2 + 2\epsilon^{-2}|\mu(\Psi)|^2.
\end{gather*}
The \defined{frequency function} associated with $(A,\Psi,\epsilon)$ is
\begin{equation*}
	N_x(r) = \frac{D_x(r)}{m_x(r)}.
\end{equation*}
\end{definition}

\begin{prop}
\label{Prop_WZEstimates}
Let $(A,\Psi,\epsilon)$ be a solution of \autoref{Eq_GeneralizedSWSimplified}.
The a priori bounds on $A$ and $\Psi$ in terms of $m$ and $N$ from \cite[Propositions 2.9--2.23]{Walpuski2019}, proved in the case $\Upsilon=0$ and $\eta=0$, hold also in the general case.
\end{prop}

\begin{proof}
Most of the proofs in \cite[Section 2]{Walpuski2019} rely on the Lichnerowicz--Weitzenböck formula, and therefore are affected by the new terms appearing in \autoref{Prop_Weitzenbock}. 
The new terms are estimated using the bounds on $\eta$ and $\Upsilon$, \autoref{Eq_EstimateNaive}, \autoref{Eq_EstimateByParts}, and rearranging some of the inequalities.
Below we comment on the necessary modifications in the proof of every estimate in \cite[Section 2]{Walpuski2019}.

\defined{[WZ 2.9]} 
The Lichnerowicz--Weitzenbock formula from \autoref{Prop_Weitzenbock} produces extra terms corresponding to $\Upsilon$ and $\eta$, which can be estimated using \autoref{Eq_EstimateNaive},  \autoref{Eq_EstimateByParts}, and the $L^\infty$ bound on $\eta$. 
This leads to
\begin{equation*}
\frac{\rd}{\rd r} \int_{\partial B_r(x)}|\Psi|^2 \geq \int_{\partial B_r(x)} \(\frac{2}{r}-c_1r - c_2\)|\Psi|^2 + \int_{B_r(x)} (2-c_3r^2)\int_{B_r(x)}|\nabla_A \Psi|^2,
\end{equation*}
for constants $c_1,c_2,c_3$.
The right-hand side is positive for small $r$, which completes the proof.

\defined{[WZ 2.10]}
In this proof, we multiply the Lichnerowicz--Weitzenböck formula by the function $\chi^2 G_y$ where $\chi$ is a cut-off function and $G_y = G(y,x)$ with $G$ being the Green kernel for the Laplacian on $B_r(x)$ and $y \in B_{r/2}(x)$. 
\autoref{Prop_Weitzenbock} again produces two extra terms:
\begin{equation*}
\int_{B_r(x)} \chi^2 G_y \langle\gamma(\eta)\Psi,\Psi\rangle \leq cr^2 \| \chi\Psi\|_{L^\infty}^2,
\end{equation*}
and, using \eqref{Eq_EstimateNaive}, 
\begin{align*}
\int_{B_r(x)} \chi^2 G_y \langle D_A(\Upsilon\Psi),\Psi \rangle &\lesssim \delta\int_{B_r(x)} \chi^2 G_y |\nabla_A\Psi|^2 + \delta^{-1} \int_{B_r(x)} \chi^2 G_y  |\Psi|^2 \\
&\leq \delta\int_{B_r(x)} \chi^2 G_y |\nabla_A\Psi|^2 + \delta^{-1}r^2 \| \chi\Psi\|_{L^\infty}^2. 
\end{align*}
For $\delta$ sufficiently small, the first term on the right-hand side can be rearranged and absorbed by the integral
\begin{equation*}
\int_{B_r(x)} \chi^2 G_y |\nabla_A\Psi|^2 
\end{equation*}
which appears on the left-hand side of the main inequality in the proof of Proposition 2.10.

\defined{[WZ 2.11]} 
The original proof uses [WZ 2.12], [WZ 2.14], and \cite[Proposition B.4]{Doan2017}. 
While [WZ 2.14] is independent of the equation, the other two formulae obtain additional terms from the Lichnerowicz--Weitzenbock formula and the $\eta$ term in $F_A = \mu(\Psi) - \eta$. 
To prove the estimate, we arrange the term in a different way than in the proof of [WZ 2.12]. 
We use the following convention: $R$ is any remainder term satisfying
\begin{equation*}
| R | \lesssim |\Psi| + |\nabla_A\Psi|
\end{equation*}
and $A \ast B$ denotes any bilinear algebraic operation involving tensors $A$ and $B$.
Both $R$ and $\ast$ can denote different terms and operations in different lines.
By [WZ 2.14],
\begin{equation*}
[\nabla_A^*\nabla_A,\nabla_A]\Psi = \rho(\rd_A^*F_A)\Psi + F_A \ast \nabla_A\Psi + R.
\end{equation*}
Therefore, by \autoref{Prop_Weitzenbock},
\begin{equation}
\label{Eq_DerivativeEstimate}
\begin{split}
\nabla_A^*\nabla_A\nabla_A \Psi &= [\nabla_A^*\nabla_A,\nabla_A]\Psi - \gamma(F_A)\nabla_A\Psi - \gamma(\nabla_A F_A)\Psi - \nabla_A( D_A(\Upsilon\Psi)) + R \\
&= \rho(\rd_A^*F_A)\Psi + F_A \ast \nabla_A\Psi - \gamma(F_A)\nabla_A\Psi - \gamma(\nabla_A F_A)\Psi - \nabla_A( D_A(\Upsilon\Psi)) + R.
\end{split}
\end{equation}
Next, we take the inner product \eqref{Eq_DerivativeEstimate} with $\nabla_A\Psi$ and estimate all terms. 
By \cite[Proposition B.4]{Doan2017},
\begin{align*}
\langle \rho(\rd_A^*F_A)\Psi, \nabla_A\Psi \rangle &= \langle \rd_A^*F_A, \rho^*(\nabla_A\Psi\Psi^*) \rangle  \\
&= -\langle \rd_A^*F_A, \rd_A^*\mu(\Psi) \rangle  + \langle \rd_A^*F_A, D_A\Psi\ast\Psi \rangle \\
&= -\epsilon^2|\rd_A^*F_A|^2 - \epsilon^2 \langle \rd_A^*F_A, \eta \rangle - \langle \rd_A^*F_A, \Upsilon\Psi\ast\Psi \rangle \\
&\lesssim -\epsilon^2|\rd_A^*F_A|^2 + |F_A||\Psi|^2 + |F_A|.
\end{align*} 
We similarly estimate
\begin{align*}
-\langle \gamma(\nabla_A F_A)\Psi,\nabla_A\Psi \rangle &= - \langle \nabla_A F_A, \nabla_A \mu(\Psi) \rangle \\
&= - \epsilon^2 |\nabla_A F_A|^2 - \epsilon^2 \langle \nabla_A F_A , \nabla_A \eta \rangle \\
&\lesssim - \epsilon^2 |\nabla_A F_A|^2 + 1.
\end{align*}
The term involving $\Upsilon$ can be estimated by
\begin{align*}
\langle \nabla_A(D_A(\Upsilon\Psi)), \nabla_A \Psi \rangle 
 &\lesssim |\nabla_A^2 \Psi||\nabla_A \Psi| + |\nabla_A\Psi|^2 + |\Psi||\nabla_A\Psi|  \\
&\lesssim \delta|\nabla_A^2\Psi|^2 + \delta^{-1}|\nabla_A\Psi|^2 + |\Psi||\nabla_A\Psi|.
\end{align*}
We choose $\delta > 0$ small.  
Let $x \in M$, $r > 0$, and let $\chi$ be a cut-off function supported in $B_{r/2}(x)$ as in the proof of [WZ 2.11]. 
Taking the inner product of \eqref{Eq_DerivativeEstimate} with $\nabla_A\Psi$, multiplying by $r\chi^2$ and integrating over $B_r(x)$, we obtain
\begin{equation*}
r \int_{B_r(x)} \chi^2( |\nabla_A^2\Psi|^2 + \epsilon^2 |\nabla_A F_A|^2 + \epsilon^2 | \rd_A^*F_A|^2 \lesssim r \int_{B_r(x)} \chi^2(1 + |\Psi|^2 + |\nabla\Psi|^2 + |F_A| |\nabla_A \Psi|^2 + \delta |\nabla^2_A\Psi|^2). 
\end{equation*}
Only the last two terms on the right-hand side are problematic. 
The one involving $F_A$ is estimated in the same way as in the proof of [WZ 2.11]. 
The second term can be moved to the left-hand side provided $\delta$ is sufficiently small. 
This leads to the desired inequality from [WZ 2.11].
 
\defined{[WZ 2.17], [WZ 2.18]} follow once [WZ 2.10], [WZ 2.11] are proved.

\defined{[WZ 2.19]} follows from the previous estimates combined with [WZ 2.20].
The proof of [WZ 2.20] uses [WZ 2.21] and [WZ 2.23]. 
[WZ 2.23] does not use the equation so it still holds.
However, the proof of [WZ 2.21] relies on the Lichnerowicz--Weitzenbock formula which now acquires additional terms. 
The first term involving $\Upsilon$ is
\begin{equation*}
\mu( D_A(\Upsilon\Psi),\Psi) 
\end{equation*}
whose contribution to the proof of [WZ 2.20] can be estimated by
\begin{equation}
\label{Eq_220Estimate}
r \int_{B_r(x)} \chi^2 (|F_A|+|\eta|)|\Psi|(|\Psi|+|\nabla_A\Psi|).
\end{equation} 
where $x$, $r$, and $\chi$ are as before. 
One of the assumptions of [WZ 2.20] is that $m_x(r)=1$ which implies that $|\Psi|$ is bounded. 
The term $|F_A||\nabla_A \Psi|^2$ is estimated as before.
The other terms in \autoref{Eq_220Estimate} are harmless and do not affect the proof of [WZ 2.20]. 
Similarly, the new term in [WZ 2.21] involving $\eta$ has the form
\begin{equation*}
2\epsilon^{-2} \mu(\gamma(\eta)\Psi, \Psi)
\end{equation*}
and can be estimated by $|\Psi|$. 

Finally, in the proof of [WZ 2.20] we use the formula for $\nabla_A^*\nabla_A \mu(\Psi)$ from [WZ 2.19] and use the original equation $\mu(\Psi) = F_A$
In our setup, $\mu(\Psi) = F_A + \eta$, so to obtain estimate on the integral of $\Gamma_\Psi F_A$ we use
\begin{equation*}
|\Gamma_\Psi F_A|^2 \lesssim   |\Gamma_\Psi(F_A+\eta)|^2  + |\Gamma_\Psi\eta|^2.
\end{equation*}	
The estimate on the first term follows from [WZ 2.19] as in the original proof while the second term is uniformly bounded. 

\end{proof}

The key step in the proof of the compactness theorem in \cite{Walpuski2019} is the monotonicity formulae [3.5, 3.12, 3.13] for the functions $m$ and $N$. 
In the perturbed setup, the monotonicity formulae have to be modified.

\begin{prop}
There exist constants $r_0, c > 0$ with the following property.
Let $(A,\Psi,\epsilon)$ be a solution of \autoref{Eq_GeneralizedSWSimplified}.  
For every $x \in M$ and $0 \leq s \leq r \leq r_0$, the functions $m_x$ and $N_x$ satisfy
\begin{align}
\label{Eq_MMonotonicity}
m_x(s) &\leq 2m_x(r), \\
\label{Eq_NMonotonicity}
N_x(s) &\leq 2N_x(r) + cr,
   \end{align}
and
\begin{equation}
\label{Eq_MMonotonicity2}
\frac12 \( \frac{r}{s} \)^{N_x(s) - cr} m_x(s) \leq m_x(r) \leq 2\( \frac{r}{s} \)^{4N_x(r) + cr} m_x(s).
\end{equation}
\end{prop}

\begin{proof}
The proof of the monotonicity formula [WZ 3.5] relies on formulae [WZ 3.6, 3.10, 3.11].
Similar formulae hold in the general case, except the estimates on the remainder terms $\fr_{D'}, \fr_D, \fr_{m'}$ appearing in these formulae are weaker. 

\defined{[WZ 3.6]} The remainder term is now estimated by
\begin{equation}
\label{Eq_DRemainder}
|\fr_{D'}|\lesssim D_x(r) + m_x(r).
\end{equation}

Indeed, we use the same argument as in the proof of [WZ 3.6] using the tensor $T$ defined therein.
In our setup, the estimate on $|\nabla^*T|$, based on the Lichnerowicz--Weitzenböck formula, has new contributions from $\eta$ and $\Upsilon$, which are estimated as follows.
The terms involving $\eta$ do not pose any problem as the contribution to $\nabla^*T(e_i)$ is bounded by
\begin{equation*}
	\langle \gamma(\eta)\Psi + \sum_ \rho(\eta(e_i,e_j))\Psi , \nabla_{A,e_j}\Psi \rangle \lesssim |\Psi||\nabla_A \Psi|,
\end{equation*}
which is the same as the original bound on $|\nabla^*T|$ in [WZ 3.8]. 
The contribution of $\Upsilon$ to $\nabla^*T$ is
\begin{equation*}
	\langle D_A(\Upsilon\Psi), \nabla_A \Psi \rangle  \lesssim |\Psi||\nabla_A\Psi| + |\nabla_A\Psi|^2.
\end{equation*}
Therefore, the total contribution to the integral
\begin{equation*}
	\frac{1}{4\pi r^2} \int_{B_r(x)} 2r_x \nabla^*T(\partial_r)
\end{equation*}
is estimated, up to a constant, by
\begin{equation*}
	\frac{1}{r} \int_{B_r(x)} |\nabla\Psi|^2 + |\Psi|^2 \lesssim D_x(r) + r^2 m_x(r),
\end{equation*}
where to estimate the integral of $|\Psi|^2$ we use [WZ 2.9].

\defined{[WZ 3.10]} The same formula holds but now the remainder term is estimated by
\begin{equation*}
|\fr_D| \lesssim r m_r(x).
\end{equation*}

Indeed, we follow the same proof as in [WZ 3.10] and estimate the new terms involving $\eta$ and $\Upsilon$. 
By \autoref{Eq_EstimateByParts}, their contribution to the formula is bounded, up to a constant, by
\begin{equation*}
	\frac{1}{r} \int_{B_r(x)} |\Psi|^2 + \frac{1}{r}\int_{\partial B_r(x)} |\Psi|^2 \lesssim r^2 m_x(r) + r m_x(r),
\end{equation*}
where we use [WZ 2.9] and the definition of $m_x(r)$. 

\defined{[WZ 3.11]} The same formula holds but now the remainder term is estimated by
\begin{equation*}
|\fr_{m'}| \lesssim m_x(r)
\end{equation*}
This is because the proof uses [WZ 3.10] which now has an estimate with a lower power of $r$.
In particular, we have
\begin{equation*}
m_x'(r) \geq - cm_x(r).
\end{equation*}
Therefore, $e^{cr}m_x$ is a non-decreasing function of $r$, which implies \autoref{Eq_MMonotonicity}.

To prove \autoref{Eq_NMonotonicity}, we proceed as in the proof of [WZ 3.3] and use the modified estimates on the remainder terms:
\begin{align*}
N_x'(r) &\geq \frac{\fr_{D'}}{m_x(r)} - \frac{4\fr_D N_x(r)}{rm_x(r)} - \frac{\fr_{m'}N_x(r)}{m_x(r)} \\
&\geq -c(1+N_x(r)).  
\end{align*}
Therefore,
\begin{equation*}
(e^{cr}N_x)'(r) = e^{cr}(cN_x(r) + N_x'(r)) \geq - ce^{cr} \geq -c,
\end{equation*}
which implies that the function
\begin{equation*}
r \mapsto e^{cr}N_x(r) + cr
\end{equation*}
is non-decreasing on $[0,r_0]$.
If $0 < s \leq r \leq r_0$ and $r_0$ is small, then
 \begin{equation*}
N_x(s) \leq e^{c(r-s)} N_x(r) + c(r-s) \leq 2N_x(r) + cr .
\end{equation*}

Finally, to prove \autoref{Eq_MMonotonicity2}, we use [WZ 3.11] with a weaker estimate on the remainder established above. 
For every $t \in [s,r]$, 
\begin{equation*}
(\log m_x)'(t) = \frac{m_x'(t)}{m_x(t)} = \frac{2N_x(t)}{t} + \frac{\fr_{m'}(t)}{m_x(t)} \leq \frac{4N_x(r)+cr}{t} + c.
\end{equation*}
Integrating over $t \in [s,r]$ and exponentiating yields
\begin{equation*}
  \frac{m_x(r)}{m_x(s)}   \leq e^{c(r-s)} \( \frac{r}{s} \)^{4N_x(r) + cr} \leq 2\( \frac{r}{s} \)^{4N_x(r) + cr},
\end{equation*}
which implies the upper bound on $m_x(r)$.
The lower bound is proved in the same way by integrating
\begin{equation*}
(\log m_x)'(t) = \frac{m_x'(t)}{m_x(t)} = \frac{2N_x(t)}{t} + \frac{\fr_{m'}(t)}{m_x(t)} \geq \frac{N_x(s)-cr}{t} - c.
\end{equation*}

\end{proof}

The monotonicity formula is used in \cite{Walpuski2019} to prove a lower bound in terms of $|\Psi(x)|$ on the regularity scale
\begin{equation*}
r_A(x) = \sup\left\{ r\in[0,r_0] : r\int_{B_r(x)} |F_A|^2 \leq c_F \right\}, 
\end{equation*}
where $c_F > 0$ is a fixed positive constant chosen so that the a priori estimates from \cite[section 3]{Walpuski2019} hold. 
This lower bound is the content of [WZ 3.2].
This is the key step in the proof of the compactness theorem.
Together with [WZ 2.11], it implies a uniform upped bound on the Hölder norm of $|\Psi|$ and the convergence of $A$, $\Psi$ on $M\setminus Z$, where $Z$ is the zero set of the limit of $|\Psi_n|$ in the Hölder norm.

Since [WZ 2.11] holds in our setting by \autoref{Prop_WZEstimates}, the same argument can be used provided that that [WZ 3.2] holds as well.
In addition to [WZ 2.9], which is still true by \autoref{Prop_WZEstimates}, the proof of [WZ 3.2] uses  [WZ 3.15] and [WZ. 3.16]. 
While the statements of these propositions have to be modified, the proofs are the same as those of [WZ 3.15] and [WZ 3.16] except that we have to use \autoref{Eq_MMonotonicity2}. 

\begin{prop}
For all $0 < s \leq r \leq r_0$ and $\delta > 0$, if
\begin{equation*}
s \leq r\left( \frac{|\Psi(x)|^2}{cm_x(r)} \right)^{1/\delta},
\end{equation*}
then
\begin{equation*}
N_x(s) \leq \delta + cr.	
\end{equation*}
\end{prop}

\begin{prop}
For every $r \in (0,r_0/4]$, $s \in (0,2r]$, and $y \in B_r(x)$, if 
\begin{equation*}
	N_x(4r) \leq 1,
\end{equation*}
then
\begin{equation*}
	N_y(s) \leq c(N_x(4r) + r).
\end{equation*}
\end{prop}

With these modifications to [WZ 3.15] and [WZ 3.16], the rest of the proof of [WZ 3.2] follows, and therefore also [WZ 3.20] and the convergence of $|\Psi_n|$, $A_n$, and $\Psi_n$. 
Finally, the proof [WZ 4.1], which implies that $Z$ is nowhere dense, is almost unchanged. 
It uses of the Lichnerowicz--Weitzenböck formula by integrating $\langle \nabla_A^*\nabla_A\Psi, \Psi \rangle$ over a neighborhood $Z_\epsilon$ of $Z$.
The new terms, however, do not change the final inequality because they are estimated by integrals
\begin{equation*}
\int_{Z_{2\epsilon}} |\eta||\Psi|^2 + |D_A(\Upsilon\Psi)||\Psi| \lesssim 	\int_{Z_{2\epsilon}} |\Psi|^2 + |\nabla_A\Psi||\Psi|
\end{equation*}
which already appear on the right-hand side of the inequality in the proof of [WZ 4.1]. 

\printbibliography

\end{document}
